%% file: NBS_mars.tex
\documentclass[11pt]{article}

\usepackage{graphicx,,psfrag}

\usepackage{amssymb,amsmath,amsthm,enumerate}
\usepackage{fullpage}
\newcommand{\url}{}

\RequirePackage[colorlinks,citecolor=blue,urlcolor=blue]{hyperref}

\newtheorem{lemma}{Lemma}
\newtheorem{remark}[lemma]{Remark}
\newtheorem{proposition}[lemma]{Proposition}

\newtheorem{theorem}[lemma]{Theorem}

\newtheorem{example}{Example}

\newtheorem{corollary}[lemma]{Corollary}
\newtheorem*{Proposition*}{Proposition}

\newcommand{\RR}{{\mathbb{R}}}

\newcommand{\VAR}{{\mathrm{Var}}}
\newcommand{\tr}{{\rm tr}}

\newcommand{\dE}{\mathbb {E}}
\newcommand{\dP}{\mathbb{P}}
\newcommand{\dN}{\mathbb {N}}
\newcommand{\dR}{\mathbb {R}}
\newcommand{\dC}{\mathbb {C}}
\newcommand{\cF}{\mathcal {F}}

\newcommand{\cE}{\mathcal {E}}
\newcommand{\cG}{\mathcal {G}}
\newcommand{\cX}{\mathcal {X}}
\newcommand{\cW}{\mathcal {W}}

\newcommand{\cT}{\mathcal {T}}
\newcommand{\cY}{\mathcal {Y}}

\newcommand{\Bin}{ \mathrm{Bin}}
\newcommand{\Poi}{ \mathrm{Poi}}
\newcommand{\DTV}{{\mathrm{d_{TV}}}}

\newcommand{\vol}{V}
\newcommand{\sign}{ \mathrm{sign}}
\newcommand{\diag}{ \mathrm{diag}}

\newcommand{\ABS}[1]{{{\left| #1 \right|}}} 
\newcommand{\BRA}[1]{{{\left\{#1\right\}}}} 
\newcommand{\NRM}[1]{{{\left\| #1\right\|}}} 
\newcommand{\PAR}[1]{{{\left(#1\right)}}} 

\newcommand{\1}{1\!\!{\sf I}}\newcommand{\IND}{\1}
\newcommand{\veps}{\varepsilon}
\newcommand{\si}{\sigma}

\newcommand{\whp}{{w.h.p.~}}

\newcommand{\BEAS}{\begin{eqnarray*}}
\newcommand{\EEAS}{\end{eqnarray*}}
\newcommand{\BEA}{\begin{eqnarray}}
\newcommand{\EEA}{\end{eqnarray}}
\newcommand{\BEQ}{\begin{equation}}
\newcommand{\EEQ}{\end{equation}}
\newcommand{\BIT}{\begin{itemize}}
\newcommand{\EIT}{\end{itemize}}
\newcommand{\BNUM}{\begin{enumerate}}
\newcommand{\ENUM}{\end{enumerate}}
\begin{document}

\title{Non-backtracking spectrum of random graphs: \\
community detection and non-regular Ramanujan graphs}
\author{Charles Bordenave, Marc Lelarge, Laurent Massouli\'e}
\maketitle

\abstract{
A non-backtracking walk on a graph is a directed path such that no edge is the inverse of its preceding edge. The non-backtracking matrix of a graph is indexed by its directed edges and can be used to count non-backtracking walks of a given length. It has been used recently in the context of community detection and has appeared previously in connection with the Ihara zeta function and in some generalizations of Ramanujan graphs. In this work, we study the largest eigenvalues of the non-backtracking matrix of the Erd\H{o}s-R\'enyi random graph and of the Stochastic Block Model in the regime where the number of edges is proportional to the number of vertices. Our results confirm the "spectral redemption conjecture" that community detection can be made on the basis of the leading eigenvectors  above the feasibility threshold. 
}

\input{NBS_intro}

\input{NBS_preliminaries}

\input{NBS_main_results}

\input{NBS_algebra}

\input{NBS_Erdos_Renyi_V2}

\input{NBS_SBM_V2}

\input{NBS_Branching_Processes}

\input{NBS_Local_Structure}

\input{NBS_Norm_NBM}

\input{NBS_vecteurs}

\bibliographystyle{abbrv}
\bibliography{BibCommunityDetection,bib}

\end{document}

%% file: NBS_intro.tex
\section{Introduction}
Given a finite (simple, non-oriented) graph $G=(V,E)$, several matrices of interest can be associated to $G$, besides its adjacency matrix $A=(\IND_{\{u,v\}\in E})_{u,v\in V}$. In this work we are interested in the so-called {\em non-backtracking} matrix of $G$, denoted by $B$. It is indexed by the set $\vec E = \{ (u,v) : \{u,v\} \in E \}$ of {\em oriented} edges in $E$ and defined by
$$
B_{e f } = \IND (  e_2 = f_1 ) \IND ( e_1 \ne f_2 ) =  \IND (  e_2 = f_1 ) \IND ( e \ne f^{-1} ),
$$
where for any $e = (u,v ) \in \vec E$, we set $e_1 = u$,  $e_2 = v$,
$e^{-1} = (v,u)$. This matrix was introduced by Hashimoto
\cite{MR1040609}. 
A non-backtracking walk is a directed path of directed edges of $G$
such that no edge is the inverse of its preceding edge. 
It is easily seen that for any $k\ge 1$, $B^k_{ef}$ counts the number of non-backtracking walks of $k+1$ edges  on $G$ starting with $e$ and ending with $f$.

Our focus is the spectrum of $B$, referred to in the sequel as the non-backtracking spectrum of $G$, when $G$ is a sparse random graph. Specifically we shall characterize the asymptotic behavior of the leading eigenvalues and associated eigenvectors in the non-backtracking spectrum of sparse random graphs in the limit $n\to\infty$ where $n=|V|$. 

The random graphs we consider are drawn according to the Stochastic Block Model, a generalization of Erd\H{o}s-R\'enyi graphs due to Holland et al. \cite{Holland83}. In this model nodes $v\in V$ are partitioned into $r$ groups. An edge between two nodes $u,v$ is drawn with probability $W(\sigma(u),\sigma(v))/n$, where $\sigma(u)\in[r]$ denotes the group node $u$ belongs to. Thus when the  $r\times r$ matrix $W$ is fixed the expected node degrees remain of order 1 as $n\to \infty$. We focus moreover on instances where the fraction of nodes in group $i$ converges to a limit $\pi(i)$ as $n\to\infty$.

An informal statement of our results for eigenvalues is as
follows. Let $G$ be drawn according to a Stochastic Block Model with
fixed number $r$ of node groups such that all nodes have same fixed
expected degree $\alpha>1$. Let $\mu_1,\ldots,\mu_r$ denote
the leading eigenvalues of the expected adjacency matrix
$\bar{A}:=\dE(A)$, ordered so that $\mu_1 = \alpha\ge |\mu_2|\ge \ldots\ge |\mu_r|$. Let $r_0\le r$ be such that $ |\mu_{r_0+1}|\le \sqrt{\alpha}< |\mu_{r_0}|$. Then the $r_0$ leading eigenvalues of $B$ are asymptotic to $\mu_1$,\ldots,$\mu_{r_0}$, the remaining eigenvalues $\lambda$ satisfying $|\lambda|\le (1+o(1))\sqrt{\alpha}$. 

\subsection*{Community detection}
Our primary motivation stems from the problem of {\em community detection}, namely: how to estimate a clustering of graph nodes into groups close to the underlying blocks, based on the observation of such a random graph $G$? Decelle et al.~\cite{Decelle11} conjectured a phase transition phenomenon on detectability, namely: the underlying block structure could be detected if and only if $|\mu_2|>\sqrt{\alpha}$. 

In the case of two communities with roughly equal sizes ($\pi(1)= \pi(2)=1/2$) and symmetric matrix $W$, the negative part (impossibility of detection when $|\mu_2|\le\sqrt{\alpha}$) was proven by Mossel et al~\cite{Mossel12}. As for the positive part (feasibility of detection when $|\mu_2|>\sqrt{\alpha}$), it was conjectured in~\cite{Decelle11} that the so-called belief propagation algorithm would succeed. Krzakala et al.~\cite{spec_red} then made their so-called ``spectral redemption conjecture'' according to which detection based on the second eigenvector of the non-backtracking matrix $B$ would succeed.

Recently a variant of the spectral redemption conjecture was proven by
Massouli\'e~\cite{LM13}: the spectrum of a matrix counting {\em
  self-avoiding paths} in $G$ allows to detect communities through
thresholding of the second eigenvector. More recently and independently of ~\cite{LM13}, an alternative proof of the positive part of the conjecture in~\cite{Decelle11} was given by Mossel et al.~\cite{MNS}, based on an elaborate construction involving countings of non-backtracking paths in $G$. 

The two approaches of ~\cite{LM13} and~\cite{MNS}  to proving the
positive part of the conjecture in~\cite{Decelle11}, while both
relying ultimately on properties of specific path counts, differ
however in the following respects. The method in~\cite{LM13} relies on
a clear spectral separation property but its implementation is
computationally expensive, as the counts of self-avoiding walks it
relies upon take super-linear (though polynomial) time. The method
in~\cite{MNS} is computationally efficient as it runs in quasi-linear
time, but the proof does not establish a spectral separation
property. The other two methods conjectured to  achieve successful
reconstruction, namely belief propagation and analysis of
non-backtracking spectrum, are computationally efficient and they are
motivated by a clear intuition as described in the spectral redemption
conjecture. 



Our present work proves the spectral redemption conjecture. More generally by characterizing all the leading eigen-elements it determines the limits of community detection based on the non-backtracking spectrum in the presence of an arbitrary number of communities.

\subsection*{Ihara zeta function}

Hashimoto introduced the matrix $B$ in the context of the Ihara zeta
function \cite{MR1040609}. We have the  identity
$$
\det( I - z B) = \frac{1}{\zeta_G(z)},
$$
where $\zeta_G$ is the Ihara zeta function of the graph,  refer to \cite{MR961485,MR1749978,horton-stark-terras-2006,MR2768284}. It follows that the poles of the Ihara zeta function are the reciprocal of the eigenvalues of $B$. Our main results have thus consequences on the localization of the poles of the zeta function of random graphs. 

\subsection*{Weak Ramanujan property}
Our result also has an interpretation from the standpoint of Ramanujan
graphs, introduced by Lubotzky et al.~\cite{lubotzky} (see
Murty~\cite{murty} for a recent survey). These are by definition
$k$-regular graphs such that the second largest modulus of its
eigenvalues is at most $2\sqrt{k-1}$. By a result of Alon and Boppana
(see \cite{MR1124768}) for fixed $k$, $k$-regular graphs on $n$ nodes
must have their second largest eigenvalue at least $2\sqrt{k-1}-o(1)$
as $n\to\infty$. Hence Ramanujan graphs are regular graphs with
maximal spectral gap between the first and second eigenvalue moduli. A
celebrated result of Friedman \cite{MR2437174} states that random
$k$-regular graphs achieve this lower bound with high probability as
their number of nodes $n$ goes to infinity.

Lubotzky~\cite{lubotzky95} has proposed an extension of the definition of Ramanujan graphs to the non-regular case. Specifically, $G$ is Ramanujan according to this definition if and only if
$$
\max\{|\lambda|:\lambda\in\hbox{spectrum}(A),\;|\lambda|\ne \rho\}\le \sigma,
$$ 
where $A$ is the adjacency matrix of $G$, $\rho$ its spectral radius, and $\sigma$ the spectral radius of the \emph{ adjacency operator on the universal covering tree of $G$}.

Using the analogy between the Ihara zeta function and the Riemann zeta function, Stark and Terras (see~\cite{horton-stark-terras-2006}) have defined a graph to satisfy the {\em graph Riemann hypothesis} if its non-backtracking matrix $B$ has no eigenvalues $\lambda$ such that $|\lambda|\in (\sqrt{\rho_B},\rho_B)$, where $\rho_B$ is the Perron-Frobenius eigenvalue of $B$. Interestingly, a regular graph $G$ is Ramanujan if and only if it satisfies the graph Riemann hypothesis (see ~\cite{murty} and ~\cite{horton-stark-terras-2006}). Thus the graph Riemann hypothesis can also be seen as a generalization of the notion of Ramanujan graphs to the non-regular case, phrased in terms of non-backtracking spectra rather than on spectra of universal covers as in the definition of Lubotzky~\cite{lubotzky95}. 

Our results imply that for fixed $\alpha>1$, Erd\H{o}s-R\'enyi graphs $\cG(n,\alpha/n)$ have an associated non-backtracking matrix $B$ such that $\rho_B\sim \alpha$ and  all its other eigenvalues $\lambda$ verify $|\lambda|\le \sqrt{\alpha}+o(1)$ with high probability as $n\to\infty$. In this sense,  Erd\H{o}s-R\'enyi graphs asymptotically satisfy the graph Riemann hypothesis, itself is a plausible extension of the notion of Ramanujan graphs to the non-regular case. This may be seen as an analogue of Friedman's Theorem \cite{MR2437174} in the context of Erd\H{o}s-R\'enyi graphs. Similarly, for the Stochastic Block Model, our main result is analogue to recent results on the eigenvalues of random $n$-lifts of base graphs, see \cite{FrKo,B_Fr}. Interestingly, in \cite{B_Fr}, the methods developed in the present paper are adapted to lead to a new and simpler proof of  Friedman's Theorem and its extensions to random $n$-lifts. The random graphs studied here will require a more delicate analysis.

\subsection*{Organization}
The paper is organized as follows. We start  in Section~\ref{sec:preliminaries} with preliminaries on non-backtracking matrices. We  state our main results in Section~\ref{sec:main_results}, namely Theorems \ref{th:ER} and \ref{th:main} characterizing properties of non-backtracking spectra of Erd\H{o}s-R\'enyi graphs and Stochastic Block Models respectively, and Theorem~\ref{th:vecteurs} establishing their consequence for community detection. 

In Section~\ref{subsec:approx}, we provide the algebraic ingredients we shall need. Specifically we establish general bounds on the perturbation of eigenvalues and eigenvectors of not necessarily symmetric matrices that elaborate on the Bauer-Fike Theorem. 

In Section~\ref{sec:arg_er_v2} we give the proof architecture for  Theorem \ref{th:ER} on the non-backtracking spectrum of Erd\H{o}s-R\'enyi graphs, and detail a central argument of combinatorial nature, namely a representation of the non-backtracking matrix $B$ raised to some power $\ell$ as a sum of products of matrices, elaborating on a technique introduced in \cite{LM13}. 

In Section \ref{sec:path_counts} we detail another combinatorial argument needed in our proof, namely we establish bounds on the norms of the various matrices involved in the previous matrix expansion. The latter norm bounds are established by the trace method, adapting arguments due to F\H{u}redi and Koml\'os~\cite{MR637828}.

Section \ref{sec:proof_sbm_v2} gives the proof strategy for Theorem \ref{th:main} on non-backtracking spectra of Stochastic Block Models.

In Section~\ref{subsec:MTGW} we establish convergence results on
multitype branching processes that extend results of Kesten and Stigum
\cite{MR0198552,MR0200979}. These are then used in Section~\ref{sec:local}
to characterize the local structure of the random graphs under
study. Specifically we relate the local statistics of Stochastic Block
Models to branching processes via coupling, and then establish weak
laws of large numbers on these local statistics. These probabilistic
arguments together with the previous algebraic and combinatorial
arguments allow us to conclude the proofs of Theorems \ref{th:ER} and
\ref{th:main}. Section \ref{sec:vecteurs} contains the proof of
Theorem \ref{th:vecteurs} on community detection.

%% file: NBS_preliminaries.tex
\section{Preliminaries on non-backtracking matrices}
\label{sec:preliminaries}
In this section, we explain how the singular value decomposition of $B^\ell$ for $\ell$ large can be used to study the eigenvalues and eigenvectors of $B$. We also comment on analogs of some classical inequalities known for adjacency or Laplacian matrices of regular graphs. 

We set $m = | \vec E|$. 
A priori, $B$ is not a normal matrix. We are interested in its eigenvalues $\lambda_1(B), \ldots, \lambda_m ( B)$ ordered non-increasingly, $| \lambda_1 ( B) |  \geq \ldots \geq | \lambda_m (B)|$.  The Perron-Frobenius Theorem implies notably that $\lambda_1 (B)$ is a non-negative real.   If $G$ is connected, $\lambda_1 (B)$ is equal to the growth rate of the universal cover of $G$ (see Angel, Friedman and Hoory \cite{AFH}).
\subsection{Oriented path symmetry}
An important remark is that despite $B$ not being a normal matrix, it contains some symmetry.  More precisely, we observe that 
$
(B^* )_{ef} = B_{fe } = B_{e^{-1} f^{-1}}
$. Introduce for all $x \in \dR^{\vec E}$ the notation 
\begin{equation}
\check x_e=x_{e^{-1}},\; e\in \vec E.
\end{equation} 
It is then easy to check that for $x,y \in \dR^{\vec E}$ and integer $k \geq 0$, 
\begin{equation}
\label{eq:check}
\langle   y , B^k x \rangle =  \langle B^{k}  \check y ,  \check x \rangle.
\end{equation}
In other words, if $P$ denotes the involution on $\dR^{\vec E}$, $P x  = \check x $, we have for any integer $k \geq 0$, 
$$
B^k P = P B^{*k}.
$$ 
Hence $B^k P$ is a symmetric matrix (in mathematical physics, this type of symmetry is called PT-invariance, PT standing for parity-time). If $(\sigma_{j,k})$, $1 \leq j \leq m$, are the eigenvalues of $B^kP$ and $(x_{j,k})$, $1 \leq j \leq m$,  is an orthonormal basis of eigenvectors, we deduce that
\begin{equation}\label{eq:BkSVD}
B^k = \sum_{j=1}^m \sigma_{j,k} x_{j,k} \check x_{j,k}^* .
\end{equation}
We order the eigenvalues, 
$$
\sigma_{1,k} \geq |\sigma_{2,k}| \geq \ldots  \geq |\sigma_{m,k}|.  
$$
From Perron-Frobenius theorem, $x_{1,k}$ can be chosen to have non-negative entries. Since $P$ is an orthogonal matrix, $(\check x_{j,k})$, $1 \leq j \leq m$, is also  an orthonormal basis of $\dR^{\vec E}$. In particular, \eqref{eq:BkSVD} gives the singular value decomposition of $B^k$. Indeed, if $s_{j,k} = \ABS{\sigma_{j,k}}$ and $y_{j,k} =\sign(\sigma_{j,k}) \check x_{j,k}$, we get  
\begin{equation}\label{eq:BkSVD2}
B^k = \sum_{j=1}^m s_{j,k} x_{j,k} y_{j,k}^* .
\end{equation}
This is precisely the singular value decomposition of $B^k$. 

For example, for $k=1$, it is a simple exercise to compute $(\sigma_{j,1})_{1\leq j  \leq m}$. We find that the eigenvalues of $BP$  are $(\deg(v) - 1)$, $1 \leq v \leq n$, and $-1$ with multiplicity $m - n$. In particular, the singular values of $B$ contain only information on the degree sequence of the underlying graph $G$.

For large $k$ however, we may expect that the decomposition \eqref{eq:BkSVD} carries more structural information on the graph (this will be further discussed in Subsection \ref{subsec:ccab} below). This will be the underlying principle in the proof of our main results. For the moment, we simply note the following. Assume that $B$ is irreducible. From Perron-Frobenius theorem, if $\xi$ is the Perron eigenvector of $B$, $\|\xi\|= 1$, then for any $n$ fixed,
\begin{equation}\label{eq:PR1}
\lambda_1 (B) = \lim_{k \to \infty} \sigma_{1,k}^{1/k}\quad  \hbox{ and } \quad \lim_{k \to \infty} \| x_{1,k} - \xi \|  = 0. 
\end{equation}
A quantitative version of the above limits will be given in the forthcoming Proposition \ref{prop:sing2eig}. Another consequence of \eqref{eq:BkSVD} is that, for $i \ne j$, $x_{i,k}$ and $\check x_{j,k}$ should be nearly orthogonal if these vectors converge as $k \to \infty$. Indeed, a heuristic computation gives
$$
\langle x_{i,k} , \check x_{j,k} \rangle  = \frac{ \langle B^k \check x_{i,k} ,  B^{* k} x_{j,k} \rangle }{ \sigma_{i,k} \sigma_{j,k} }  = \frac{ \langle B^{2k} \check x_{i,k} ,   x_{j,k} \rangle }{ \sigma_{i,k} \sigma_{j,k} }  
\simeq \frac{ \langle B^{2k} \check x_{i,2k} ,   x_{j,2k} \rangle }{ \sigma_{i,k} \sigma_{j,k} }   = \frac{ \si_{i,2k}  \langle  x_{i,2k} ,   x_{j,2k} \rangle }{ \sigma_{i,k} \sigma_{j,k} }  = 0. 
$$ 
We will exploit this general phenomenon in the proof of our main results.

\subsection{Chung, Cheeger and Alon-Boppana inequalities for non-backtracking matrices}

\label{subsec:ccab}

The aim of this subsection is to advocate the use of non-backtracking matrices. Here, we discuss briefly candidate counterparts for irregular graphs of inequalities that are classical in the context of regular graphs.  This subsection will not be used in the proof of our main results, it may be skipped.

The diameter bound of Chung~\cite{MR965008} gives an upper bound on the diameter of a regular graph in terms of its spectral gap. The following lemma, expressed in terms of the decomposition \eqref{eq:BkSVD2} of $B^k$, is an analogue: 

\begin{lemma}
Let $e = (u,u')$, $f = (v,v') \in \vec E$ be such that  
$$
x_{1,k} (e) x_{1,k} (f) > s_{2,k} / s_{1,k}.
$$
Then, the graph distance between $u$ and $v$ is at most $k+1$. 
\end{lemma}
\begin{proof}
Observe that if $B^{k}_{e f^{-1}} > 0$ then $u$ and $v$ are at most at distance $k+1$. On  the other hand from \eqref{eq:BkSVD2}, 
$$
B^{k}_{e f^{-1}} - s_{1,k} x_{1,k} (e) x_{1,k} (f)  = \sum_{j=2}^m  s_{j,k} x_{j,k} (e) x_{j,k} (f) 
$$
has absolute value at most $s_{2,k}$ from Cauchy-Schwartz inequality and the orthonormality of the $x_{j,k}$, $1 \leq j \leq m$. Thus finally, 
$
B^{k}_{e f^{-1}} \geq  s_{1,k} x_{1,k} (e) x_{1,k} (f) -s_2.
$
\end{proof}

Cheeger-type inequalities connect the expansion ratio (isoperimetry) of the graph and its spectral gap, for a survey see \cite{MR2247919}. For a subset $X \subset \vec E$ of edges, we measure its volume by
$$
\vol_k (X) = \sum_{e \in X} x^2_{1,k} (e).
$$
By construction, our volume is normalized with $\vol_k (\vec E) = 1$. We say that $X \subset \vec E$ is edge-symmetric if $\check X = X$. For example the set of edges adjacent to a given subset of vertices is edge-symmetric. If $X,Y$ are edge-symmetric, we define
$$
E_k ( X, Y) =  \sum_{e \in X, f \in Y} x_{1,k} (e) x_{1,k} (f) B^{k}_{ef} .
$$
Since $B^{k}_{ef}$ is the number of non-backtracking walks of length $k+1$ starting with $e$ and ending with $f$, $E_k(X,Y)$ measures a kind of conductance between $X$ and $Y$ with a proximity range of radius $k+1$. If $X^c = X \backslash \vec E$, the scalar
$$
\Sigma_k (X) = \Sigma_k (  X^c  ) = E_k (X, X^c).  
$$
can be thought of as the outer surface of a set $X$. The $k$-th expansion ratio of $G$ is then defined as 
$$
h_k  = \min_{X\subset \vec E, \check X=X} \frac{\Sigma_k (X)}{\vol_k (X) \wedge \vol_k (X^c)}.
$$

In \eqref{eq:BkSVD}, after reordering the eigenvalues of $B^k P$ as $\sigma_{1,k} \geq \sigma_{2,k} \geq \ldots  \geq \sigma_{m,k}$, $  \sigma_{1,k} - \sigma_{2,k}$ plays the role of the spectral gap in the classical Cheeger inequality. With this new convention, the following lemma is the analog of the easy part of Cheeger's inequality for graphs. 

\begin{lemma}
$$
\si_{1,k} - \si_{2,k} \leq 2 h_k. 
$$ 
\end{lemma}
\begin{proof}
The argument is standard. For simplicity, we drop the index $k$. From Courant-Fisher min-max Theorem, we have
$$
\sigma_{2} = \max_{ \langle  x  , x_ 1 \rangle  = 0 }\frac{ \langle x , B^k P x  \rangle }{ \|x \|^2} =  \max_{ \langle  x  , x_ 1 \rangle  = 0} \frac{ \langle \check x , B^k x  \rangle }{ \|x \|^2} .
$$
Let $X\subset \vec E$ be edge-symmetric. We set $$x (e) = \frac{ x_1 (e)} {  \vol(X) }\IND ( e \in X) - \frac{ x_1 (e)}{  \vol(X^c)} \IND ( e \in X^c) .$$
By construction $\langle x , x_1 \rangle = 0$ and $\|x \|^2  = 1 / \vol (X) + 1 / \vol(X^c)$. Hence, 
$$
\PAR{ \frac{1}{\vol (X)}+ \frac{1}{ \vol(X^c)}} \si_2  \geq  \langle \check x , B^k x \rangle =  \sum_{e ,f} B^k_{ef} x(e) \check x(f). 
$$
However, using the edge-symmetry of $X$ and $X^c$,
\begin{align*}
\langle \check x , B^k x \rangle  & =   \frac 1 {\vol(X)^2} \sum_{e , f \in X } B^k_{ef} x_1(e) \check x_1(f) + \frac 1 {\vol(X^c)^2} \sum_{e , f \in X^c } B^k_{ef} x_1(e) \check x_1(f)  -  \frac {2 \Sigma(X)} {\vol(X) \vol(X^c)}.
\end{align*}
Also,  using the singular value equation $B^k \check x_1 = \si_1 x_1$, 
$$
 \sum_{e ,f \in X} B^k_{ef} x_1(e) \check x_1(f)  =  \sum_{e  \in X , f  \in \vec E} B^k_{ef} x_1(e) \check x_1(f)  - \Sigma(X) = \si_ 1 \vol(X) - \Sigma(X),
$$
and similarly for $X^c$. So finally, 
$$
\PAR{ \frac{1}{\vol (X)}+ \frac{1}{ \vol(X^c)}} \si_2 \geq \PAR{ \frac{1}{\vol (X)}+ \frac{1}{ \vol(X^c)}} \si_1   -   \PAR{ \frac{1}{\vol (X)}+ \frac{1}{ \vol(X^c)}}^2 \Sigma(X).
$$
Since, for $x,x' >0$, $1 / x + 1 / x ' \leq  2 / (x \wedge x')$, it concludes the proof.
\end{proof}

The Alon-Boppana theorem gives a lower bound on the second largest eigenvalue of the adjacency matrix of a regular graph (see  \cite{MR1124768,MR2247919}). 
We conclude this paragraph with an elementary bound of this type. We introduce for $e \in \vec E$,
$$
S_k (e)  = \langle \delta_e, B^k \chi \rangle= \| B^k P \delta_{e}\|_1. 
$$
In words, $S_k (e)$ is the number of non-backtracking walks of length $k+1$ starting with $e$. As already pointed, if $B$ is irreducible, the Perron eigenvalue is the growth rate of the universal cover of the graph: for any $e \in \vec E$,
$$
\lambda_1 (B)  = \lim_{k \to \infty}  s_{1,k}^{1/k} = \lim_{k\to \infty} S_k (e)^{1/k}. 
$$
We observe that 
$$
s_{1,k} ^2 + (m-1) s^2_{2,k} \geq \tr ( B^k B^{*k})  \geq \sum_{e \in \vec E} S_{k} (e). 
$$
Hence, we find, 
\begin{equation}\label{eq:weakAB}
s^2_{2,k} \geq  \frac{1}{m} \sum_{e \in \vec E} S_{k} (e)  - \frac {s_{1,k} ^2} m .
\end{equation}
This last crude inequality gives a lower bound on the second largest singular value of $B^k$.

%% file: NBS_main_results.tex
\section{Main results}\label{sec:main_results}
We now state our results on the non-backtracking spectra of Erd\H{o}s-R\'enyi graphs first, and Stochastic Block Models next.
\subsection{Erd\H{o}s-R\'enyi graphs}
Let the vector $\chi$ on $\dR^{\vec E}$ be defined as
$$
\chi (e) = 1,\; e \in \vec E.
$$
The Euclidiean norm of a vector $x \in \dR^d$ will be denoted by $\|x\|$. We have the following theorem.
\begin{theorem}\label{th:ER}
Let $G$ be an Erd\H{o}s-R\'enyi graph with parameters $(n, \alpha/n)$ for some fixed parameter $\alpha >1$. Then, with probability tending to $1$ as $n\to\infty$, the eigenvalues $\lambda_i(B)$ of its non-backtracking matrix $B$ satisfy
$$
 \lambda_1 (B)  =  \alpha  + o(1) \; \hbox{ and }  \;  |\lambda_2 (B) | \leq \sqrt \alpha + o(1).  
$$
Moreover the normalized Perron-Frobenius eigenvector associated to $\lambda_1 (B)$ is asymptotically aligned with 
$$
 \frac{  B^{\ell} B^{* \ell} \chi }{\|  B^{\ell}   B^{* \ell}  \chi\|},
$$
 where  $\ell \sim  \kappa \log_\alpha n $ for any $0 < \kappa < 1/6$.  
\end{theorem}

Theorem \ref{th:ER} is illustrated by Figure \ref{fig:simu}. We conjecture that the lower bound $|\lambda_2 (B)| \geq \sqrt \alpha - o(1)$ holds, it is reasonable in view of Figure \ref{fig:simu}. We shall prove a weaker lower bound, see forthcoming Remark \ref{re:wRama}. It is also an interesting open problem to study the convergence of the empirical distribution of the eigenvalues of $B$.

\begin{figure}[htb]\label{fig:simu}
\begin{center}\includegraphics[width = 8cm]{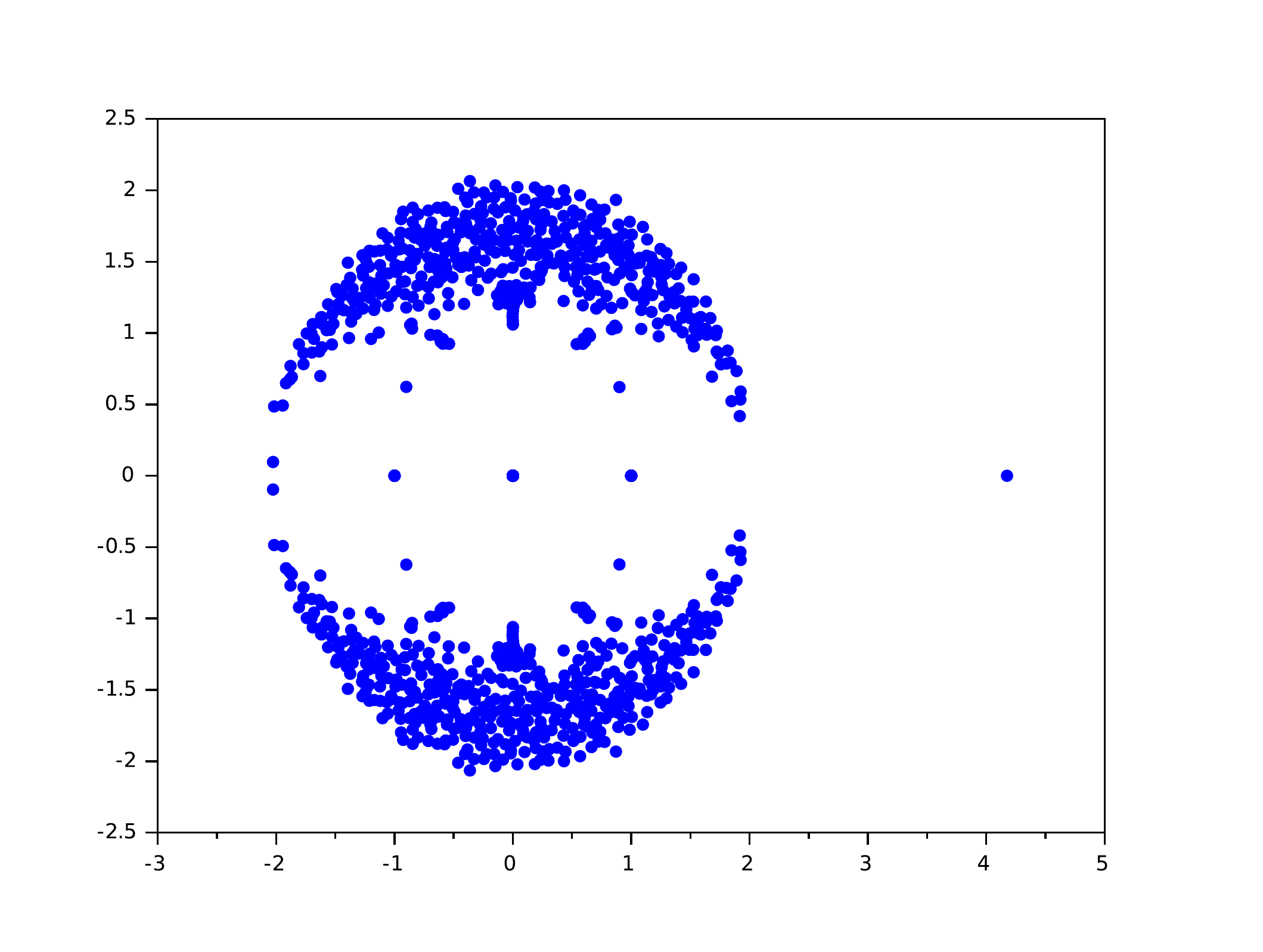} \quad  \includegraphics[width = 8cm]{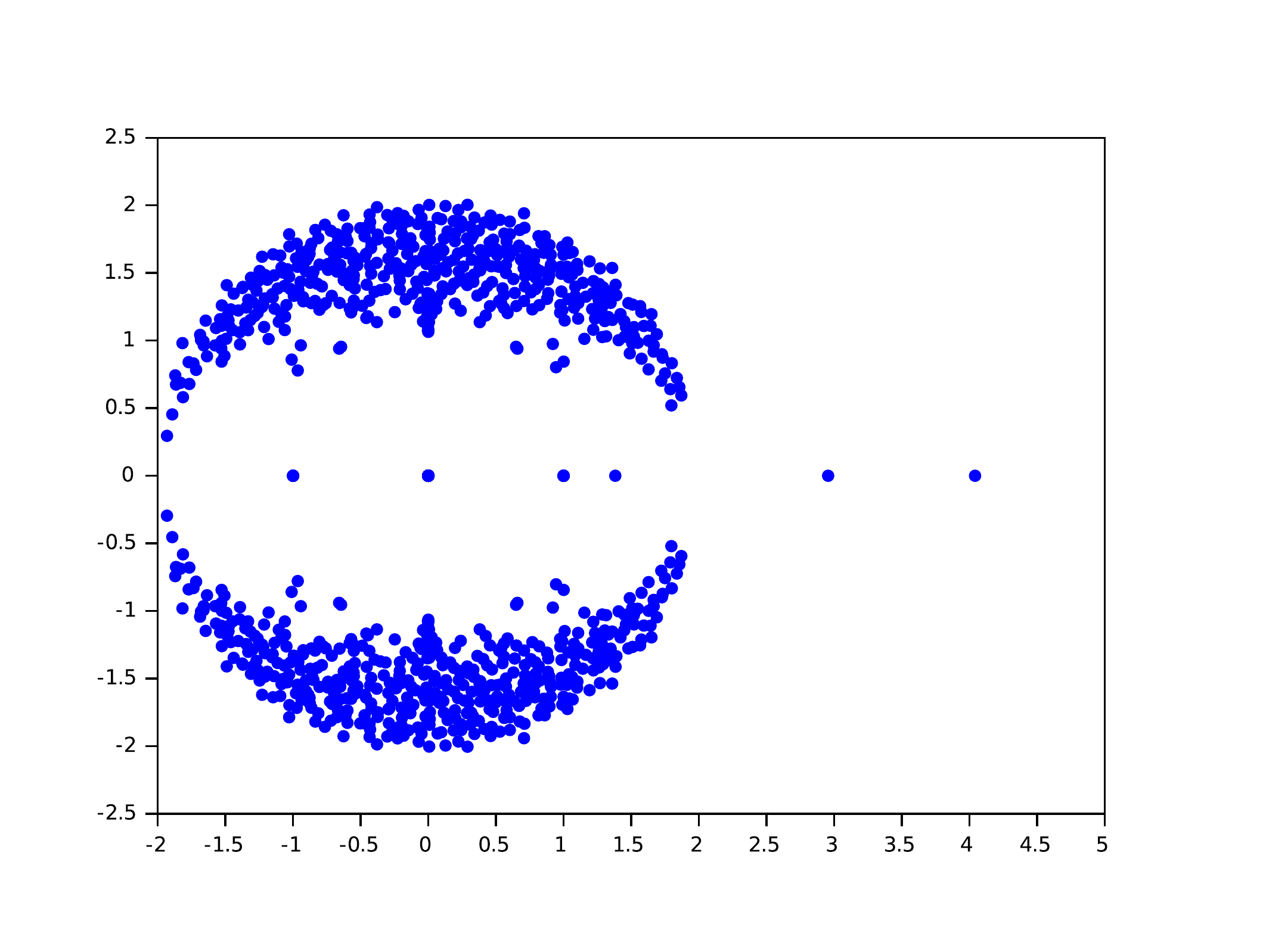}
\caption{Left : eigenvalues of $B$ for a realization of an Erd\H{o}s-R\'enyi graph with parameters  $(n, \alpha/n)$ with $n = 500$, $\alpha =4$. Right : eigenvalues of $B$ for Example \ref{ex:ex2} with $n = 500$, $r =2$, $a = 7$, $b=1$.}
\end{center}\end{figure}

\subsection{Stochastic Block Model}\label{sec:main_sbm}

For integer $k \geq 1$, we set $[k] = \{1, \ldots, k\}$. 
We consider a random graph $G = ( V,E)$ on the vertex set $V = [n]$ defined as follows.   Each vertex $v \in [n]$ is given a {\em  type} $\sigma_n(v)$ from the set $[r]$ where the number of types $r$ is assumed fixed and the map $\sigma_n : [n] \to [r]$ is such that, for all $i \in [r]$,
\begin{equation}\label{eq:convMET}
\pi_n(i) := \frac 1 n \sum_{v=1} ^n \IND ( \sigma_n (v) = i) = \pi(i) + o \PAR{ 1},
\end{equation} 
for some probability vector $\pi = (\pi(1), \cdots, \pi(r))$. For ease of notation, we often write $\si$ in place of $\si_n$. 

Given a symmetric matrix  $W \in M_{r} (\dR_+)$ we assume that there is an edge between vertices $u$ and $v$ independently with probability
$$\frac{W(\sigma(u),\sigma(v))}{n}\wedge 1.$$
We set $\Pi = \diag ( \pi(1), \ldots, \pi(r))$ and introduce the mean
progeny matrix $M = \Pi W$ (the branching process terminology will be
clear in Section \ref{subsec:MTGW}). Note that the eigenvalues of $M$ are the same as the ones of the
symmetric matrix $S = \Pi^{1/2}W\Pi^{1/2}$ and in particular are
real-valued. They are also the same as the ones of the expected adjacency matrix
$\bar{A}:=\dE(A)$.
We denote them by $\mu_k$ and order them by their absolute value,
\BEAS
| \mu_r | \leq \dots \leq | \mu_2 | \leq  \mu_1 ,
\EEAS
We shall make the following assumptions:
\begin{equation}\label{hyp1}
\mu_1>1\hbox{ and }M\hbox{ is positively regular,}
\end{equation}
i.e. for some integer $k \geq 1$, $M^k$ has positive coefficients. In particular, $\mu_1> \max_{k\geq 2}|\mu_k|$ is the Perron-Frobenius eigenvalue. It implies notably that for all $i \in [r]$, $\pi(i) >0$. We define $r_0$ by  
$$
\mu_k^2 > \mu_1 \quad \hbox{ for all } k \in [r_0] \quad  \hbox{ and } \quad \mu^2_{r_0 +1} \leq \mu_1,
$$
 (with $\mu_{r+1} = 0$). Since $M = \Pi^{1/2} S \Pi^{-1/2}$, the matrix $M$ is diagonalizable. Let $\{u_i\}_{i\in [r]}$ be an orthonormal basis of eigenvectors of $S$ such that $S u_i = \mu_i u_i$. Then $\phi_i := \Pi^{-1/2} u_i$ and  $\psi_i = \Pi^{1/2} u_i$ are the left and right-eigenvectors associated to eigenvalue $\mu_i$, $\phi_i ^* M = \mu_i \phi^* _i$, $ M \psi_i = \mu_i \psi_i$. We get 
\begin{eqnarray}
\langle \phi_i , \psi_j \rangle  =  \delta_{ij}, &\mbox{and, }&
\langle \phi_i , \phi_j \rangle_{\pi} =  \delta_{ij} \label{eq:phiONpi},
\end{eqnarray}
where  $\langle x , y \rangle_{\pi} = \sum_{k} \pi(k) x_k y_k$ denotes the usual inner product in $\ell^2 (\pi)$. The following spectral decompositions will also be useful
\begin{eqnarray}
M = \sum_{k =1} ^{r} \mu_k  \psi_k   \phi_k ^*\quad  \hbox{ and } \quad W = \sum_{k =1} ^{r} \mu_k  \phi_k  \phi_k ^* \label{eq:Wspec},
\end{eqnarray}
where the second identity comes from $\psi_k = \Pi \phi_k$ and $W =\Pi^{-1} M$.

We will make the further assumption that each vertex type has the same asymptotic average degree $\alpha > 1$, i.e., 
\BEA
\label{cond:deg} \alpha =\sum_{i=1}^r \pi(i) W_{ij} = \sum_{i=1}^r M_{ij} \quad  \; \hbox{ for all } j\in [r].
\EEA
This entails that $M^*/ \alpha$ is a stochastic matrix and we then have
\begin{equation}\label{eq:stoch}
\mu_1=\alpha >1 , \quad \phi_1 = \IND \quad  \hbox{ and } \; \psi_1 =  \pi^*.
\end{equation} 
We will also assume that a quantitative version of \eqref{eq:convMET} holds, namely that for some $\gamma \in (0,1]$,
\begin{equation}\label{eq:defgamma}
\| \pi - \pi_n \|_{\infty} = \max_{i  \in [r]} | \pi(i) - \pi_n(i) | = O ( n ^{-\gamma}). 
\end{equation}

The random graph $G$ is usually called the stochastic block model (SBM for short) or inhomogeneous random graph, see Bollob{\'a}s, Janson and Riordan \cite{MR2337396} and Holland, Laskey and Leinhardt~\cite{Holland83}. A popular case is when the map $\sigma$ is itself random and $\sigma(v)$ are i.i.d. with distribution  $(\pi(1), \cdots , \pi(r))$. In this case, with probability one, condition \eqref{eq:defgamma} is met for any $\gamma < 1/2$.

\begin{example}
If $r=2$, then we have $\pi(1)=1-\pi(2)$. Under condition
(\ref{cond:deg}), we have $W_{22} =  (\pi(1)W_{11}+(1-2\pi(1))W_{12}) / (1-\pi(1))$
so that
$\mu_1= \alpha =\pi(1)W_{11}+(1-\pi(1))W_{12}$ and $\mu_2 =\pi(1) \left(W_{11}-W_{12}\right)$.
\end{example}

\begin{example}\label{ex:ex2}
If $r\geq 2$, $\pi(i)= 1 / r$ and $W_{ii}=a\neq b = W_{ij}$ for all
$i\neq j$ so that condition (\ref{cond:deg}) is satisfied. We have
$\mu_1=\alpha =  (a+(r-1)b  ) / r$ and $\mu_2= \ldots = \mu_r = (a-b )/ r$.
\end{example}

For $k \in [r]$, we  introduce the vector on $\dR^{\vec E}$,
\begin{equation}\label{eq:def_chi}
\chi_k (e) = \phi_k ( \sigma (e_2) )  \quad \hbox{ for all $e \in \vec E$}.
\end{equation}
In particular, $\chi_1 = \chi$.  
Our main result is the following generalization of Theorem \ref{th:ER}.

\begin{theorem}\label{th:main}
Let $G$ be an SBM as above such that hypotheses~(\ref{hyp1},\ref{cond:deg},\ref{eq:defgamma}) hold. Then 
with probability tending to $1$ as $n\to\infty$,  
\begin{align*}
&  \lambda_k (B)  =  \mu_k  + o(1) \; \hbox{ for } k\in [r_0], \; \hbox{ and for $k > r_0$, }  \;  |\lambda_{k} (B) | \leq \sqrt \alpha+ o(1).
\end{align*}
Moreover,  if $\mu_k$ is a simple eigenvalue of $M$ for some $k\in[r_0]$, then a normalized  eigenvector, say $\xi_k$, of $\lambda_k (B)$ is asymptotically aligned with 
\begin{equation}\label{eq:defpseudvec}
 \frac{  B^{\ell}   B^{*\ell} \check \chi_k }{\|  B^{\ell} B^{*\ell} \check \chi_k\|},
\end{equation}
 where  $\ell \sim  \kappa \log_\alpha n $ for any $0 < \kappa < \gamma /6$.  Finally, the vectors $\xi_k$ of these simple eigenvalues are  asymptotically orthogonal.
\end{theorem}

It follows from this result that a non-trivial estimation of the node types $\sigma(v)$ is feasible on the basis of the eigenvectors $\{\xi_k\}_{2\le k\le r_0}$ provided $r_0>1$.  More precisely, for vertex type estimators $\hat{\sigma}(v):[n]\to[r]$ based on the observed random graph $G$, following Decelle et al.~\cite{Decelle11}, define the overlap $\hbox{ov}(\hat{\sigma},\sigma)$ as the minimum over permutations $p:[r]\to[r]$ of the quantity
\begin{equation}\label{eq:defoverlap}
\frac{1}{n}\sum_{v=1}^n\IND_{\hat{\sigma}(v)=p\circ\sigma(v)}-\max_{k\in[r]}\pi(k).
\end{equation}
We shall say that $\hat \si$ has asymptotic overlap $\delta$ if $\hbox{ov}(\hat{\sigma},\sigma)$ converges in probability to $\delta$ as $n$ grows. It has asymptotic positive  overlap if for some $\delta >0$, $\hbox{ov}(\hat{\sigma},\sigma) > \delta$ with probability tending to $1$ as $n$ grows.  Note that an asymptotic overlap of zero is always achievable by assigning to each vertex the type $k^*$ that maximizes $\pi(k)$. In the case where all communities have asymptotically the same size, i.e. $\pi(i)\equiv 1/r$, zero overlap is also achieved by assigning types at random.

As conjectured in~\cite{Decelle11} and proven in~\cite{MNS}, in the setup of Example 2 with $r=2$, the best possible overlap is $o(1)$ with high probability when $r_0=1$, i.e. when $\mu_2\le \sqrt{\mu_1}$. Conversely, adapting the argument in~\cite{LM13}, when $r_0>1$, we have the following

\begin{theorem}\label{th:vecteurs}
Let $G$ be an SBM as above such that hypotheses~(\ref{hyp1},\ref{cond:deg},\ref{eq:defgamma}) hold. Assume further that $\pi(i)\equiv 1/r$, that $r_0>1$ and that for some $k\in\{2,\ldots,r_0\}$, $\mu_k$ is a simple eigenvalue of $M$. Let $\xi_k\in \RR^{\vec E}$ be a normalized  eigenvector of $B$ associated with $\lambda_k (B) $.

Then, there exists a deterministic threshold $\tau \in\dR$,  a partition $(I^+,I^-)$ of $[r]$ and a random signing $\omega \in \{-1,1\}^V$ dependent of $\xi_k$ such that the following estimation procedure yields asymptotically positive overlap: assign to each vertex $v$ a label $\hat \sigma(v)$ picked uniformly at random from $I^+$ if $\omega (v) \sum_{e : e_2 = v} \xi_k(e)>  \tau/\sqrt{n}$ and from $I^-$ otherwise.
\end{theorem}

The reason for the existence of the signing $\omega \in \{-1,1\}^V$ in the above statement is that we do not know a priori whether the vector $\xi_k$ or $-\xi_k$ is asymptotically close to \eqref{eq:defpseudvec}. In the simplest case, we will be able to estimate this sign and the vector $\omega$ will be equal to $-\IND$ or $\IND$ and $I^+ = \{ i \in [r] : \phi_k (i) > 0 \}$, $I^- = [r] \backslash I^+$. 


\subsection{Notation}

We say that a sequence of events $E_n$ holds {\em with high probability}, abbreviated \whp, if $
\lim_{n \to \infty} \dP (E_n )   = 1. 
$
The operator norm of a square matrix $C$ is denoted by $$\| C\| =
\sup_{ x  \ne 0}\frac{ \|C x \|}{ \|x\|}.$$
We denote by $C^*$ the transpose of $C$.

Given a (non-oriented) graph $G=(V,E)$, we denote by $\gamma=(\gamma_0,\dots,
\gamma_k)$ a walk of length $k$ where each $\gamma_i\in V$ and
$\{\gamma_i, \gamma_{i+1}\}\in E$ for all $i\in\{0,\dot k-1\}$. We
also denote the concatenation of two walks $\gamma$ and $\gamma'$ by
$(\gamma,\gamma')$.
A walk is non-backtracking if for all $i\in \{0,\dots, k-2\}$,
$\gamma_i\neq \gamma_{i+2}$. A walk contains a cycle if there exists
$i\neq j$ with $\gamma_i=\gamma_j$.

%% file: NBS_algebra.tex
\section{Algebraic tools: Perturbation of Eigenvalues and Eigenvectors}

\label{subsec:approx}

One main tool in our analysis is the Bauer-Fike Theorem. The form given below elaborates on the usual statement of the Theorem which in general  omits the second half.
\begin{theorem}\label{bauer-fike} (Bauer-Fike Theorem; see \cite[Theorem VI.5.1]{MR1477662}). 
Let $D$ be a diagonalizable matrix such that for some invertible matrix $V$ and diagonal matrix $\Lambda$ one has $D=V^{-1}\Lambda V$. Let $E$ be a perturbation matrix.

Then any eigenvalue $\mu$ of $D+E$ verifies
\begin{equation}\label{B-F-1}
\min_{i}|\mu-\lambda_i|\le \|E\| \cdot \|V\| \cdot \|V^{-1}\|,
\end{equation}
where $\lambda_i$ is the $i$-th diagonal entry of $\Lambda$. 

Denote by $R$ the right-hand side of~(\ref{B-F-1}) and $C_i:={\mathcal{B}}(\lambda_i,R)$ the ball centered at $\lambda_i$ with radius $R$. Let ${\mathcal{I}}$ be a set of indices such that
\begin{equation}
\left(\cup_{i\in {\mathcal{I}}} C_i\right)\bigcap \left(\cup_{i\notin {\mathcal{I}}} C_i\right)=\emptyset.
\end{equation}
Then the number of eigenvalues of $D+E$ in $\cup_{i\in {\mathcal{I}}} C_i$ is exactly $| {\mathcal{I}}|$. 
\end{theorem}
The following proposition on perturbation of rank one matrices will be  a basic ingredient to deduce from expressions like \eqref{eq:BkSVD} quantitative versions of \eqref{eq:PR1}. It relies on the stability of eigenvalues and eigenvectors of matrices which are not too far from being normal ($A$ is normal if and only if $A^*A=A A^*$; see \cite{MR1091716}), and with a well separated spectrum.

\begin{proposition}\label{prop:sing2eig}
Let $A \in M_n ( \dR)$, $\ell' , \ell \geq 1$ be mutually prime integers, $\theta \in \dR\setminus\{0\} $, and $c_0,\; c_1 > 0$ such that for any $ k \in \{ \ell, \ell' \}$, for some $x_k,y_k \in \dR^n$, $R_k \in M_n(\dR)$,
$$
A^k = \theta^k x_k y_k^* + R_k,
$$
with $\langle y_k , x_k \rangle \geq c_0$, $\|x_k \| \|y_k \| \leq c_1$ and 
$$\| R_k \| <   \frac{c_0^2 }{ 2(\ell\vee \ell') c_1 }  | \theta|^k.$$ Let $(\lambda_i)$, $1 \leq i \leq n$ be the eigenvalues of $A$, with $|\lambda_n| \leq \ldots \leq | \lambda_1|$. Then, $\lambda_1$ has multiplicity one and we have
$$
| \lambda_1 - \theta | \leq   C |\theta | /  \ell \quad \hbox{ and, for $i \geq 2$,} \quad \ABS{ \lambda_i}\leq \PAR{ \frac{2 c_1}{c_0} } ^{1 /\ell} \| R_\ell \|^{1/\ell},
$$
with $C=\pi/2+2\sqrt{c_1\vee 1}\log(2(c_1\vee c_0^{-1}))$. Moreover, there exists a unit eigenvector $\psi$ of $A$ with eigenvalue $\lambda_1$ such that 
$$
\NRM{ \psi - \frac{x_\ell}{\|x_\ell\|} } \leq 8 c_0^{-1} \|R_\ell\| |\theta|^ {-\ell}.
$$
\end{proposition}

\begin{proof}
We can assume without loss of generality that $\theta = 1$. 
We fix $k \in \{ \ell, \ell'\}$ and let $ \tilde x = x_k / \|x_k\|$, $\tilde y = y_k / \| y_k \|$, $\sigma =    \|x_k \| \|y_k\|$, $\nu = \langle y_k , x_k \rangle$.  We have 
$$A^{k} =  \sigma \tilde x \tilde y^* + R_k = S  + R_k.$$
Our objective is to apply Bauer-Fike Theorem.  To this end, we write $S = U DU^{-1}$, where $D = \diag( \nu, 0, \ldots , 0)$, $U = ( \tilde x, f_2, \ldots, f_n)$ with $f_1 = \tilde y $ and $(f_i)_{i \geq 2}$ is an orthogonal basis of $\tilde y^\perp$ (we will soon check that $U$ is indeed invertible).  We can also assume that $\tilde x \in \mathrm{span}  ( f_1, f_2)$. Obviously, the eigenvalues of $A^k$ coincide with the eigenvalues of $D + U^{-1} R_k U$. We have to compute the condition number of $U$ : $$\kappa (U) = \| U \| \| U^{-1} \|.$$ We consider the unitary matrix $V = (f_1, \ldots, f_n)$.  Let $a,b\in\dC$ be such that  $\tilde x = a f_1 + b f_2$. We find 
$$
V^* U  = 
\begin{pmatrix} 
W  &  0 \\
 0 &  I_{n-2} \\
\end{pmatrix}
\quad \hbox{ with } \quad W = \begin{pmatrix} 
a  &  0 \\
 b &   1 \\
\end{pmatrix}.$$
In particular, 
$$
\kappa ( U ) = \kappa ( V^* U ) = \kappa (W).
$$
As $|a|^2+|b|^2=1$ we have 
$$W^* W  = \begin{pmatrix} 
1  &  \bar b  \\
 b &   1 \\
\end{pmatrix}.
$$
The eigenvalues of $WW^*$ are $1 \pm |b|$. We deduce that 
$$
\kappa (U ) = \sqrt{\frac{ 1 + |b|}{  1 - |b|}}.
$$
Now, by assumption, $|b| =  \sqrt{1 - |a|^2} = \sqrt{ 1 - \ABS{\langle \tilde x , \tilde y \rangle}^2} \leq \sqrt{1 - c_0^2 c_1^{-2}} \leq 1 - c_0^2c_1^{-2} /2 $. We obtain that
$$
\kappa (U ) \leq \kappa = 2 c_1 c_0^{-1}.
$$
Notice that by assumptionn $\ell \vee \ell' \geq 2$ and $2 \kappa \| R_k \| < c_0 \leq| \nu|$.
An application of the Theorem~\ref{bauer-fike} to $D + U^{-1} R_k U$ then implies that there is a unique eigenvalue of $A^k$ in the ball $\{ z \in \dC : |  \nu - z | \leq \kappa \| R_k \| \}$ and all the other eigenvalues lie in the disjoint domain $\{ z \in \dC : | z | \leq \kappa \|R_k\| \}$. Consequently,
\begin{equation}\label{eq:lambdakA}
| \lambda_1^k  - \nu | \leq \kappa \| R_k \|  \quad \hbox{ and, for $i \geq 2$,} \quad \ABS{ \lambda _i}^k\leq \kappa \| R_k \| .
\end{equation}
In particular, the eigenvalue $\lambda_1^k$ has multiplicity one in $A^k$ and thus $\lambda_1$ has multiplicity one in $A$. 
We now bound the difference between $\lambda_1$ and $\theta = 1$. First, by assumption, 
\begin{equation}\label{eq:mukA}
c_0 \leq \nu \leq c_1.
\end{equation}
From \eqref{eq:lambdakA}, we deduce that $||\lambda_1^k|-\nu|\le c_0/2$ and hence
$c_0/ 2 \leq  | \lambda_1^{k}|  \leq 2c_1.$
Since for all $x\in\dR$, $|e^x - 1| \leq |x|e^{x_+} $, we get
$$
|  | \lambda_1|  - 1| \leq \ABS{ ( 2 c_1 )^{1/k} - 1 } \vee \ABS{ ( c_0 /2 )^{1/k} - 1 } \leq {c_2 / k}, 
$$
with $c_2 = \sqrt{(2c_1) \vee 1} \log ( 2 (c_1\vee c_0^{-1} ))$. 

We now control the argument $\omega \in (-\pi,\pi]$ of $\lambda_1 =  | \lambda_1|  \exp ( i \omega) $. By assumption, $\nu> c_0\ge 0$, hence by \eqref{eq:lambdakA} the real part of $ \lambda_1^{k} $ is positive. There thus exist  an integer $q\in\mathbb{Z}$ and some $\epsilon\in(-1/2,1/2)$ such that
$$
k\omega=2q\pi +\veps \pi.
$$
Since $|x| \leq (\pi/2) |\sin (x)|$ for $x \in [-\pi/2,\pi/2]$, we obtain from \eqref{eq:lambdakA} that
$$
\kappa \| R_k\|\ge |\lambda_1|^k |\sin(\veps\pi)|\ge |\lambda_1|^k 2|\veps|\ge c_0 |\veps|,
$$
so that $|\veps| \leq (\kappa /c_0)  \| R_k \| = c_1 c_0^{-1}  \| R_k \|$. The above holds for $k \in \{ \ell, \ell' \}$. Hence, with the notation $\veps  = \veps(\ell)$, $\veps' = \veps (\ell')$, $q = q(\ell)$, $q' = q(\ell')$ we have  
$$
\omega = \frac{2 q \pi }{  \ell}  + \frac{\veps \pi }{ \ell} =  \frac{2 q' \pi }{  \ell'}  + \frac{\veps' \pi }{ \ell'},
$$
so that 
$$
q \ell' - q' \ell = \frac{\veps' \ell -\veps \ell' }{2}. 
$$
Using the assumption $\|R_k\|< c_0^2/(2c_1 (\ell\vee \ell'))$, we find that both $\veps' \ell$ and $\veps\ell'$ have modulus strictly less than 1 and hence so does the right hand side of the last display. It follows that $q \ell' = q' \ell$. Since $\ell$ and $\ell'$ are mutually prime, we deduce that $\ell$  divides $q$ and $\ell'$ divides $q'$, so that modulo $2\pi$, $\omega=\veps\pi/\ell=\veps'\pi/\ell'$. Thus for $k \in \{ \ell, \ell'\}$,
$$
| \lambda_1 - 1 | \leq | \lambda_1 | | e^{i \pi \veps  / k }- 1 | + ||\lambda_1| - 1 | \leq \PAR{ 1 + \frac {c_2} k } \frac{\pi | \veps | }{k}+ \frac {c_2} k . 
$$
As $\epsilon<1/k$, the right-hand side of the above is no larger than $k^{-1}(2c_2+\pi/2)$ if $k\ge 2$, which must hold for some $k\in\{\ell,\ell'\}$. This concludes the proof of the claim of Proposition \ref{prop:sing2eig} on eigenvalues. 

Consider now a normed eigenvector $z$ of $A$ associated with $\lambda_1$, which for fixed $k\in\{\ell,\ell'\}$ admits the orthogonal decomposition $z=z_0 +z^{\perp}$, 
where $z_0 \in\hbox{span}(\tilde y_k,\tilde x_k)$. Applying $A^k$, we obtain
$$
\lambda_1^k z= \sigma \tilde x_k \tilde{y}_k^* z_0 +R_k z.
$$
Projecting onto $\{\tilde y_k,\tilde x_k\}^{\perp}$ yields $|\lambda_1|^k \|z^{\perp}\|\le \|R_k\|$. Using the bound $|\lambda_1|^k\ge c_0/2$ gives
$$
\|z^{\perp}\|\le \frac{2}{c_0}\|R_k\|.
$$
Projecting onto $\hbox{span}(\tilde y_k,\tilde x_k)$ yields
$$
\| \lambda_1^k z_0 -\sigma \tilde x_k \tilde{y}_k^* z_0 \|\le \|R_k\|.
$$
This entails
$$
\NRM{ \frac{z_0}{\|z_0\|} -c\tilde x_k  } \le \frac{1}{|\lambda_1|^k \|z_0\|}\|R_k\|\le \frac{2\|R_k\|}{c_0\|z_0\|},
$$
where $c$ is some scalar and we have used $|\lambda_1|^k\ge c_0/2$. We now use the following general inequality
\begin{equation}\label{eq:distS1}
 \hbox{if $\|u\| = \|v  \| = 1$ and for some $t \geq 0 $, $\| u - t v\| \leq \veps$ then $\|u - v \| \leq 2 \veps$.}
\end{equation}
We deduce that
$$
\NRM{
\frac{z_0}{\|z_0\|} -\tilde x_k  } \le \frac{4\|R_k\|}{c_0\|z_0\|}.
$$
From the trianle inequality, $\|z_0\|\ge 1-\|z^{\perp}\|$. We then have
$$
\|z-\tilde x_k\|\le \|z^{\perp}\|+\NRM{z_0-\|z_0\|\tilde x_k}+1-\|z_0\|\le \frac{8\|R_k\|}{c_0}.
$$
Finally, since $\lambda = \lambda^k _1$ has multiplicity one in $A^k$, the eigenspace of $\lambda_1$ for $A$ coincides with the eigenspace of $\lambda_1^k$ for $A^k$. This  concludes the proof of Proposition \ref{prop:sing2eig}. \end{proof}

We now provide an extension of Proposition \ref{prop:sing2eig} to arbitrary rank tailored to our future needs. For $x = ( x_1, \ldots, x_n) \in \dC^n$, the multiplicity of $z \in \dC$ in $x$ is defined as $\sum_i \IND ( x_i = z)$, the number of coordinates of $x$ equal to $z$.

\begin{proposition}\label{prop:sing2eig2}
Let $A \in M_n ( \dR)$,  $\ell' < \ell < 2 \ell'$ be mutually prime odd integers, $\theta = (\theta_1, \ldots, \theta_r) \in (\dR\setminus\{0\})^r $  such that for any $ k \in \{ \ell, \ell' \}$, for some  vectors $x_{k,1},y_{k,1}, \ldots , x_{k,r},y_{k,r} \in \dR^n$ and some matrix $R_k \in M_n(\dR)$,
$$
A^k = \sum_{j=1}^r  \theta_j^k x_{k,j} y_{k,j}^* + R_k.
$$
Assume there exist $c_0, c_1 > 0$ such that for all $i \ne j \in [r]$, $\langle y_{k,j} , x_{k,j} \rangle \geq c_0$, $\|x_{k,j} \| \|y_{k,j} \| \leq c_1$, $\langle x_{k,j} , y_{k,i} \rangle = \langle x_{k,j} , x_{k,i} \rangle = \langle y_{k,j} , y_{k,i} \rangle = 0$ and 
$$\| R_k \| <  \PAR{ \frac{c_0 (c_0 \gamma^k -c_1)_+}{4 c_1} \wedge \frac{c_0^2 }{2(\ell\vee \ell') c_1 } } \vartheta^k,$$
where  $\vartheta = \min_{i} |\theta_i| $, $\gamma = \min \{ \theta_i / \theta_j : \theta_i > \theta_j >0 \hbox{ or } \theta_i < \theta_j < 0 \}$ (the minimum over the empty set being $+\infty$). Let $(\lambda_i)$, $1 \leq i \leq n$, be the eigenvalues of $A$ with $|\lambda_n| \leq \ldots \leq | \lambda_1|$. Then, there exists a permutation $\sigma \in S_r$ such that for all $i\in[r]$,
$$
| \lambda_i - \theta_{\sigma(i)} | \leq  \frac{C | \theta_{\sigma(i)}| }{ \ell} \quad \hbox{ and, for $i \geq r+1$,} \quad \ABS{ \lambda_i}\leq \PAR{ \frac{2 c_1}{c_0} } ^{1 /\ell} \| R_\ell \|^{1/\ell},
$$
with $C=\pi/2+2\sqrt{c_1\vee 1}\log(2(c_1\vee c_0^{-1}))$.  Moreover, if   $\theta_{\sigma(i)}$ has multiplicity one in $\theta$, $\lambda_i$ is a simple eigenvalue and there exists a unit eigenvector $\psi_i$ of $A$ with eigenvalue $\lambda_{i}$ such that 
$$
\NRM{ \psi_i - \frac{x_{\ell,\si(i)}}{\|x_{\ell,\si(i)}\|} } \leq  C'  \|R_\ell \| |\vartheta |^{-\ell} .
$$
with $C' = 24 c_1 c_0^{-1}  / (1 \wedge (c_0 \gamma^k -c_1)_+ \wedge c_0)$. 
\end{proposition}

\begin{proof} We may assume that $\vartheta = 1$.  Fix $k \in \{ \ell, \ell'\}$ and let $ \tilde x_j = x_{k,j} / \|x_{k,j}\|$, $\tilde y_j = y_{k,j} / \| y_{k,j} \|$, $\sigma_j = \theta_j^k    \|x_{k,j} \| \|y_{k,j}\|$, $\nu_j = \theta_j^k \langle y_{k,j} , x_{k,j} \rangle = \si_j  \langle \tilde y_{k,j} , \tilde x_{k,j} \rangle$.  Let  $H_j = \mathrm{span} (x_{k,j}, y_{k,j})$. By assumption, the vector spaces $H_j$, $1 \leq j \leq n$, are orthogonal. For ease of notation, let us assume for all $j \in [r]$, $H_j$ has dimension $2$ (the case where $x_{k,j}$ and $y_{k,j}$ are colinear is identical). We consider an orthonormal basis $(f_1, \cdots, f_n)$ of $\dC^n$, such that $\mathrm{span} (f_{2j-1}, f_{2j}) = H_j$, $f_{2j-1} = \tilde y_j$. We have 
$$A^{k} =  \sum_{j=1}^r \sigma_j \tilde x_j \tilde y_j^* + R_k = U D U ^{-1}  + R_k,$$
where $D = \diag( \nu_1, 0, \nu_2, 0, \ldots, \nu_r, 0, \ldots , 0)$, $U = ( \tilde x_1, f_2, \tilde x_2, f_4, \ldots, \tilde x_r, f_{2r} , \ldots, f_n)$ (provided that $U$ is indeed invertible).  Arguing as in the proof of Proposition \ref{prop:sing2eig}, denote by $V$ the unitary matrix $V = (f_1, \ldots, f_n)$ and  decompose $\tilde x_j$ as $\tilde x_j = a_j f_{2j-1} + b_j f_{2j}$. Then
$
V^* U $ has a block diagonal structure with blocks $W_j$, $1 \leq j \leq r$, and $I_{n - 2 r}$, where  
$$W_j =  \begin{pmatrix} 
a_j  &  0 \\
 b_j &   1 \\
\end{pmatrix}.$$
We find, as in Proposition \ref{prop:sing2eig}, 
$$
\kappa (U) = \max_j \kappa (W_j) \leq  \kappa = 2 c_1 c_0^{-1}.
$$
Now, by assumption, $2 \kappa \|R_k \| < c_0 \wedge (c_0\gamma^k - c_1)_+$ is less than the minimal distance between the distinct eigenvalues of $D$. We deduce from Theorem~\ref{bauer-fike}  applied to $D + U^{-1} R_k U$ that there is a permutation $s \in S_r$ such that
\begin{equation}\label{eq:lambdakA2}
| \lambda_i^k  - \nu_{s(i)} | \leq \kappa \| R_k \|  \quad \hbox{ and, for $i \geq r+1$,} \quad \ABS{ \lambda _i}^k\leq \kappa \| R_k \| .
\end{equation}

Importantly, we claim that the permutation $s = s_k \in S_r$ is such that for $1 \leq i \leq r$, $\theta_{s_{\ell}(i)} = \theta_{s_{\ell'} (i)}$. Indeed, we first observe that the assumptions $\gamma^k c_0 > c_1$ and $k$ odd imply that $\nu_i = \nu_j$ is equivalent to $\theta_i = \theta_j$. Assume first for simplicity that all $\theta_i$ are positive and let $m_1$ be the multiplicity in $\theta$ of $t_1 = \max_{i} \theta_i$. Then the $m_1$ eigenvalues such that $|\lambda_i ^{k} - \nu_j | \leq \kappa \|R_k \|$ for some $j$ such that $\theta_j = t_1$ are precisely the $m_1$ largest eigenvalues of $A$. If $m_1 < r$, we may then repeat the same argument for the second largest value of the set $\{ \theta_1, \cdots, \theta_r \}$. By iteration, we deduce the claimed statement when all $\theta_i$ have the same sign. In the general case, we notice that if $| \lambda_i^{\ell'}  - \nu_{j} | \leq \veps = \kappa \| R_{\ell'} \|$ with $\theta_{j} > 0$, then $\nu_{j} > c_0$,  $|\sin ( \lambda_i^{\ell'} ) | \leq \veps / c_0$, $| \arg( \lambda_i^{\ell'} )  | \leq \pi  \veps /(2 c_0)$ and $|\arg(\lambda_i^{\ell})| \leq \pi  \veps \ell  / (2 c_0 \ell') \leq  \pi/2$ (we use here the assumption $\ell < 2 \ell'$). It follows that if $|\lambda_i^{\ell'} - \nu_j| \leq \kappa \| R_{\ell'} \|$ with $\theta_j >0$ then we cannot have   $|\lambda_i^{\ell} - \nu_{j'}| \leq \kappa \| R_{\ell} \|$ with $\theta_{j'} < 0$. We may thus repeat the previous argument by considering the largest eigenvalues of $A$ with positive real part and the largest eigenvalues of $A$ with negative real part separately. 

We now bound the difference between $\lambda_i$ and $\theta_{s(i)}$. By assumption, 
$
c_0 | \theta_j |  \leq | \nu_j |\leq c_1 |\theta_j|.
$
Hence, arguing as in the proof of Proposition \ref{prop:sing2eig}, 
$$
\ABS{  | \lambda_i|  - |\theta_{s(i)}|  } \leq  c_2  |\theta_{s(i)}| / k, 
$$
with $c_2 = \sqrt{c_1 \vee 1} \log ( 2 (c_1\vee c_0^{-1} ))$. We next control the argument $\omega_i\in (-\pi,\pi]$ of $\lambda_i =  | \lambda_i| \sign(\theta_{\si(i)}) e^{ i \omega} $. Arguing as in the proof  of Proposition \ref{prop:sing2eig}, we get for $p \in \mathbb{Z}$,
$
|\omega - 2 p \pi | \leq   \pi |\veps| / k 
$
and we may conclude the proof of the claim of Proposition \ref{prop:sing2eig2} on eigenvalues as in Proposition \ref{prop:sing2eig}.

It now remains to control the eigenvector of $\lambda_i$ such that $\theta_{\si(i)}$ has multiplicity one. First from \eqref{eq:lambdakA2}, $\lambda_i$ is a simple eigenvalue of $A$. Let $z$ be a corresponding normed eigenvector of $A$. Applying $A^k$ yields
\begin{equation}\label{decomp_1}
\lambda_i^k z=\sum_{j\in[r]}\sigma_j (\tilde y_j^*z)\tilde x_j +R_k z.
\end{equation}
Applying $A^k$ once more to (\ref{decomp_1}) yields
$$
\lambda_i^{2k}z=\lambda_i^kR_kz +\sum_{j\in[r]}\sigma_j[\sigma_j (\tilde y_j^* z)(\tilde y_j^* \tilde x_j)+\tilde y_j^* R_kz]\tilde x_j.
$$
Multiplying (\ref{decomp_1}) by $\lambda_i ^k$ and subracting it to the previous display yields
$$
\sum_{j\in[r]}\sigma_j[(\lambda_i^k-\nu_j)\tilde y_j^* z-\tilde y_j^* R_kz]\tilde x_j=0.
$$
Thus for all $j\in[r]$,
\begin{equation}\label{decomp_2}
(\lambda_i^k-\nu_j)\tilde y_j^* z-\tilde y_j^* R_kz=0.
\end{equation}

Now for $j \ne s(i)$, from \eqref{eq:lambdakA2}, we have 
$$\ABS{\lambda_i^k - \nu_j} \geq  \ABS{ \nu_{s(i)}  - \nu_j} - \ABS{\lambda_i^k - \nu_{s(i)}} \geq \frac 1 2  \PAR {  (c_0 \gamma^k -c_1)_+ \wedge c_0}.
$$
It follows that 
$$
 \PAR{ \sum_{j \ne s(i)}|\tilde{y}_j^*z | ^2 }^{1/2} \leq \frac{2 \|R_k \| }{(c_0 \gamma^k -c_1)_+ \wedge c_0}. 
$$
Moreover, this implies upon dividing~(\ref{decomp_2}) by $\lambda_i^k$:
$$
\PAR{ \sum_{j \ne s(i)}\ABS{\frac{\sigma_j}{\lambda_i^k}\tilde{y}_j^*z }^2 }^{1/2} 
\leq 
\frac{c_1}{c_0}\PAR{ \sum_{j \ne s(i)}\ABS{\frac{\nu_j}{\lambda_i^k}\tilde{y}_j^*z } ^2 }^{1/2} 
\leq 
\frac{c_1}{c_0}\PAR{
\frac{2 \|R_k \| }{(c_0 \gamma^k -c_1)_+ \wedge c_0}+\frac{2}{c_0}\|R_k\|}, 
$$
where we have used the fact that $|\lambda_i^k|\ge c_0/2$. It then follows from (\ref{decomp_1}) that for some constant $c = \si_{s(i)} ( \tilde y_{s(i)} ^* z) / \lambda_{i}^k$,
$$
\|z-c \tilde x_{s(i)}\|=\|\sum_{j \ne s(i)}\frac{\sigma_j}{\lambda_i^k}( \tilde{y}_j^*z )  \tilde x_j +\lambda_i^{-k}R_kz\|\le 
\frac{c_1}{c_0}\PAR{\frac{2 \|R_k \| }{(c_0 \gamma^k -c_1)_+ \wedge c_0}+\frac{4}{c_0}\|R_k\|}.
$$
We then obtain the announced bound on $\|z- \tilde x_{s(i)}\|$ by appealing to \eqref{eq:distS1}.
\end{proof}


We conclude this paragraph with an elementary lemma on Gram-Schmidt orthonormalization process. It will be used to obtain vectors which are exactly orthogonal as in the assumptions of Proposition  \ref{prop:sing2eig2}. 

\begin{lemma}\label{le:GS}
Let $u_1, \cdots, u_k$ be  vectors in $\dC^n$ with unit norms such that $\ABS{ \langle u_i , u_j \rangle } \leq \delta$ for all $i \ne j$. If $\delta < k^{-k}$ then $(u_1, \cdots, u_k)$ are linearly independent and, if $(\bar u_1, \cdots, \bar  u_k)$ is the Gram-Schmidt orthonormalization process of $(u_1, \cdots, u_k)$, we have for all $j \in [k]$,
$$
\| u_j -  \bar  u_j \| \leq \delta j^j.
$$
\end{lemma}
\begin{proof}
We prove the statement by induction. For $k=1$, $\bar u_1 = u_1$. For $k \geq 1$, we denote by $v_{k+1}$ the orthonormal projection of $u_{k+1}$ on the span of $( u_1, \cdots,  u_k)$. We have 
$$
\| v_{k+1} \|^2  =   \sum_{j=1} ^k \ABS{ \langle u_{k+1} , \bar  u_j \rangle}^ 2 .
$$
Now, from the induction hypothesis,
$$
\sum_{j=1} ^k \ABS{ \langle u_{k+1} , \bar  u_j \rangle}^ 2 \leq 2 \sum_{j=1} ^k \PAR{ \ABS{ \langle u_{k+1} ,   u_j \rangle}^ 2 +  \| \bar  u_j -  u_j \|^ 2} \leq 2 k \delta^2 ( 1   +  k^{2k} ).
$$
It is easy to check that $\sqrt{ 2 k  ( 1+ k^k) } \leq 2^{-1} (k+1)^{k+1}$ for all $k \geq 1$. In particular, if $\delta < (k+1)^{-(k+1)}$, $v_{k+1} \ne u_{k+1}$ and  then from \eqref{eq:distS1}, $\| u_{k+1} - \bar u_{k+1} \| \leq 2  \| v_{k+1} \|  \leq \delta (k+1) ^{k+1} $. 
\end{proof}

%% file: NBS_Erdos_Renyi_V2.tex
\section{Erd\H{o}s-R\'enyi graph: proof strategy for Theorem \ref{th:ER}}
\label{sec:arg_er_v2}

In what follows, we consider a sequence $\ell = \ell(n)\sim \kappa \log_{\alpha}n$ for
some $\kappa\in(0,1/6)$ as in Theorem \ref{th:ER}. 

\subsection{Proof of Theorem \ref{th:ER}}

Let
$$
\varphi = \frac{ B^{\ell} \chi }{ \| B^{\ell} \chi \|}\,  , \quad   \theta = \| B^{\ell} \check \varphi  \|,
$$
and
$$
\zeta=   \frac{ B^{\ell} \check \varphi }{ \theta} =  \frac{ B^{\ell} B^{*\ell} \chi  }{ \| B^{\ell}B^{*\ell} \chi \|} , 
$$
(if $\theta = 0$, we set $\zeta= 0$).  The proof relies on the following two propositions.

\begin{proposition} \label{prop:Bellpsi} For some $c_1, c_0>0$, \whp
\begin{equation*}
\langle \zeta, \check \varphi \rangle  \geq  c_0 \quad  \hbox{ and } \quad c_0 \alpha^{\ell} \leq \theta  \leq c_1 \alpha^{\ell}.  
\end{equation*}
\end{proposition}

\begin{proposition} \label{prop:Bellx}For some $c >0$, \whp
\begin{equation*}
\sup_{ x :  \langle x ,  \check \varphi \rangle = 0, \|x\| = 1} \| B^{\ell} x \| \leq (\log n)^c \alpha^{\ell/2}. 
\end{equation*}
\end{proposition}

Let us check that the last two propositions \ref{prop:Bellpsi} and \ref{prop:Bellx} imply Theorem~\ref{th:ER} . Let $R = B^{\ell} - \theta \zeta \check \varphi^{*}$ and $y \in \dR^{\vec E}$ with $\| y \| = 1$. We write $ y = s \check \varphi + x$ with $x \in \check \varphi^{\perp}$ and $s \in \dR$. We find
$$
\|R  y \| =  \| B^{\ell} x + s(  B^{\ell}   \check \varphi - \theta \zeta)   \| \leq    \sup_{ x :  \langle x ,  \check \varphi \rangle = 0, \|x\| = 1} \| B^{\ell} x \| .
$$
Hence, Proposition \ref{prop:Bellx} implies that \whp
\begin{equation}\label{eq:boundnormC}
\| R \| \leq (\log n)^ c \alpha^{\ell /2}.
\end{equation}

We may now apply Proposition \ref{prop:sing2eig}. If $\lambda_i = \lambda_i (B)$, we find that \whp
$$
| \lambda_1  - \alpha | = O ( 1 / \ell ), \quad \quad |\lambda_2 | \leq   \PAR{C (\log n)^c \alpha^ {\ell/2} }^{1/\ell} = \sqrt \alpha   + O \PAR{ \frac{\log \log n }{ \log n}},
$$
and the normalized Perron eigenvector $\xi$ of $B$ satisfies \whp 
$$
\| \xi - \zeta \|   = O ( (\log n)^c \alpha^ {- \ell/2}). 
$$
This concludes the proof of Theorem \ref{th:ER}. 

\begin{remark}
\label{re:wRama}
We note that from \cite[Theorem 3.3.16]{MR1091716},  we  get that in \eqref{eq:BkSVD2}, 
$$
| s_{1,\ell} -  \theta | \leq  \| R \| \quad \hbox{ and } \quad s_{2,\ell} \leq \|R\|. 
$$
Hence, from (\ref{eq:boundnormC}) \whp 
$$
s_{1,\ell}=  O (  \alpha^\ell)  \quad \hbox{ and } \quad s_{2,\ell}=  O (  (\log n )^ c \alpha^{\ell/2})
$$
On the other hand, \eqref{eq:weakAB} implies that  the above upper bound on $s_{2,\ell}$ is also a lower bound up to the logarithmic factors, more precisely, \whp, $s_{2,\ell} \geq c_0 \alpha^{\ell /2}$ for some $c_0 >0$ (it follows from the proof of the forthcoming Theorem \ref{th:locmart}). Therefore, the naive lower bound on $s_{2,\ell}$ in \eqref{eq:weakAB} is asymptotically tight and Propositions \ref{prop:Bellpsi}-\ref{prop:Bellx} may be interpreted as a weak Ramanujan property for Erd\H{o}s-R\'enyi random graphs. 
\end{remark}

Proposition \ref{prop:Bellpsi} will follow from a {\em local analysis}. Namely the statistics of node neighborhoods up to distance $\ell$ in the original random graph will be related by coupling to a Galton-Watson branching process; relevant properties of the corresponding Galton-Watson process will be established; finally we shall deduce weak laws of large numbers for the $\ell$-neighborhoods of the random graph from the estimations performed on the branching process combined with some asymptotic decorrelation property between distinct node neighborhoods. This is done in Section~\ref{sec:local} where Proposition \ref{prop:Bellphi2}, which contains Proposition \ref{prop:Bellpsi}, is proven. 

The proof of Proposition \ref{prop:Bellx} relies crucially on a matrix expansion given in Proposition~\ref{prop:expansion}, which extends the argument introduced in \cite{LM13} for matrices counting self-avoiding walks to the present setup where non-backtracking walks instead are considered. We now introduce some notation to state it.


\subsection{Matrix expansion for $ B^{\ell} $}
\label{subsec:decomp}
For convenience we extend matrix $B$ and vector $\chi$ to $\dR^{\vec E(V)}$ where $\vec E(V) = \{ ( u ,v ) : u \ne v   \in V\}$ is the set of directed edges of the {\em complete graph}. We set for all $e , f\in \dR^{\vec E(V)}$, $\chi (e) = 1$ and
$$
B_{ef} =A_e A_f \IND (e_2 = f_1) \IND (e_1 \ne f_2 ),
$$
where $A$ is the graph's adjacency matrix.
For integer $k \geq 1$, $e,f \in \vec E(V)$, we define $\Gamma^ k_{ ef}$ as the set of non-backtracking walks $ \gamma = (\gamma_0, \ldots, \gamma_{k} )$ of  length  $k $ starting from $(\gamma_0,\gamma_1) = e$ and ending at $(\gamma_{k-1},\gamma_k) = f$ in the complete graph on the vertex set $V$. We have that
$$
(B^{k} )_{e f} = \sum_{\gamma \in \Gamma^{k+1} _{e f}} \prod_{s=0}^{k} A_{\gamma_{s} \gamma_{s+1}}.
$$

We associate to each walk $\gamma = (\gamma_0,\ldots, \gamma_k)$, a graph $G(\gamma) = ( V(\gamma), E(\gamma) ) $ with vertex set $V(\gamma) = \{ \gamma_{i}, 0 \leq i \leq k \}$ and edge set $E (\gamma)$ the set of distinct visited edges $ \{ \gamma_{i} , \gamma_{i+1}\}, 0 \leq i \leq k-1$. Following \cite{MNS}, we say that a graph $H$ is {\em tangle-free} 
(or $\ell$-tangle free to make the dependence in $\ell$ explicit) 
if every neighborhood of radius $\ell$ in $H$ contains at most one cycle. Otherwise, $H$ is said to be tangled. We say that $\gamma$ is tangle-free or tangled if $G(\gamma)$ is. Obviously, if $G$ is tangle-free and $1 \leq k \leq \ell$ then 
$
B^{k}   = B^{(k)},
$
where 
$$
B^{(k)}_{ef} = \sum_{\gamma \in F^{k+1} _{e f}} \prod_{s=0}^{k} A_{\gamma_{s} \gamma_{s+1}},
$$
and $F^{k+1} _{e f}$ is the subset of tangle-free paths in $\Gamma^{k+1}_{ef}$. For $u \ne v$, we set 
$$
\underline A_{u v}  = A_{uv} - \frac \alpha n.
$$

We define similarly the matrix $\Delta^{(k)}$ on $\dR^{\vec E(V)}$ 
$$
 \Delta^{(k)}_{ef} = \sum_{\gamma \in F^{k+1} _{e f}}   \prod_{s=0}^{k}  \underline A_{\gamma_{s} \gamma_{s+1}}.
$$
The matrix $\Delta^{(k)}$ can be thought of as an attempt to center the non-backtracking matrix $B^k$ when the underlying graph is tangle-free.   We use the convention that a product over an empty set is equal to $1$. 
We also set 
\begin{equation}\label{eq:defDeltaB0}
\Delta^{(0)}_{ef} = \IND ( e = f )\underline A_{e} \quad \hbox{ and } \quad B^{(0)}_{ef} = \IND ( e= f)A_e.
\end{equation}
Notably, $B^{(0)}$ is the projection on $\vec E$. We have the following telescopic sum decomposition. 
\begin{eqnarray}\label{eq:telescope}
B^{(\ell)} _{e f} &  =  & \Delta^{(\ell)}_{ef}  +\sum_{t = 0}^\ell \sum_{\gamma \in F^{\ell+1} _{e f}}     \prod_{s=0}^{t-1} \underline A_{\gamma_{s} \gamma_{s+1}} \PAR{ \frac \alpha n  } \prod_{s=t+1}^\ell A_{\gamma_{s} \gamma_{s+1}}.   
\end{eqnarray}
Indeed, 
$$
\prod_{s=0}^\ell x_s = \prod_{s=0}^\ell y_s  + \sum_{t=0}^{\ell}\prod_{s=0}^{t-1} y_s  ( x_t - y_t) \prod_{s=t+1}^{\ell} x_s.
$$
We denote by $K$ the non-backtracking matrix of the complete graph on $V$. For $0 \leq t \leq \ell$, we define $R^{(\ell)}_t$ via
$$
(R_t ^{(\ell)} )_{ef}   =  \sum_{\gamma \in F^{\ell+1} _{t,e f}} \prod_{s=0}^{t-1} \underline A_{\gamma_{s} \gamma_{s+1}} \prod_{s=t+1}^\ell A_{\gamma_{s} \gamma_{s+1}},
$$
where for $1 \leq t \leq \ell-1$, $F^{\ell+1}_{t, ef} \subset \Gamma^{\ell+1}_{ef} $ is the set of non-backtracking tangled paths $\gamma = (\gamma_0, \ldots, \gamma_{\ell +1}) =  (\gamma',\gamma'') \in  \Gamma^{\ell+1}_{ef}$ with $\gamma' = (\gamma_0, \ldots, \gamma_t)  \in F^{t}_{eg}$, $\gamma'' = ( \gamma_{t+1}, \ldots , \gamma_{\ell +1}) \in F^{\ell- t}_{g' f}$ for some $g,g' \in \vec E (V)$.  For $t = 0$, $F_{0,ef}^{\ell+1}$ is the set of non-backtracking tangled paths $\gamma = (\gamma',\gamma'')$ with $\gamma' = e_1$, $\gamma'' = ( \gamma_{1}, \ldots , \gamma_{\ell +1}) \in F^{\ell}_{g' f}$ for some $g' \in \vec E (V)$ (necessarily $g'_1 = e_2$). Similarly, for $t = \ell$, $F_{\ell,ef}^{\ell+1}$ is the set of non-backtracking tangled paths $\gamma = (\gamma_0, \ldots, \gamma_{\ell +1}) =  (\gamma',\gamma'')$ with $\gamma'' = f_2$, $\gamma' = ( \gamma_{0}, \ldots , \gamma_{\ell}) \in F^{\ell}_{e g}$ for some $g \in \vec E(V)$ (necessarily $g_2 = f_1$). 

We define
$$
L = K^2- \chi \chi^*,
$$
($L$ is nearly the orthogonal projection of $K^2$ on $\chi^\perp$).  We further denote for $1 \leq t \leq \ell-1$
$$
S_t^{(\ell)} = \Delta^{(t-1)} L B^{(\ell - t -1)}.
$$
We then have 
\begin{proposition}\label{prop:expansion}
With the above notations matrix $B^{(\ell)}$ admits the following expansion
 \begin{equation}\label{eq:decompBk}
B^{(\ell)}  =  \Delta^{(\ell)}   +  \frac \alpha n K B^{(\ell-1)}  +  \frac \alpha n  \sum_{t = 1} ^{\ell-1}  \Delta^{(t-1)} K^2  B^{(\ell - t -1)}  +  \frac \alpha n \Delta^{(\ell-1)} K  -   \frac \alpha n    \sum_{t = 0}^{\ell} R^{(\ell)}_t.
\end{equation}
If $G$ is tangle-free, for any normed vector $x\in\dC^{\vec E(V)}$, one has
\begin{eqnarray}
\| B^{\ell} x \| &   \leq &  \|  \Delta^{(\ell)} \|   +    \frac{ \alpha}{  n}  \| K B^{(\ell-1)} \|    +    \frac \alpha  n   \sum_{t = 1} ^{\ell-1}    \| \Delta^{(t-1)} \chi  \| \ABS{ \langle \chi , B^{(\ell-t-1)} x \rangle }  \nonumber \\
& & \quad +  \;  \frac{\alpha}{n}    \sum_{t = 1} ^{\ell-1} \| S_t^{(\ell)} \|   +  \alpha \| \Delta^{(\ell-1)}  \|   +  \frac \alpha n   \sum_{t = 0}^\ell \| R^{(\ell)}_t \|.\label{eq:decompBkx}
\end{eqnarray}
\end{proposition}
\begin{proof}
Equation~\eqref{eq:decompBk} readily follows by adding and subtracting $\frac \alpha n R^{(\ell)}_t$ to the $t$-th term of the summation in~\eqref{eq:telescope} and noticing that this term plus $\frac \alpha n R^{(\ell)}_t$factorizes into a matrix product.

Inequality~\eqref{eq:decompBkx} follows from  ~\eqref{eq:decompBk} by noting that $B^{\ell}=B^{(\ell)}$ as $G$ is tangle-free, decomposing $K^2$ into $L+\chi \chi^*$, and finally using the fact that $\|K\|\le n$.
\end{proof}

\subsection{Norm bounds}
The following proposition will be established in Section~\ref{sec:path_counts} using path counting combinatorial arguments.
\begin{proposition}\label{prop:norm_bounds}
Let $\ell\sim \kappa \log_\alpha n$ with $\kappa \in (0,1/6)$.
With high probability, the following norm bounds hold for all $k$, $0 \leq k \leq \ell$:
\begin{eqnarray}
\label{prop:normDelta}
\| \Delta^{(k)} \| \leq ( \log n) ^{10} \alpha^{k /2},
\\
\label{prop:Dscalar}
\|  \Delta^{(k)} \chi  \|    \leq (\log n) ^{5}  \alpha^{k / 2} \sqrt n,
\\
\label{prop:normR}
\| R^{(\ell)}_k \|    \leq (\log n) ^{25}  \alpha^{\ell  - k /2},\\
\label{prop:normB2}
 \| B^{(k)} \|    \leq (\log n) ^{10}  \alpha^{k} \quad \hbox{ and } \quad 
 \|  K B^{(k)} \|    \leq \sqrt n (\log n) ^{10}  \alpha^{k},
\end{eqnarray}
and the following bound holds for all $k$, $1 \leq k \leq \ell-1$:
\begin{equation}
\label{prop:normS}
\| S^{(\ell)}_k \|    \leq \sqrt n (\log n) ^{20}  \alpha^{\ell  - k /2}.
\end{equation}
\end{proposition}



\subsection{Proof of Proposition \ref{prop:Bellx}}
\label{subsec:proofBellx}
Together with Propositions \ref{prop:expansion}  and~\ref{prop:norm_bounds}, we shall also need the next two results, established by local analysis in Section \ref{sec:local}. In particular the forthcoming Lemma \ref{le:tls2} implies that
\begin{lemma} \label{le:tls}
For $\ell \sim  \kappa \log_\alpha n $ with $\kappa < 1/2$, \whp  the random graph $G$ is $\ell$-tangle-free. 
\end{lemma}
For the Erd\H{o}s-R\'enyi graph, Corollary \ref{cor:Blperp} states the following.

\begin{proposition}\label{cor:bls2}
For $\ell \sim  \kappa \log_\alpha n $ with $\kappa < 1/2$, \whp, for any $0 \leq t \leq \ell -1$, it holds that 
$$
\sup_{\| x \| = 1, \langle B^{\ell}  \chi , x \rangle = 0 }  \ABS{ \langle B^{t}  \chi ,   x \rangle }\leq    (\log n)^5 n ^{1/2} \alpha^{t/2}.
$$
\end{proposition}
We now have all the ingredients necessary to prove Proposition~\ref{prop:Bellx}. In view of Lemma \ref{le:tls}, we may use the bound \eqref{eq:decompBkx} of Proposition~\ref{prop:expansion} and take the supremum over of all $x$, $\| x \| =1$,  $\langle \check \varphi , x \rangle = \langle \chi, B^\ell x \rangle = 0$. 
By the norm bounds (\ref{prop:normDelta})-(\ref{prop:normR})-(\ref{prop:normS}) of Proposition~\ref{prop:norm_bounds}, \whp
\begin{eqnarray*}
\alpha \| \Delta^{(\ell-1)}  \|  +  \|  \Delta^{(\ell)} \|  +  \frac \alpha n   \sum_{t = 0}^\ell \| R^{(\ell)}_t \| + \frac \alpha n   \sum_{t = 1}^{\ell-1} \| S^{(\ell)}_t \| \leq  C  (\log n )^{c} \alpha^{\ell/2} ( 1 + \alpha^{\ell/2} / \sqrt  n) =  O (( \log n)^c   \alpha^{\ell/2} ). 
\end{eqnarray*}
Also, from \eqref{eq:check}, since $\check \chi = \chi$,
$$
\sup_{\| x \| = 1, \langle  \chi ,  B^{\ell}  x \rangle = 0 }  \ABS{ \langle  \chi ,    B^{(t)}  x \rangle } = \sup_{\| x \| = 1, \langle  \chi ,  B^{\ell}  \check x \rangle = 0 }  \ABS{ \langle  \chi ,    B^{(t)}  \check x \rangle } = \sup_{\| x \| = 1, \langle B^{\ell}  \chi , x \rangle = 0 }  \ABS{ \langle B^{(t)}  \chi ,   x \rangle }.
$$
Hence, from Proposition \ref{cor:bls2} and norm bound (\ref{prop:Dscalar}), \whp
\begin{eqnarray*}
 \|  \Delta^{(t-1)} \chi \| | \langle  \chi ,  B^{(\ell - t -1)} x \rangle | & \leq & C (\log n )^c    n    \alpha ^{\ell /2}  
\end{eqnarray*}
Hence, \whp
$$
\frac{\alpha}{n} \sum_{t=1}^{\ell-1}  \| \Delta^{(t-1)} \| | \langle  \chi ,  B^{(\ell - t -1)} x \rangle | = O \PAR{ (\log n )^{c+1}   \alpha^{\ell/2} }. 
$$
It remains to use norm bound ((\ref{prop:normB2}) to deal with the term $\|  K B^{(\ell-1)} \| / n $  in \eqref{eq:decompBkx} to conclude the proof of Proposition~\ref{prop:Bellx}. 

\section{Proof of Proposition~\ref{prop:norm_bounds}: path count combinatorics}
\label{sec:path_counts} 
In this section, we use the method of moments to prove the norm upper bounds stated in Proposition~\ref{prop:norm_bounds}. Recall that $\ell \sim \kappa \log_\alpha n$ with $\kappa >0$. All our constants will depend implicitly on $\kappa$.  
\subsection{Bound~(\ref{prop:normDelta}) on $\| \Delta^{(k)}\|$}
The proof will use a version of the trace method. For $n \geq 3$, we set 
\begin{equation}\label{eq:choicem}
m = \left\lfloor  \frac{ \log n }{ 13 \log (\log n)} \right\rfloor.
\end{equation}
The symmetry \eqref{eq:check} implies that $\Delta^{(k)}_{ef} = \Delta^{(k)}_{f^{-1} e^{-1}}$. 
With the convention that $e_{2m + 1} = e_1$, we get 
\begin{eqnarray}
\| \Delta^{(k-1)}  \| ^{2 m} = \| \Delta^{(k-1)} { \Delta^{(k-1)} }^*  \| ^{m} & \leq & \tr \BRA{ \PAR{  \Delta^{(k-1)} { \Delta^{(k-1)} }^*}^{m}  } \nonumber\\
& = & \sum_{e_1, \ldots, e_{2m}}\prod_{i=1}^{m}  (\Delta^{(k-1)} ) _{e_{2i-1} , e_{2 i}}(\Delta^{(k-1)} ) _{e_{2i+1} , e_{2 i}} \nonumber \\
& = & \sum_{e_1, \ldots, e_{2m}}\prod_{i=1}^{m}  (\Delta^{(k-1)} ) _{e_{2i-1} , e_{2 i}}(\Delta^{(k-1)} ) _{ e^{-1}_{2 i} e^{-1}_{2i+1}} \nonumber \\
& =  &  \sum_{\gamma \in W_{k,m} }   \prod_{i=1}^{2m}  \prod_{s=1}^{k} \underline  A_{\gamma_{i,s-1}  \gamma_{i,s}} \label{eq:trDeltak}
\end{eqnarray}
where $W_{k,m}$ is the set of sequence of paths $\gamma = ( \gamma_1, \ldots, \gamma_{2m})$ such that $\gamma_i = (\gamma_{i,0}, \cdots, \gamma_{i,k})\in V^{k+1}$ is non-backtracking tangle-free of length $k$ and for all $i = 1, \ldots, 2m$, 
$$
(\gamma_{i, k-1},\gamma_{i, k}) = (\gamma_{i+1,1}, \gamma_{i+1,0} ),
$$
with the convention that $\gamma_0 = \gamma_{2m}$.

We take expectations in \eqref{eq:trDeltak} and use independence of the edges $A_{xy}$ together with $\dE \underline A_{xy}  =0$. We find 
\begin{eqnarray}
\dE \| \Delta^{(k-1)}  \| ^{2 m} & \leq  &  \sum_{\gamma \in W'_{k,m} }  \dE  \prod_{i=1}^{2m}  \prod_{s=1}^{k} \underline  A_{\gamma_{i,s-1}  \gamma_{i,s}} \label{eq:trDeltak2}
\end{eqnarray}
 where $W'_{k,m}$ is the subset of $W_{k,m}$ where each non-oriented edge is visited at least twice.  For each $\gamma \in W_{k,m}$ we associate the graph $G(\gamma) = ( V(\gamma), E(\gamma) ) $ of visited vertices and edges.  We set 
$$
v(\gamma) = |V(\gamma)|\;  \hbox{ and }  \;  e(\gamma) =  |E (\gamma) |. 
$$

We say that a path $\gamma$ is canonical if $V(\gamma) = \{ 1, \cdots, v(\gamma) \}$ and the vertices are first visited in order. $\cW_{k,m}$ will denote the set of canonical paths in $W_{k,m}$. Every canonical path is isomorphic to ${n \choose v(\gamma)}v(\gamma)!$ paths in $W_{k,m}$. We also have the following

\begin{lemma}[Enumeration of canonical paths]\label{le:enumpath}
Let $\cW_{k,m} (v,e) $ be the set of canonical paths with $v(\gamma) = v$ and $e(\gamma) = e$. We have 
$$
| \cW _{k,m} (v,e) | \leq  k^{2m} (2km)^{ 6 m ( e -v +1)}.
$$
\end{lemma}

\begin{proof}
In order to upper bound $| \cW_{k,m} ( v, e) | $ we need to find an injective way to encode the canonical paths $x \in \cW_{k,m} ( v, e) $.

Let $x = (x_{i,t})_{1 \leq i \leq 2 m, 0 \leq t \leq k}  \in \cW _{k,m} (v,e) $. We set $y_{i,t} = \{ x_{i,t},x_{i,t+1}\}$, $y_{i,t}$ will be called an edge of $x$. We explore the sequence $(x_{i,t})$ in lexicographic order denoted by $\preceq$ (that is $(i,t)\preceq (i+1,t')$ and $(i,t)\preceq(i,t+1)$). We think of the index $(i,t)$ as a time.  For $0 \leq t \leq k-1$, we say that $(i ,t)$ is a {\em first time}, if  $x_{i,t +1}$ has not been seen before (that is $x_{i,t+1} \ne x_{i', t'}$ for all $(i',t') \preceq (i,t)$).  If $(i,t)$ is a first time the edge $y_{i,t}$ is called a {\em tree edge}.  By construction, the set of tree edges forms a tree $T$ with vertex set $\{1, \ldots , v\}$. The edges which are not in $T$ are called the {\em excess edges} of $x$. Any vertex different from $1$ has its associated tree edge. It follows that the cardinal of excess edges is $\epsilon = e - v +1$.

We build a first encoding of $\cW_{k,m}(v,e)$. If $(i,t)$ is not a first time, we say that $(i,t)$ is an {\em important time} and we mark the time $(i,t)$  by the vector $(x_{i,t+1},x_{i,\tau})$, where $(i,\tau)$ is the next time that $y_{i,\tau}$ will not be a tree edge (by convention $\tau = k$ if $x_{i,s}$ remains on the tree for all $t+1 \leq s \leq k$). Since there is a unique non-backtracking path between two vertices of a tree, we can reconstruct $x \in \cW_{k,m}$ from the position of the important times and their mark. It gives rise to our first encoding.

The main issue with this encoding is that the number of important
times could be large. We have however not used so far the hypothesis
that each path $x_i$ is tangle free. To this end, we are going to
partition important times into three categories, {\em short cycling},
{\em long cycling} and {\em superfluous} times. First consider the
case where the $i$-th path $x_i$ contains a cycle. For each $i$, the
first time $(i,t)$ such that $x_{i,t+1} \in \{ x_{i,0}, \ldots,
x_{i,t} \}$ is called a short cycling time.  Let $0 \leq \sigma \leq
t$ be such that $x_{i,t+1} = x_{i,\sigma}$. By the  assumption of
tangle-freeness, $C := (x_{i,\sigma},\cdots, x_{i,t+1})$ is the only
cycle visited by $x_i$. We denote by $(i,\tau)$ the first time after
$(i,t)$ that $y_{i,\tau}$ in not an edge of $C$ (by convention $\tau =
k$ if $x_i$ remains on $C$). We  add the extra mark $\tau$ to the
short cycling time. Important times $(i,t)$ with $1 \leq t < \sigma$
or $\tau < t \leq k$ are called long cycling times. The other
important times are called superfluous. The key observation  is that
for each $1 \leq i \leq 2 m$, the number of long cycling times $(i,t)$
is bounded by  $\epsilon-1$ (since there is at most one cycle, no edge
of $x$ can be seen twice outside those of $C$, the $-1$ coming from
the fact that  the short cycling time is an excess edge). 
Now consider the case where the $i$-th path does not contain a cycle,
then all important times are called long cycling times and their
number is bounded by $\epsilon$.

We now have our second encoding. We can reconstruct $x$ from the
positions of the long cycling  and the short cycling times and their
marks. For each $1 \leq i \leq 2m$, there are at most $1$ short
cycling time and $ \epsilon-1$ long cycling times within $x_i$ if
$x_i$ contains a cycle and $0$ short cycling time and $\epsilon$ long
cycling times if $x_i$ does not contain a cycle. 
There are at most $k ^{2m \epsilon}$ ways to position them (in time). There are at most $v^2$ different possible marks for a long cycling time and $v^2 k$ possible marks for a short cycling time. We deduce that    
$$
| \cW _{k,m} (v,e) | \leq    k ^{2 m \epsilon} (v^2 k )^{2m}   (v^2 ) ^{2m (\epsilon-1)}. 
$$
We use $v \leq 2 k m$ to obtain the announced bound.
\end{proof}

\begin{proof}[Proof of Proposition~\ref{prop:norm_bounds}, norm bound (\ref{prop:normDelta})]
From \eqref{eq:trDeltak2} and Markov inequality, it suffices to prove that 
\begin{equation}\label{eq:boundS}
S = \sum_{\gamma \in W'_{k,m} }  \dE   \prod_{i=1}^{2m}  \prod_{s=1}^{k} \underline  A_{\gamma_{i,s-1}  \gamma_{i,s}} \leq (C \log n )^{16m} \alpha^{km}. 
\end{equation}
Observe that if $\gamma \in W'_{k,m}$, $v(\gamma) - 1 \leq e(\gamma) \leq k m $ and $ v(\gamma) \geq 3$. As ${n \choose v(\gamma)}v(\gamma)!< n^{v(\gamma)}$,  any $\gamma \in \cW_{k,m}$ is isomorphic to less that $n^{v(\gamma)}$ elements in $W_{k,m}$.  Also,  from the independence of the edges and $\dE \underline A_{xy} ^p \leq \alpha / n$ for integer $p \geq 2$, we get that 
\begin{equation}\label{eq:expAunder}
 \dE   \prod_{i=1}^{2m}  \prod_{s=1}^{k} \underline  A_{\gamma_{i,s-1}  \gamma_{i,s}} \leq \PAR{ \frac{ \alpha } n }^{e(\gamma)}. 
\end{equation}
Hence, using Lemma \ref{le:enumpath}, we obtain for $k\leq \ell$,
\begin{eqnarray}
S & \leq & \sum_{v=3}^{k m +1} \sum_{e = v - 1} ^{ km } | \cW _{k,m} (v,e) | \PAR{ \frac{ \alpha}{n} }^e n ^v \nonumber\\
& \leq & n  \alpha^{km}\sum_{v=3}^{k m +1} \sum_{e = v - 1} ^{ km } k^{2m} ( 2k m )^{ 6 m ( e -v +1) } n ^{v - e -1}\nonumber \\
& \leq &  n  \alpha^{km} \ell ^{ 2m } (\ell m) \sum_{s = 0} ^{ \infty } \PAR{ \frac{ (2\ell m)^{6m }}{n}}  ^{ s}\cdot \label{eq:Sboundddd}
 \end{eqnarray}
For our choice of $m$ in \eqref{eq:choicem}, we have, for $n$ large enough,
$$
n^{1/(2m)} =  o ( \log n)^7  \; , \quad   \ell m = o ( \log n)^2  \quad  \hbox{ and }  \quad (2 \ell m) ^{6m} \leq  n^{12/13}. 
$$
In particular, the above geometric series converges  and  \eqref{eq:boundS} follows. 
\end{proof}
\subsection{Bound (\ref{prop:Dscalar}) on $\|\Delta^{(k)} \chi\|$}
\begin{proof}
The bound (\ref{prop:Dscalar}) on $\|\Delta^{(k)} \chi\|$ we will now establish improves by a factor $\sqrt n$ on the trivial estimate  $\| \Delta^{(k)} \chi\| \leq \|\chi \| \|\Delta^{(k)} \|$. 
Its proof parallels the argument used to show (\ref{prop:normDelta}). We have 
\begin{eqnarray*}
\dE  \|  \Delta^{(k-1)} \chi  \| ^2 & = & \dE \sum_{e,f,g}  \Delta^{(k-1)}_{ef} \Delta^{(k-1)}_{eg} \\
& = & \sum_{\gamma \in W''_{k,1} } \dE \prod_{i=1}^2 \prod_{s=1}^k \underline A_{\gamma_{i,s-1}, \gamma_{i,s}}, 
\end{eqnarray*}
where $W''_{k,1}$ is the set of pairs of paths $(\gamma_1, \gamma_2)$ such $\gamma_i = (\gamma_{i,0}, \ldots, \gamma_{i,k})$ is non-backtracking and $(\gamma_{1,k-1}, \gamma_{1,k}) = (\gamma_{2,1}, \gamma_{2,0})$ and each edge is visited at least twice. The only difference with $W_{k,1}$ defined above is that we do not require that $(\gamma_{1,0},\gamma_{1,1}) = (\gamma_{2,k},\gamma_{2,k-1})$. However, this last condition $(\gamma_{1,0},\gamma_{1,1}) = (\gamma_{2m,k},\gamma_{2m,k-1})$  was not used in the proof of Lemma \ref{le:enumpath}. It follows that the set of canonical paths in $W''_{k,1}$ with $v$ distinct vertices and $e$ distinct edges has cardinal bounded by $k^2 (2 k) ^{ 6 ( e -v +1) }$. Since the paths are connected and each edge appears at least twice, we have $v-1 \leq e \leq k$. As in the proof of (\ref{prop:normDelta}), we get from \eqref{eq:Sboundddd} with $m=1$
$$
\dE  \|  \Delta^{(k-1)} \chi  \| ^2 \leq  C  n \alpha^k ( \log n)^3. 
$$ 
We conclude with Markov inequality and the union bound.
\end{proof}
\subsection{Bound (\ref{prop:normR}) on $\|R^{(\ell)}_k\|$}
For $n \geq 3$, we set 
\begin{equation}\label{eq:choicem2}
m = \left\lfloor  \frac{ \log n }{ 25 \log (\log n)} \right\rfloor.
\end{equation}
For $0 \leq k \leq \ell-1$, we have that 
\begin{eqnarray}
\| R^{(\ell-1)}_k  \| ^{2 m} & \leq & \tr \BRA{ \PAR{  R^{(\ell-1)}_k { R^{(\ell-1)}_k }^*}^{m}  } \nonumber\\
& =  &  \sum_{\gamma \in T'_{\ell,m,k} }   \prod_{i=1}^{2m}  \prod_{s=1}^{k} \underline  A_{\gamma_{i,s-1}  \gamma_{i,s}}  \prod_{s=k+2}^{\ell} A_{\gamma_{i,s-1}  \gamma_{i,s}}  \label{eq:trRk}
\end{eqnarray}
where $T'_{\ell,m,k}$ is the set of sequence of paths $\gamma = ( \gamma_1, \ldots, \gamma_{2m})$ such that $\gamma^1_i = (\gamma_{i,0}, \cdots, \gamma_{i,k})$ and $\gamma^2_i = (\gamma_{i,k+1}, \cdots, \gamma_{i,\ell})$ are non-backtracking tangle-free, $\gamma_i = (\gamma^1 _i, \gamma^2_i)$ is non-backtracking tangled and for all odd $i\in\{ 1, \ldots, 2m\}$, 
\begin{equation*}
(\gamma_{i,0}, \gamma_{i,1} ) = (\gamma_{i-1, 0},\gamma_{i-1,1}) \quad \hbox{ and } \quad (\gamma_{i,\ell-1}, \gamma_{i,\ell} ) = (\gamma_{i+1, \ell-1},\gamma_{i+1, \ell}) ,
\end{equation*}
with the convention that $\gamma_0 = \gamma_{2m}$.

We define $G(\gamma)  = (V(\gamma), E(\gamma))$ as the union of the graph $G(\gamma_i^z)$, $1 \leq i \leq 2m$, $z \in \{1,2\}$. Note that the edges $(\gamma_{i,k}, \gamma_{i,k+1})$ are not taken into account in $G(\gamma)$.  As usual, we set $v(\gamma) = |V(\gamma)|$  and $e(\gamma) =  |E (\gamma) | \geq v(\gamma) $. Since $\gamma_i$ is tangled  either (a) $G(\gamma_i)$ contains a cycle and is connected or (b) both $G(\gamma_i^1)$ and $G(\gamma_i^2)$ contain a cycle. In particular, all connected components of $G(\gamma)$ contain a cycle and it follows that 
$$
v(\gamma) \leq e(\gamma). 
$$
 Taking the expectation in \eqref{eq:trRk}, we find that 
\begin{eqnarray}
\dE \| R^{(\ell-1)} _k \| ^{2 m} & \leq  &  \sum_{\gamma \in T_{\ell,m,k} }  \dE  \prod_{i=1}^{2m}  \prod_{s=1}^{k} \underline  A_{\gamma_{i,s-1}  \gamma_{i,s}}  \prod_{s=k+2}^{\ell} A_{\gamma_{i,s-1}  \gamma_{i,s}}  \label{eq:trRk2},
\end{eqnarray}
 where $T_{\ell,m,k}$ is the subset of $\gamma \in T'_{\ell,m,k}$ such that 
\begin{equation}\label{eq:egammaT}
v(\gamma) \leq e(\gamma) \leq k m + 2 m ( \ell - 1- k) = m (2 \ell - 2 - k).
\end{equation}  Indeed, for the contribution of a given $\gamma$ in \eqref{eq:trRk2} to be non-zero, each pair $\{\gamma_{i,s-1} , \gamma_{i,s}\}$, $1 \leq i \leq 2m$, $1 \leq s \leq k$, should appear at least twice in the sequence of the $2 (\ell -1)m$ pairs $\{\gamma_{i,s-1} , \gamma_{i,s}\}$, $s \ne k+1$.

\begin{lemma}[Enumeration of canonical tangled paths]\label{le:enumpath2}
Let $\cT_{\ell,m,k} (v,e) $ be the set of canonical paths in $T_{\ell,m,k}$ with $v(\gamma) = v$ and $e(\gamma) = e$. We have 
$$
| \cT _{\ell,m,k} (v,e) | \leq   (4 \ell m )^{ 12 m ( e - v +1) + 8 m }.
$$
\end{lemma}

\begin{proof}
We will adapt the proof of Lemma \ref{le:enumpath} and use the same terminology. We start by reordering $\gamma \in \cT_{\ell,m,k}$ into a new sequence which preserves as much as possible the connectivity of the path. First, we reorder $\gamma = (\gamma_1, \ldots, \gamma_{2m})$ into $\hat \gamma = (x_1, \ldots, x_{2m})$ by setting for $i$ odd, $\hat \gamma_i = \gamma_i$ and for $i$ even, $\hat \gamma_{i,t} = \gamma_{i,\ell-t}$. Also, for $i$ odd, we set $k_i = k$ and for $i$ even $k_i = \ell -k-1$. Finally, we write $\hat \gamma_i = (\hat \gamma'_i, \hat \gamma''_i)$ with $\hat \gamma'_i = (\hat \gamma_{i,0}, \ldots, \hat \gamma_{i,k_i})$ and $\hat \gamma''_i = (\hat \gamma_{i,k_i+1}, \ldots, \hat \gamma_{i,\ell})$. To each $i$, we say that $\gamma_i$ is connected or disconnected whether $ G(\hat \gamma''_i)$ intersects the graph $H_i = \cup_{j < i} G(\hat \gamma_j) \cup   G(\hat \gamma'_i)$ or not. If $\gamma_i$ is disconnected, we define for $0 \leq t \leq \ell$, $x_{i,t} = \hat \gamma_{i,t}$. If $\gamma_i$ is connected, for $0 \leq t \leq k_i$, we set $x_{i,t} = \hat \gamma_{i,t}$, and if $q_i > k_i$ is the first time such that $\hat \gamma_{i,q_i} \in H_i$, we set for $k_i +1 \leq t \leq q_i$, $x_{i,t} = \hat \gamma_{q_i+k_i +1 -t}$  and for $q_i + 1 \leq t \leq \ell$, $x_{i,t} = \hat \gamma_{i,t}$. We then explore the sequence $(x_{i,t})$ in lexicographic order and set $y_{i,t} = \{ x_{i,t}, x_{i,t+1} \}$. The definition of first time, tree edge and excess edge carry over, that is $(i,t) \ne (i,k_i)$ is a first time if the end vertex of $x_{i,t+1}$ has not been seen before. When $\gamma_i$ is connected, we add the extra mark $(q_i,\hat \gamma_{q_i})$, if $\gamma_i$ is disconnected this extra mark is set to $0$. With our ordering, all vertices of $V(\gamma) \backslash\{1 \}$ will have an associated tree edge, at the exception of $x_{i,k_i+1}$ when $\gamma_{i}$ is disconnected. If $\delta$ is the number of disconnected $\gamma_i$'s, we deduce that there are $\delta + e + v -1$ excess edges. Note however that there are at most  $\epsilon = e + v -1$ excess edges in each connected component of $G(\gamma)$.   

We may now repeat the proof of Lemma \ref{le:enumpath}. The main difference is that, for each $i$, we use that $\hat \gamma'_i$ and $\hat \gamma''_i$ are tangle free, it gives short cycling times and long cycling times for both $\hat \gamma'_i$ and $\hat \gamma''_i$.  For each $i$, there are at most $2$ short cycling times and $2 (\epsilon-1)$ long cycling times. Since there are at most $ \ell  ^{4m \epsilon}$ ways to position these cycling times, we arrive at    
$$
| \cT _{\ell,m,k} (v,e)) | \leq (2 \ell v )^{2m} \ell^{4 m \epsilon} (v^2 \ell )^{4m}   (v^2 ) ^{4m (\epsilon-1)}, 
$$
where the factor $(2 \ell v )^{2m}$ accounts for the extra mark.  Using $v \leq 2 \ell m$, we obtain the claimed statement. \end{proof}

\begin{proof}[Proof of bound (\ref{prop:normR})]
From \eqref{eq:trRk2}, it suffices to prove that 
\begin{equation}\label{eq:boundS2}
S = \sum_{\gamma \in T_{\ell,m,k} }  \dE   \prod_{i=1}^{2m}  \prod_{s=1}^{k} \underline  A_{\gamma_{i,s-1}  \gamma_{i,s}}  \prod_{s=k+2}^{\ell} A_{\gamma_{i,s-1}  \gamma_{i,s}}  \leq (C \log n )^{24m} \alpha^{(2\ell - k)m}. 
\end{equation}

As in \eqref{eq:expAunder}, we find 
\begin{equation}\label{eq:expAunder2}
\dE \prod_{i=1}^{2m}  \prod_{s=1}^{k} \underline  A_{\gamma_{i,s-1}  \gamma_{i,s}}  \prod_{s=k+2}^{\ell} A_{\gamma_{i,s-1}  \gamma_{i,s}}  \leq \PAR{ \frac{ \alpha } n }^{e(\gamma)}. 
\end{equation}
From \eqref{eq:egammaT} and Lemma \ref{le:enumpath2}, we obtain
\begin{eqnarray*}
S & \leq & \sum_{v=1}^{ m (2 \ell - 2 - k)} \sum_{e = v } ^{ m (2 \ell - 2 - k) } | \cT _{\ell,m,k} (v,e) | \PAR{ \frac{ \alpha}{n} }^e n ^v \\
& \leq & \alpha^{(2\ell - k)m }\sum_{v=1}^{m (2 \ell - 2 - k)} \sum_{e = v } ^{ \infty } (4 \ell m )^{ 12 m ( e - v) + 20 m } n ^{v - e } \\
& \leq &  \alpha^{(2\ell - k)m} ( 4 \ell m )^{ 20 m } (2 \ell m) \sum_{s = 0} ^{ \infty } \PAR{ \frac{ (4 \ell m)^{12m }}{n}}  ^{ s} 
 \end{eqnarray*}
For our choice of $m$ in \eqref{eq:choicem2}, we have
$\ell m = o (\log n)^2$ and $(4 \ell m)^{12m } \leq n^{24/25}$ and \eqref{eq:boundS2} follows. 
\end{proof}

\subsection{Bound (\ref{prop:normB2}) on $\|B^{(k)}\|$}

Since $\| K \|$ is of order $n$,  we observe that the second statement in (\ref{prop:normB2}) improves by a factor $\sqrt n$ the crude bound   $  \|  K B^{(k)} \|   \leq \|K \| \| B^ {(k)}\|$. 

\begin{proof}
We only prove the second statement. The first statement is proved similarly. The argument again parallels that used to show (\ref{prop:normDelta}). We take $m$ as in \eqref{eq:choicem}. We have  
\begin{eqnarray}
\|  K B^{(k-2)}  \| ^{2 m} & \leq & \tr \BRA{ \PAR{  ( K  B^{(k-2)} ) (  K B^{(k-2)} )^*}^{m}  } \nonumber\\
& =  &  \sum_{\gamma \in W_{k,m} }   \prod_{i=1}^{m}  \prod_{s=2}^{k}  A_{\gamma_{2i-1,s-1}  \gamma_{2i-1,s}}  \prod_{s=1}^{k-1} A_{\gamma_{2i,s-1}  \gamma_{2i,s}} \label{eq:trBk}
\end{eqnarray}
where $W_{k,m}$ is the set of sequence of paths $\gamma = ( \gamma_1, \ldots, \gamma_{2m})$ defined below \eqref{eq:trDeltak}. From \eqref{eq:trBk} and Markov inequality, it suffices to prove that 
\begin{equation}\label{eq:boundS222}
S = \sum_{\gamma \in W_{k,m} }  \dE  \prod_{i=1}^{m}  \prod_{s=2}^{k}  A_{\gamma_{2i-1,s-1}  \gamma_{2i-1,s}}  \prod_{s=1}^{k-1} A_{\gamma_{2i,s-1}\gamma_{2i,s}} \leq (C \log n )^{16m}  n^m \alpha^{2km}. 
\end{equation}
If $\gamma \in \cW_{k,m}$ is a canonical element of $W_{k,m}$ then $v(\gamma) - 1 \leq e(\gamma) \leq 2 k m $ and $ v(\gamma) \geq 3$. Also, any $\gamma \in \cW_{k,m}$ is isomorphic to less that $n^{v(\gamma)}$ elements in $W_{k,m}$.  Moreover, we have that 
\begin{equation}\label{eq:expA}
 \dE   \prod_{i=1}^{m}  \prod_{s=2}^{k}  A_{\gamma_{2i-1,s-1}  \gamma_{2i-1,s}}  \prod_{s=1}^{k-1} A_{\gamma_{2i,s-1}\gamma_{2i,s}} \leq \PAR{ \frac{ \alpha } n }^{e(\gamma) -m}, 
\end{equation}
indeed, for any $p \geq 1$, $\dE A_{u,v} ^p \leq \alpha / n$ and, since $(\gamma_{2i+1,0},\gamma_{2i+1,1} )  = (\gamma_{2i,k},\gamma_{2i,k-1} )$ at most $m$ distinct edges are covered by the union of $\{ \gamma_{2i-1,0} , \gamma_{2i-1,1}\}$ and $ \{\gamma_{2i,k-1}  ,\gamma_{2i,k}\}$. Hence, using Lemma \ref{le:enumpath}, we obtain
\begin{eqnarray*}
S & \leq &\PAR{ \frac{ n } \alpha }^{m}\sum_{v=3}^{k m +1} \sum_{e = v - 1} ^{ 2km } | \cW _{k,m} (v,e) | \PAR{ \frac{ \alpha}{n} }^e n ^v \\
& \leq & n^{m+1}  \alpha^{(2 k -1) m}\sum_{v=3}^{k m +1} \sum_{e = v - 1} ^{ 2 km } k ^{2m} (2 k m )^{ 6 m ( e -v +1) } n ^{v - e -1} \\
& \leq &  n^{m+1} \alpha^{(2k-1)m} ( \ell )^{ 2m } (\ell m) \sum_{s = 0} ^{ \infty } \PAR{ \frac{ (2 \ell m)^{6m }}{n}}  ^{ s} 
 \end{eqnarray*}
For our choice of $m$ in \eqref{eq:choicem}, the above geometric series converges  and  \eqref{eq:boundS222} follows. 
 \end{proof}
\subsection{Bound (\ref{prop:normS}) on $\|S_k^{(\ell)}\|$}
\begin{proof}
Observe that 
$ L_{ef} = 0 $ unless $e = f$, $K_{ef} = 1$, $K_{f^ {-1} e} = 1$ or $K_{e  f^{-1}} = 1$ in which cases $ L_{ef}  = -1 $. We may thus decompose 
$$
L=    - I -  K',
$$
where $I$ is the identity, and the non-zero entries of $K'$ are equal $1$ and are the pairs $(e,f)$ such that $K_{ef} = 1$, $K_{f^ {-1} e} = 1$ or $K_{e  f^{-1}} = 1$. Thus
\begin{eqnarray*}
\| S^{(\ell)}_k \| & \leq &    \| \Delta^{(k-1)} \|  \| B^{(\ell-k -1)} \|  +  \| \Delta^{(\ell-1)} K' \| \|  B^{(\ell-k-1)} \| .
\end{eqnarray*}
Bounds (\ref{prop:normDelta})-(\ref{prop:normB2}) imply that the first term has a smaller order than the intended bound (\ref{prop:normS}). Hence we only need to bound the last term. We use again the method of moments. We observe that $K'_{ef} \leq K_{ef} + (P K)_{ef} + (KP)_{ef}$. A straightforward adaptation of the proof of bound (\ref{prop:normB2}) shows that \whp, for any $1 \leq k \leq \ell-1$,
$$
\| \Delta^{(k-1)} K' \|  \leq  \sqrt n (\log n) ^{10}  \alpha^{k/2},
$$
which concludes the proof.
\end{proof}



%% file: NBS_SBM_V2.tex
\section{Stochastic Block Model : proof of Theorem \ref{th:main}}
\label{sec:proof_sbm_v2}
In this section, we give the strategy of proof for Theorem \ref{th:main}. Let $\ell = \ell (n)\sim \kappa \log_{\alpha} n$  for some $\kappa \in(0,\gamma/6)$ as in Theorem \ref{th:main}. Recalling the definition (\ref{eq:def_chi}) of vector $\chi_k$,  we further introduce for all $k\in[r]$:
\begin{equation}
\label{eq:defvarphik}
\varphi_k = \frac{ B^{\ell} \chi_k }{\| B^{\ell} \chi_k \|}\,  , \quad   \theta_k = \| B^{\ell} \check \varphi_k  \|,
\end{equation}
and
$$
\zeta_k=   \frac{ B^{\ell} \check \varphi_k }{ \theta_k} =  \frac{ B^{\ell} B^{*\ell} \check \chi_k }{ \| B^{\ell}B^{*\ell} \check \chi_k \|}.
$$
(in the above, if $\theta_k = 0$, we set $\zeta_k = 0$).  We also define
$$
H = \mathrm{span} (  \check \varphi_k , k \in [r]). 
$$
We then have the following
\begin{proposition} \label{prop:Bellphi2} For some $b,c>0$, \whp 
\begin{enumerate}[(i)]
\item \label{i1}$b | \mu^{\ell}_k | \leq \theta_k  \leq   c | \mu^{\ell}_k |$  if $k \in [r_0]$,
\item\label{i3}$ \sign( \mu_k ^{\ell}) \langle \zeta_k , \check \varphi_k\rangle   \geq  b $ if $k  \in [r_0]$,
\item \label{i2}$\theta_k  \leq  (\log n)^{c} \alpha^{\ell/2}$ if $k \in [r] \backslash [r_0]$,
\item\label{i4}$ \ABS{ \langle  \varphi_j ,   \varphi_k\rangle }  \leq  (\log n)^c  \alpha^{3\ell /2} n^{- \gamma /2} $ if $k \ne j \in [r]$, 
\item\label{i5}   $\ABS{ \langle \zeta_j , \check  \varphi_k \rangle }  \leq  (\log n)^c  \alpha^{2 \ell} n^{- \gamma /2} $  if $k \ne j \in [r_0]$.
\item \label{i6} $\ABS{ \langle \zeta_j , \zeta_k \rangle }  \leq  (\log n)^c  \alpha^{5 \ell/2} n^{- \gamma /2}$  if $k \ne j \in [r_0]$.
\end{enumerate}
\end{proposition}
Proposition~\ref{prop:Bellphi2} will follow from the local analysis done in Section~\ref{sec:local}. The next Proposition will be established in Section~\ref{sec:norm_nbm} using a matrix expansion together with norm bounds derived by combinatorial arguments parallel to the proof of Proposition~\ref{prop:Bellx} for the Erd\H{o}s-R\'enyi graph.
\begin{proposition} \label{prop:Bellx2}For some $c >0$, \whp
\begin{equation*}
\sup_{ x  \in H^\perp, \|x\| = 1} \| B^{\ell} x \| \leq (\log n)^c \alpha^{\ell/2}. 
\end{equation*}
\end{proposition}
We now check that the two preceding propositions imply  Theorem
\ref{th:main}.  We consider $(\bar  \varphi_1, \cdots, \bar
\varphi_{r'} )$ obtained by the Gram-Schmidt orthonormalization of $(\check \varphi_1, \cdots, \check \varphi_{r} )$. By Lemma \ref{le:GS} and Proposition \ref{prop:Bellphi2}\eqref{i4}, \whp $r' = r$ and  for all $k \in [r]$,
\begin{equation}\label{eq:checkbarphi}
\|\check \varphi_k -\bar \varphi_k \| = O (  (\log n) ^c  \alpha^{3 \ell /2} n ^{ - \gamma/2}).
\end{equation} 

Similarly, for $k \in [r_0]$, we denote by $\tilde \zeta_k$ the orthogonal projection of $\zeta_k$ on the orthogonal of the vector space spanned by $\bar \varphi_j$, $j  \in [r_0]$, $j\ne k$ and  $\tilde \zeta_j$, $j < k$. We set $\bar \zeta_k = \tilde \zeta_k / \| \tilde \zeta_k\|$. From Proposition \ref{prop:Bellphi2}\eqref{i5}-\eqref{i6}, we find \whp for $k \in [r_0]$,  
\begin{equation}\label{eq:barzeta}
\| \zeta_k -\bar \zeta_k \| = O (  (\log n) ^c  \alpha^{5 \ell/2} n ^{ - \gamma/2}).
\end{equation} 
We then set 
$$
D_0 = \sum_{k = 1} ^{r_0} \theta_k  \bar \zeta_k \bar   \varphi_k^*.
$$

Since $\| \check \varphi_k -\bar \varphi_k \| = o(1)$, from Proposition \ref{prop:Bellphi2}\eqref{i1}-\eqref{i2},  we find by induction on $k \in [r]$, \whp for all $k \in [r]$,
$$
\| B^{\ell} \bar \varphi_k \| = O (   \alpha^{\ell}  ).
$$
Consequently, from Proposition \ref{prop:Bellx2}, we have \whp 
$$
\| B^{\ell} \| = O (  \alpha^{\ell }  ).
$$

In particular, since $D_0 \bar\varphi_k = \theta_k \bar \zeta_k = B^{\ell} \check \varphi_k +  \theta_k(  \bar \zeta_k  - \zeta_k)$, we get for $k \in [r_0]$, 
\begin{eqnarray*}
\| B^{\ell} \bar \varphi_k - D_0 \bar \varphi_k \|&  \leq&    \| B^{\ell} \| \| \bar \varphi_k -  \check \varphi_k \|   +\| B^{\ell} \check \varphi_k - D_0 \bar \varphi_k \|  +\theta_k  \| \bar \zeta_k - \zeta_k \| \nonumber \\
& = & O ( (\log n)^{c} \alpha^{7\ell/2} n ^{ - \gamma/2} ).
\end{eqnarray*}
We have $\alpha^{7\ell/2} n ^{ - \gamma/2} = n ^{7 \kappa  /2 + o(1) - \gamma/2}$. Since  $0 < \kappa < \gamma/6$,  $7\kappa /2  - \gamma / 2 < \kappa /2$, we thus obtain, if $P_0$ is the orthogonal projection of $H_0 = \mathrm{span} (  \bar \varphi_k , k \in [r_0] )$,  
\begin{equation}
\| B^{\ell} P_0  - D_0 \|  =   O ( \alpha^{\ell / 2}  ).\label{eq:boundBD0}
\end{equation}

We also set $D_1 = B^{\ell} P_1$ where $P_1$ is the orthogonal projection of $H_1 = \mathrm{span} (  \bar \varphi_k , k \in [r] \backslash [r_0] )$  and $C = B^{\ell} - D_0 - D_1$. Arguing similarly, from Proposition \ref{prop:Bellphi2}\eqref{i2}, \whp, for $k \in [r] \backslash [r_0]$, 
\begin{equation*}
\| D_1 \bar \varphi_k   \| = \| B^\ell \bar \varphi_k   \|  \leq    \| B^{\ell} \| \| \bar \varphi_k -  \check \varphi_k \|  + \| B^{\ell} \check \varphi_k  \| = O ( (\log n)^{c} \alpha^{\ell / 2}  ).
\end{equation*}
Hence 
\begin{equation}\label{eq:boundBD1}
\| D_1  \| = O ( (\log n)^{c} \alpha^{\ell / 2}  ).
\end{equation}

Also, let $y \in \dR^{\vec E}$ with $\| y \| = 1$. We write $ y = x + h_0 + h_1$ with $x \in H^{\perp}$, $h_1  \in H_1$, $h_0 \in H_0 = \mathrm{span} ( \varphi_k , k \in  [r_0] ) $. We find
$$
\|C  y \| =  \| B^{\ell} x + (  B^{\ell} -D_0 )  h_0  \| \leq    \sup_{ x  \in H ^\perp, \|x\| = 1} \| B^{\ell} x \|  + \| B^{\ell} P_0 -D_0  \|.
$$
Hence, Proposition \ref{prop:Bellx2} and \eqref{eq:boundBD0}-\eqref{eq:boundBD1} imply that \whp
\begin{equation*}
\| C \|  = O (  (\log n)^ c \alpha^{\ell /2}).
\end{equation*}

We decompose $B^{\ell} =   D_0  + R$ with $R = C + D_1$, from what precedes \whp
\begin{equation*}\label{eq:boundnormC2}
\| R \| =  O ( (\log n)^{c} \alpha^{\ell / 2}  ).
\end{equation*}

We are now in position to apply Proposition \ref{prop:sing2eig2}. From \eqref{eq:barzeta}, the statement of Proposition \ref{prop:Bellphi2}\eqref{i3} also holds with $\zeta_k$ replaced by $\bar \zeta_k$. It readily implies Theorem \ref{th:main}.

%% file: NBS_Branching_Processes.tex
\section{Controls on the growth of Poisson multi-type branching processes}
\label{subsec:MTGW}
In this section we derive results for multi-type Galton-Watson
branching processes with Poisson offspring that will be crucial for the local analysis of
Section \ref{sec:local}.
We refer to Section \ref{sec:main_sbm} for the notation used below.
\subsection{Two theorems of Kesten and Stigum}
We consider a multi-type branching process where a particle of type $j \in [r]$ has a $\Poi(M_{ij})$ number of children with type $i$. We
denote by $Z_t=(Z_t(1),\dots, Z_t(r))$ the population at generation
$t$, where $Z_t(i)$ is the number of particles at generation $t$ with
type $i$. We denote by $\cF_t$ the natural filtration associated to
$Z_t$. Following Kesten and Stigum \cite{MR0198552,MR0200979}, we
have the following statement.
\begin{theorem}\label{th:kestenstigum}
For any $k \in [r_0]$,
$$
X_k(t) =  \frac{\langle \phi_{k} , Z _t \rangle}{\mu^t_k} - \langle  \phi_k, Z_0  \rangle,
$$
is an $\cF_t$-martingale converging a.s. and in $L^2$ such that for some $C > 0$ and all $t \geq 0$, $\dE X_k (t) = 0$ and $\dE [  X_k^2 (t) | Z_0 ] \leq C \| Z_0 \|_1$. 
\end{theorem}
\begin{proof}
We include the proof for later use. For $0\leq s<t$, we have
\BEAS
Z _t-M^{t-s} Z _s =\sum_{u=s}^{t-1} M^{t-u-1}(Z_{u+1}-M Z_u),
\EEAS
so that, as $\phi_k^*M=\mu_k \phi_k^*$,
\BEA\label{eq:phikZ}
\frac{\langle \phi_k ,  Z_t  \rangle}{\mu^t_k} = \frac{\langle \phi_k
  ,  Z_0 \rangle}{\mu_k^s} 
+\sum_{u=s}^{t-1} \frac{\langle \phi_k , (Z_{u+1}-MZ_u ) \rangle }{\mu_k^{u+1}}.
\EEA
It follows easily that $(X_k(t))$ is an $\cF_t$-martingale with mean $0$. From Doob's martingale convergence Theorem, the statement will follow if we prove that for some $C >0$ and all integer $t \geq 0$, 
$$
\dE [ X^2_k (t) | Z_0  ] \leq C \| Z_0 \|_1 = C \langle \IND, Z_0 \rangle. 
$$
To this end, we denote by $Z_{s+1}(i,j)$ the number of individuals of type $i$ in the $s+1$-th generation  which
descend from a particle of type $j$ in the $s$-th generation. Thus
$\sum_{j \in [r]} Z_{s+1}(i,j) = Z_{s+1} (i)$. We then have
\BEA
\dE\left[ \|Z_{s+1}-MZ_s \|_2^2\right|Z_s] &=& \sum_{i \in
  [r]}\dE\left[(Z_{s+1}(i)-\sum_{j\in [r]}
  M_{ij} Z_{s}(j) )^2|Z_s\right] \nonumber \\
&=& \sum_{i,j\in [r]} \dE\left[ \left(Z_{s+1} (i,j) - 
    M_{ij} Z_s (j) \right)^2|Z_s(j)\right]  \nonumber \\
&=& \sum_{i,j\in [r]} M_{ij}Z_s(j)  \nonumber  \\
& = & \langle \IND, M Z_s \rangle, \label{eq:dedufheo}
\EEA
where in the penultimate equality we used the fact that the variance of a
Poisson random variable equals its mean. It follows that
$$
\dE  \left[ \|Z_{s+1}-M Z_s \|^2\right|Z_0]  =\langle  \IND , M^{s+1} Z_0 \rangle . 
$$
The Perron-Frobenius Theorem implies that the matrix $(M / \mu_1 )^s $
converges elementwise to $\psi_1 \phi_1 ^*  / \phi_1 ^* \psi_1 $ as $s
\to \infty$. In our case, $\phi_1 = \IND $, $\psi_ 1= \pi^*$ and
$\langle \IND , \psi_1 \rangle  =1$. 
Consequently, for some $C \geq 1$,  
$$
\dE  \left[ \|Z_{s+1}- MZ_s \|_2^2\right|Z_0]  =  (1 + o(1) )   \langle  \IND , Z_0 \rangle \mu_1 ^{s+1} \leq C   \langle  \IND , Z_0 \rangle \mu_1 ^{s+1}.
$$
 Hence finally, 
\BEAS
\dE [ X^2_k (t) | Z_0  ]  &= & \sum_{s=0}^{t-1}  \frac{ \dE [ \langle \phi_k , (Z_{s+1}-MZ_s ) \rangle ^2 | Z_0 ] }{\mu_k^{2(s+1)}}\\
& \leq & \sum_{s=0}^{t-1}  \frac{ \| \phi_k \|_2^2 \dE [   \| Z_{s+1}-MZ_s \|_2 ^2 | Z_0 ] }{\mu_k^{2(s+1)}}\\
& \leq & C  \langle  \IND , Z_0 \rangle \sum_{s=0}^{t-1} \PAR{ \frac{ \mu_1 }{\mu_k^2} }^{s+1}.
\EEAS
Since $\mu_k ^2 > \mu_1$ the above series is convergent.  \end{proof}

We also need to control the behavior of $\langle \phi_{k} , Z _t \rangle$ for $k \in  [r] \backslash[r_0]$.  The next result is contained in Kesten and Stigum \cite[Theorem 2.4]{MR0200979}.
\begin{theorem}\label{th:kestenstigum2}
Assume $Z_0 = x$. For  $k \in [r]\backslash [r_0]$ define 
$$
X_k (t)  = \left\{ \begin{array}{ll} 
  \frac{\langle \phi_{k} , Z _t \rangle}{\mu_1^{t/2}} & \hbox{if } \, \mu_k^2 < \mu_1  \\
   \frac{\langle \phi_{k} , Z _t \rangle}{\mu_1^{t/2} t^{1/2}} & \hbox{if }\,  \mu_k^2 =  \mu_1. 
\end{array}\right.
$$
Then $X_k(t)$ converges  weakly to a random variable $X_k$ with finite positive variance.
\end{theorem}

Note that Theorem 2.4 in \cite{MR0200979} expresses $X_k$ as a mixture of Gaussian variables.  The normalization in the case $ \mu_k^2 =  \mu_1$ comes from the fact that $M$ is diagonalizable, and hence all its Jordan blocks are of size $1$.

\subsection{Quantitative versions of the Kesten-Stigum Theorems}
We will also need probabilistic bounds on the growth of the total population at generation $t$ defined as
$$
S_t = \|Z_t \|_1 = \langle \phi_1 , Z_t \rangle.
$$
We observe that \eqref{cond:deg} implies that $S_t$ itself is a Galton-Walton branching process with offspring distribution $\Poi (\mu_1)$.   

\begin{lemma}\label{le:growthS}
Assume $S_0 =1$.  There exist $c_0, c_1 >0$ such that for all $s \geq 0$, 
$$
\dP \PAR{ \forall k\ge 1, S_{k}  \leq  s \mu_1^k  }  \geq 1 -  c_1 e^{ - c_0 s}. 
$$
\end{lemma}

\begin{proof}
For $k \geq 1$, we set 
$$
\veps_k = \mu_1^{-k/2} \sqrt k \quad  \hbox{ and } \quad f_k = \prod_{\ell=1}^k ( 1 + \veps_{\ell}).
$$
It is straightforward to check that $f_k$ converges, hence there exist constants $c_0,c_1 >0$ such that for all $k \geq 1$,  
\begin{equation}\label{eq:boundfkvepsk}
c_0 \leq f_k \leq c_1 \quad  \hbox{ and } \quad \veps_k \leq c_1.
\end{equation}
Using Chernov bound, if $Y_i$ are i.i.d. $\Poi(\mu_1)$ variables then, for any integer $ \ell \geq 1$ and positive real $s>1$, 
\begin{equation}\label{eq:chernovPoi}
\dP \PAR{ \sum_{i=1} ^\ell Y_i \geq \ell \mu_1 s  }  \leq e^{- \ell \mu_1 \gamma (s)},
\end{equation}
where we have set $\gamma (s) =  s \log s - s + 1$. In particular, on the event $\{ S_k  \leq   s f_k \mu_1^{k } \} \in \cF_k$, we have
$$
\dP \PAR{ S_{k+1} > s f_{k+1}\mu_1^{k+1}   | \cF_k } \leq  e^{ - s \mu_1^{k+1} f_k  \gamma ( 1 + \veps_{k+1} )} \leq e^{ - c'_0  s \mu_1^{k+1}   \veps_{k+1}^2 } = e^{ - c'_0  (k+1)s} ,
$$
where we have used the existence of some $\theta >0$ such that for $x\in[0,c_1]$, one has $\gamma( 1 + x) \geq \theta x^2$. Finally by our choice of $\veps_k$ and \eqref{eq:boundfkvepsk}, if $s \geq \max(1/c'_0, 1 / c_1) $, 
$$
\dP \PAR{ \exists k : S_{k}  > s c_1 \mu_1^{k+1}  } \leq  \sum_{\ell=1}^k e^{ -   c'_0s \ell  } \leq \frac{ e^{ - c'_0 s} }{ 1 - e^{-c'_0s}}. 
$$
Hence we deduce the statement of the lemma for some (suitably redefined) constants $c_0,c_1 >0$. 
\end{proof}

A key ingredient in the subsequent analysis will be the following result, which bounds by how much the growth of processes $s\to \langle \phi_{k} , Z _s \rangle$ deviates from a purely deterministic exponential growth.

\begin{theorem}\label{th:growthZBP}
Let $\beta >0$ and $Z_0 = x\in\dN^r$ be fixed. There exists $C  = C(x,\beta) > 0$  such that with probability at least $1 - n^{-\beta}$, for all $k \in [r_0]$, all $s, t\geq 0$,  with $0 \leq s < t $, 
$$
\ABS{  \langle \phi_{k} , Z _s \rangle -  \mu^{s-t}_{k} \langle  \phi_k, Z_t  \rangle } \leq C  (s +1) \mu_1 ^{s / 2} ( \log n )^{3/2}. 
$$
and for all $k \in [r] \backslash [r_0]$, all $ t \geq 0  $, 
$$
\ABS{  \langle \phi_{k} , Z _t \rangle  } \leq C  (t +1)^2  \mu_1 ^{t / 2} ( \log n )^{3/2}. 
$$
Finally, for all $k \in [r] \backslash [r_0]$, all  $ t \geq 0  $, $\dE \ABS{  \langle \phi_{k} , Z _t \rangle  }^2  \leq C (t +1)^3  \mu_1 ^{t }$. 
\end{theorem}

\begin{proof}
We start with  classical tail bounds for $Y \stackrel{d}{=} \Poi(\lambda)$. From \eqref{eq:chernovPoi} for $s >0$,  
$$
\dP \PAR{  Y - \lambda  >\lambda s } \leq  e^{ - \lambda \gamma(1+s)},
$$
with $\gamma (s) =  s  \log s + 1 - s$. Similarly, for $ s < 1$  one has
$$
\dP \PAR{  Y - \lambda < - \lambda s } \leq e^{ - \lambda \gamma(1-s)},
$$
where by convention $\gamma (x) = + \infty$ for $x \leq 0$. Let $\delta (x) :=  \gamma(1-x) \wedge \gamma (1+x)$. Then  for any $s \geq 0$, 
$$
\dP \PAR{ \ABS{ Y - \lambda } > \lambda s } \leq 2 e^{ - \lambda \delta(s)}.
$$
In particular, for any $i \in [r]$, letting $y := M Z_t$, we have, if $Z_t  \ne 0$, 
$$
\dP \PAR{ | Z_{t+1} (i) -  y (i)  |  > s \| y \|^{1/2}_1 \bigm| \cF_t } \leq  2e^{ - y(i) \delta ( s \| y \|^{1/2}_1 / y(i) )}.
$$
Consider first the case where $ s \| y \|^{1/2}_1  \leq y(i)  $. As there exists $\theta>0$ such that  for all $x\in[0,1]$, $\delta (x) \geq \theta x^2$, we get 
$$
\dP \PAR{ | Z_{t+1} (i) -  y (i)  |  > s \| y \|^{1/2}_1 \bigm| \cF_t } \leq  2e^{ - \frac{ \theta s^2  \| y \|_1 }{ y(i) }} \leq  2e^{ -\theta s^2 } .
$$
Consider now the case where $ s \| y \|^{1/2}_1   > y(i)$. As there exists $\theta'>0$ such that, for all $x \geq 1$, $\delta (x) \geq \theta' x$, we get 
$$
\dP \PAR{ | Z_{t+1} (i) -  y (i)  |  > s \| y \|^{1/2}_1 \bigm| \cF_t } \leq  2e^{ -  \theta' s  \| y \|^{1/2}_1 } \leq  2e^{ -\theta' \sqrt \mu_1  s },
$$
since $Z_t \ne 0$ implies that $\| y \|_1 \geq \mu_1$ from \eqref{cond:deg}. Thus there exists some $ c_0 >0$ such that, for any $s \geq 0$, 
$$
\dP \PAR{ \| Z_{t+1}  -  M Z_t  \|_2  > s \| Z_t \|^{1/2}_1 \bigm| \cF_t } \leq \sum_{i=1}^r \dP \PAR{ | Z_{t+1} (i) -  y (i)|  > \frac{ s \| Z_t \|^{1/2}_1 }{ \sqrt r}\bigm| \cF_t }  \leq 2r e^{ - c_0 (s \wedge s^2)}.
$$
 If $Z_t = 0$,  then $Z_{t+1} = 0$ and the same bound trivially holds. We thus obtain the existence of constants $c_0,c_1>0$ such that, for any $u \geq 1$, 
\begin{equation}\label{eq:4U}
\dP \PAR{ \forall t \ge 0,  \| Z_{t+1}  -  M Z_t  \|_2  \leq   u (t+1)  \log n  \| Z_t \|_1^{1/2}   } \geq 1 - \sum_{t\geq 1} 2r e^{-c_0 u  t \log n} \geq 1 - c_1 n^{ - c_0 u}.   
\end{equation}

Now, from \eqref{eq:phikZ}, for any $s$, $0\le s\le t$, 
$$
\ABS{  \langle \phi_{k} , Z _s \rangle -  \mu^{s-t}_{k} \langle  \phi_k, Z_t  \rangle } \leq \mu_k ^ s \sum_{h=s}^{t-1}   \frac{ \| \phi_k \|_2 \| Z_{h+1}  -  M Z_h  \|_2}{ \mu_k ^{h+1}}\cdot
$$

From Equation (\ref{eq:4U}) and Lemma \ref{le:growthS}, for $C$ large enough, with probability at least $1 - n ^{-\beta}$, we have for all $h \geq 0$ that $ \| Z_{h+1}  -  M Z_h  \|_2 \leq C  (\log n )  ( h+1)  \| Z_h \|_1^{1/2} $ and $\| Z_h \|_1 \leq C( \log n )  \mu_1^h$. On this event, we get,  for $k \in [r_0]$,
$$
\ABS{  \langle \phi_{k} , Z _s \rangle -  \mu^{s-t}_{k} \langle  \phi_k, Z_t  \rangle } \leq C' (\log n )^{3/2}  \mu_k ^ s  \sum_{h=s}^{t-1}  (h+1) \PAR{ \frac{   \sqrt \mu_1  }{ \mu_k }}^{h} \leq C''  (\log n )^{3/2} (s+1)  \mu_1^{s/2},
$$
where at the last line, we used that $\mu_k ^2 > \mu_1$ and $\sum_{h
  \geq s} h a^h \leq c(a) s a ^s$ for $0 < a < 1$. Similarly, on the
same event, for $k \in [r]\backslash [r_0]$, from \eqref{eq:phikZ}, for
$t \geq 1$ and $s=0$, 
$$
\ABS{  \langle \phi_{k} , Z _t \rangle   - \mu_k ^ t  \langle \phi_{k} , Z _0 \rangle } \leq   \mu_k ^t  \sum_{u=0}^{t-1}    \frac{\|\phi_k \|_2 \| Z_{u+1}  -  M Z_u  \|_2 }{\mu_k ^{ u + 1}} \leq C' (\log n )^{3/2}  \mu_k ^t \sum_{u=0} ^{t-1}  (u+1)  \PAR{ \frac{   \sqrt \mu_1  }{ \mu_k }}^{u}.
$$
Using now $\mu_k ^2 \leq \mu_1$, we have $ \sum_{u=0} ^{t-1}  (u+1)  \PAR{ \frac{   \sqrt \mu_1  }{ \mu_k }}^{u} = O ( t^2 (\sqrt { \mu_1}/\mu_k) ^t )$. 

For the last result, we define 
$$
U = \sup_{t \geq 0}  \frac{ \| Z_{t+1}  -  M Z_t  \|_2 }{ (t+1)  \| Z_t\|_1^{1/2} }.
$$
From  \eqref{eq:4U} (with $n = 2$), for any $p \geq 1$, $\dE U ^p = O(1)$. We obtain from \eqref{eq:phikZ} and Cauchy-Schwartz inequality
\begin{eqnarray*}
\dE \ABS{  \langle \phi_{k} , Z _t \rangle   - \mu_k ^ t  \langle \phi_{k} , Z _0 \rangle }^2 &\leq&   \mu_k ^{2t}   \sum_{s=0}^{t-1} \dE \frac{\|\phi_k \|^2_2 \| Z_{s+1}  -  M Z_s  \|^2_2 }{\mu_k ^{2( s + 1)}} \\
&\leq &    \mu_k ^{2t}   \sum_{s=0}^{t-1} \dE \frac{U^2  (s+1)^ 2   \| Z_s \|_1 }{\mu_k ^{2( s + 1)}} \\
& \leq &  t^2   \mu_k ^{2t} \sqrt{\dE U^4 }    \sum_{s=0}^{t-1}   \frac{\sqrt{ \dE \| Z_s \|^2_1 }}{\mu_k ^{2( s + 1)}}  \\
& = & O \PAR{ t^3  \mu_1 ^{t} },
\end{eqnarray*}
where for the last equality, we used the fact that $\dE  \| Z_s \|^2_1
= O ( \mu_1^{2s})$ which follows from Theorem
\ref{th:kestenstigum} with $k=1$ (recall that $\phi_1=\IND$), and the bound of $O(t \mu_1^t/\mu_k^{2t})$ on the sum, which holds for $k\notin[r_0]$. 
 \end{proof}

\subsection{A cross-generation functional}
For the subsequent analysis, in order to control the law of the candidate eigenvectors $B^{\ell}B^{*\ell}\check \chi_k$, we also need to consider a functional of the multi-type branching process which depends on particles in more than one generation. More precisely, assuming that $\|Z_0 \|_1= 1$, we denote by $V$ the particles of the random tree and $o \in V$ the starting particle. Particle $v \in V$ has type $\sigma(v) \in [r]$ and generation $|v|$ from $o \in V$. For $v \in V$ and integer $t \geq 0$, let $Y^v_t$ denote the set of particles of generation $t$ from $v$ in  the subtree of particles with common ancestor $v \in V$. Finally, $Z^v_t  = ( Z^v_t (1), \cdots , Z^v _t (r))$ is the vector of population at generation $t$ from $v$, i.e. $Z_t ^v (i) = \sum_{u \in Y^v_t} \IND(  \sigma(u) = i)$. We set 
$$
S^v _t = \| Z^v _t \|_1 = \langle \phi_1 , Z^v _t \rangle.
$$
With our previous notation, $Z^o_t = Z_t$, $S^o_t = S_t$. We fix an integer $ k \in [r]$, $\ell \geq 1$ and set 
\begin{equation}\label{eq:QklNBW}
Q_{k,\ell} = \sum_{(u_0,\ldots,u_{2\ell+1})\in{\mathcal P}_{2\ell+1}} \phi_k ( \sigma (u_{2 \ell +1}) ) ,
\end{equation}
where the sum is over $ (u_0, \ldots, u_{2 \ell +1})\in{\mathcal P}_{2\ell+1}$, the set of paths in the tree
tree starting from $u_0 = o$ of length $2 \ell +1$ with $(u_0, \ldots,
u_{\ell})$ and $(u_{\ell}, \ldots, u_{2 \ell +1})$ non-backtracking and
$u_{\ell -1}  = u_{\ell +1}$ (i.e. $ (u_0, \ldots, u_{2 \ell +1})$
backtracks exactly once at the $\ell+1$-th step). 

The following alternative representation of $Q_{k,\ell}$ will prove useful. By distinguishing paths $ (u_0, \ldots, u_{2 \ell +1})$ according to the smallest depth $t\in\{0,\ldots,\ell-1\}$ to which they climb back after visiting $u_{\ell+1}$ and the node $u_{2\ell -t}$ they then visit at level $t$ we have that
\begin{equation}\label{eq:defQkl}
Q_{k,\ell} =  \sum_{t=0} ^{\ell -1} \sum_{ u \in Y^o_{t}} L_{k,\ell}^u,
\end{equation}
where we let for $|u|=t\geq 0$,
$$
L^u_{k,\ell} = \sum_{w\in Y^u_1} S^w_{\ell-t-1}\left( \sum_{v\in
    Y^u_1\backslash \{w\}} \langle \phi_k, Z^v_t\rangle \right).
$$
We then have
\begin{theorem}\label{th:growthQkl}
Assume $Z_0 = \delta_x$.  For $k \in [r_0]$, $Q_{k,\ell}/ \mu_k ^{2 \ell}$ converges in $L^2$ as $\ell$ tends to infinity to a random variable with mean $ \mu_k\phi_k
(x)  / (\mu_k^2 / \alpha-1)$.
For $k \in [r] \backslash[r_0]$, there exists  a constant $C$ such that $\dE Q_{k,\ell}^2 \leq  C \alpha ^{2 \ell} \ell^5$. 
\end{theorem}

\begin{proof}[Proof of Theorem \ref{th:growthQkl}]
Let $\cF_t$ be the filtration generated by $(Z_0, \ldots, Z_t)$. The
variables $(L_{k,\ell}^u,  u \in  Y^o_{t})$  are independent given
$\cF_t$. We will show that the sum (\ref{eq:defQkl}) concentrates around its mean. Let
us first compute the mean of $L_{k\ell}^u$ for $u \in Y^o _t$. We use the
fact, that given $\cF_{t+1}$ and $v\ne w \in Y^o_{t+1}$, $Z^v_t$ and $S^w_{\ell-t-1}$
are independent. Hence we have with the short-hand notation $\dE_{\cF_t}=\dE(\cdot|\cF_t)$:
$$
\dE_{\cF_t} L_{k,\ell}^u  = \dE_{\cF_t}  \sum_{ (v,w) \in Y_1 ^u , v \ne w} \dE_{\cF_{t+1}}  \langle \phi_k , Z^v_ t  \rangle  \dE_{\cF_{t+1}}  S^w_{\ell - t -1}.
$$
By assumption \eqref{cond:deg}, $\dE_{\cF_{t+1}}  S^w_{\ell - t
  -1}=\alpha^{\ell-t-1}$. Moreover, we have $\dE_{\cF_{t+1}}  \langle
\phi_k , Z^v_ t  \rangle = \mu_k^t \langle \phi_k,Z^v_0\rangle$ so
that 
\begin{eqnarray}
\dE_{\cF_t} L_{k,\ell}^u   
& = & \alpha^{\ell-t-1} \mu_k^t \dE_{\cF_t} \PAR{(|Y^u_1|-1)\sum_{v\in
    Y^u_1}\langle \phi_k, Z^v_0\rangle}  \nonumber\\
&=& \alpha^{\ell-t-1} \mu_k^t \dE_{\cF_t} \PAR{(|Y^u_1|-1)|Y^u_1|}
\sum_{i\in [r]} \phi_k(i)\frac{M_{i,\sigma(u)}}{\alpha}\nonumber\\
& = & \mu_k^{t+1} \alpha^{\ell - t} \langle  \phi_k , Z^u_0 \rangle, \label{eq:meanLu2}
\end{eqnarray}
and 
\begin{equation}\label{eq:meanLu}
\dE_{\cF_t} \sum_{ u \in Y^o_{t}} L_{k,\ell}^u = \mu_k^{t+1} \alpha^{\ell - t} \langle  \phi_k , Z_t \rangle =\mu_k^{2t+1} \alpha^{\ell - t} Y_k(t),
\end{equation}
where $Y_k (t) = X_k ( t) + \langle \phi_k , Z_0 \rangle$ and $X_k$ is the centered martingale defined in Theorem \ref{th:kestenstigum}.

We now prove the statements of the theorem for $k \in [r_0]$. We find similarly
$$
\begin{array}{ll}
\VAR_{\cF_t} (L_{k,\ell}^u) &= \dE_{\cF_t} ( L_{k,\ell}^u - \dE _{\cF_t} L_{k,\ell}^u ) ^2 \\
&\leq \dE_{\cF_t} ( L_{k,\ell}^u ) ^2 \\
&= \dE_{\cF_t} \left(\sum_{v\ne w\in Y^u_1}  S^w_{\ell-t-1} \langle \phi_k, Z^v_t\rangle \right)^2\\
&\leq C  \dE_*  \langle \phi_k , Z_ t  \rangle^2   \dE_* S_{\ell - t -1}^2, 
\end{array}
$$
where $\dE_* ( \cdot ) = \max_{i \in [r]} \dE (  \cdot | Z_0 = \delta_i )$ and constant $C$ can be taken equal to $\dE_{*}|Y^o_1|^4$. For $k \in [r_0]$, we deduce from Theorem \ref{th:kestenstigum}, that for some new $C >0$, 
\begin{equation}\label{eq:varLu}
\VAR_{\cF_t} \PAR{\sum_{ u \in Y^o_{t}} L_{k,\ell}^u}   = \sum_{ u \in Y^o_{t}} \VAR_{\cF_t} (L_{k,\ell}^u)  \leq C  \mu_k^{2t} \alpha^{2 (\ell - t)} S_t. 
\end{equation}
We now define 
\begin{equation}\label{barqkl}
\bar Q_{k,\ell} = \sum_{t=0} ^{\ell-1}   \dE_{\cF_t}  \sum_{ u \in Y^o_{t}} L_{k,\ell}^u  =  \sum_{t=0} ^{\ell-1}  \mu_k^{2t+1} \alpha^{\ell - t} Y_k(t).
\end{equation}
Since $k \in [r_0]$, $\rho_k := \mu^2_k / \alpha > 1$. We write
$$
\frac{\bar Q_{k,\ell}}{ \mu_k ^{2 \ell}} = \mu_k \sum_{t=0}^{\ell -1} \rho_k^{t - \ell} Y_k(t).
$$
From  Theorem \ref{th:kestenstigum}, $\bar Q_{k,\ell}/ \mu_k ^{2 \ell}$ converges a.s. to $\mu_kY_k(\infty) / ( \rho_k - 1)$ where $Y_k(\infty) = X_k ( \infty) + \langle \phi_k , Z_0 \rangle$ and $X_k (\infty)$ is the limit of the martingale defined in Theorem \ref{th:kestenstigum}. Moreover, $\bar Q_{k,\ell}/ \mu_k ^{2 \ell}$ also converges in $L^2$. Indeed,  we find easily from Cauchy-Schwartz inequality,
\begin{eqnarray*}
\dE\PAR{ \frac{\bar Q_{k,\ell}}{ \mu_k ^{2 \ell}}  -  \frac{\mu_kY_k(\infty)}{  \rho_k - 1}  (1 - \rho_k ^{-\ell}) }^2 &= & \mu_k^2   \dE \PAR{ \sum_{t= 0} ^{\ell -1} \rho_k^{t-\ell}  ( Y_k (t) - Y_k (\infty) )  }^2  \\
& \leq & \mu_k^2  \PAR{  \sum_{t = 0} ^{\ell - 1} \rho_k^{t-\ell} } \PAR{ \sum_{t = 0} ^{\ell - 1}  \dE ( Y_k (t) - Y_k (\infty) ) ^2 \rho_k^{t-\ell} }
\end{eqnarray*}
Since $\rho_k > 1$, the first term of the above expression is of order $O(1)$. For the second term, from Theorem \ref{th:kestenstigum}, for any $\veps > 0$, there is $t_0$ such that for all $t \geq t_0$, $\dE ( Y_k (t) - Y_k (\infty) ) ^2 \leq \veps$. We find that the second term is $O(\veps  + \rho_k^{t_0 - \ell}) = o(1)$. It proves that $\bar Q_{k,\ell} / \mu_k ^{2\ell}$ converges in $L^2$.

We now check that $Q_{k,\ell}$ and $\bar Q_{k,\ell}$ are close in $L^2$ for $k \in [r_0]$.  For a real random variable $Z$, set
$\|Z\|_2 = \sqrt { \dE Z^2}$. From \eqref{eq:defQkl}-\eqref{eq:varLu} and the triangle inequality, we get
\begin{eqnarray*}
\| Q_{k,\ell} - \bar Q_{k,\ell} \|_2 & \leq & \sum_{t=0}^{\ell-1} \NRM{\sum_{ u \in Y^o_{t}} L_{k,\ell}^u - \dE_{\cF_t} \sum_{ u \in Y^o_{t}} L_{k,\ell}^u }_2 \nonumber\\
& = &  \sum_{t=0}^{\ell-1} \NRM{ \PAR{ \VAR_{\cF_t} \PAR{\sum_{ u \in Y^o_{t}} L_{k,\ell}^u}}^{1/2} }_2\\
& \leq & C \sum_{t=0}^{\ell}  \mu_k ^{t} \alpha^{\ell - t} \| \sqrt {S_t}\|_2  = O ( \mu_k ^{ \ell} \alpha^{\ell/2}) = o ( \mu_k^{2\ell})\nonumber,
\end{eqnarray*}
where at the last line, we have used that Lemma \ref{le:growthS} and $k \in [r_0]$. It follows that $\|(Q_{k,\ell} - \bar Q_{k,\ell}) / \mu_k^{2\ell}\|_2$ goes to $0$ and it concludes the statements of the theorem for $k \in [r_0]$.

For $k \notin [r_0]$, we note that $\dE Z^2 \leq \dE \PAR{\dE_Y Z}^2 + \dE \VAR_Y(Z)$ so that
$\|Z\|_2 \leq \|\dE_Y Z\|_2 + \|\VAR_Y(Z)^{1/2}\|_2$. From \eqref{eq:defQkl} and the triangle inequality, we get
\begin{eqnarray}
\| Q_{k,\ell} \|_2 & \leq & \sum_{t=0}^{\ell-1} \NRM{\sum_{ u \in Y^o_{t}} L_{k,\ell}^u}_2 \nonumber\\
& \leq &  \sum_{t=0}^{\ell-1} \NRM{ \dE_{\cF_t} \PAR{\sum_{ u \in Y^o_{t}} L_{k,\ell}^u}}_2 + \NRM{ \PAR{ \VAR_{\cF_t} \PAR{\sum_{ u \in Y^o_{t}} L_{k,\ell}^u}}^{1/2} }_2 \label{eq:L2triangle}.
\end{eqnarray}
The last statement of Theorem \ref{th:growthZBP} and \eqref{eq:varLu} give
$$
 \VAR_{\cF_t} \PAR{ \sum_{ u \in Y^o_{t}} L_{k,\ell}^u }= O ( S_t   \alpha^{2 (\ell - t)}   \alpha^t t^3  ) .
$$
We deduce from \eqref{eq:L2triangle} and \eqref{eq:meanLu},
$$
\| Q_{k,\ell}\|_2 \leq C \sum_{t=0}^{\ell-1} \PAR{ \mu_k^{t} \alpha^{\ell - t}\| \langle  \phi_k , Z_t \rangle \|_2 +  \alpha^{ \ell - t}   \alpha^{t/2} t^{3/2} \| \sqrt {S_t}\|_2 }= O ( \alpha^\ell \ell^{5/2}  ).
$$
It concludes the proof.
\end{proof}

We finish this section with a rough bound on $Q_{k,\ell}$. 
\begin{lemma}\label{le:growthQklLp}
For any $p \geq 1$, there exists a constant $C = C(p,\alpha) >0$ such that for any $k \in [r]$, 
$$
\dE |Q_{k,\ell}|^p \leq C \alpha^{2 p\ell}. 
$$
\end{lemma}
\begin{proof}
We use the notation of Theorem \ref{th:growthQkl}. First, from Lemma \ref{le:growthS}, for any $p \geq 1$, $\dE S_{t} ^p  \leq C \alpha^{t p}$. In particular, for any $ v \in  Y^o_{t+1}$ and $k \in [r]$,
$$
\dE_{\cF_{t+1} } \ABS{ \langle \phi_k, Z^v_t \rangle    }^p  = O \PAR{  \dE_{\cF_{t+1} } \PAR{S^v_t }^p  }  =O (  \alpha^{t p}).
$$ 
We use twice the bound $|\sum_{i=1}^n  x_i |^p \leq n^{p-1} \sum_{i=1}^n |x_i|^p$. We find 
\begin{eqnarray*}
\dE_{\cF_t} \ABS{ \sum_{ u \in Y^o_{t}} L_{k,\ell}^u }^p &\leq & S_{t} ^{p-1} \sum_{ u \in Y^o_{t}} \dE_{\cF_t}  |L_{k,\ell}^u|^p \\
&\leq &S_{t} ^{p-1} \sum_{ u \in Y^o_{t}} \dE_{\cF_t}  \dE_{\cF_{t+1}} \ABS{\sum_{(v,w) \in Y_1 ^u , v \ne w } \langle \phi_k, Z^v_t \rangle   S^w_{\ell - t -1} }^p \\
&\leq & S_{t} ^{p-1} \sum_{ u \in Y^o_{t}} \dE_{\cF_t} (S^u_1)^{2(p-1)}  \sum_{(v,w) \in Y_1 ^u , v \ne w }  C \alpha^{tp} \alpha^{(\ell - t - 1)p} \\
&\leq & S_{t} ^{p-1} \sum_{ u \in Y^o_{t}}  C \alpha^{(\ell - 1) p}  \dE_{\cF_t} (S^u_1)^{2p}  \\
&\leq & C' \alpha^{\ell p} S_{t} ^{p},
\end{eqnarray*}
for some new constant  $C'$ depending on $\alpha$ and $p$. We deduce that for some new $C >0$,
$$
\dE \ABS{ \sum_{ u \in Y^o_{t}} L_{k,\ell}^u }^p \leq C^p \alpha^{(\ell + t) p  }.
$$
For a real random variable $Z$, set $\|Z\|_p = ( \dE Z^p)^{1/p}$.  We use \eqref{eq:defQkl} and the triangle inequality, we get
\begin{eqnarray*}
\| Q_{k,\ell}\|_p  \leq  \sum_{t=0}^{\ell-1} \NRM{\sum_{ u \in Y^o_{t}} L_{k,\ell}^u}_p  \leq   C \sum_{t=0}^{\ell-1} \alpha^{(\ell + t) } = O ( \alpha^{2\ell}).
\end{eqnarray*} \end{proof}

\subsection{Decorrelation in homogeneous Galton-Watson branching processes}
We now establish that the variables $Q_{k,\ell}$ and $Q_{j,\ell}$ are uncorrelated when $k\ne j$. To this end we need the following lemma.
\begin{lemma}\label{mkv_field}
Assume that the spin $\sigma(o)$ at the root node $o$ is distributed according to the stationary distribution $\pi$. 
Conditionally on the branching tree ${\mathcal T}$, the process of spins $\sigma(u)$ attached to the vertices of the tree is a Markov random field. For any two neighbor nodes $u,v$ of ${\mathcal T}$ and any $i,j\in[r]$, one has the following transition probabilities 
$$
\dP(\sigma(u)=i|\sigma(v)=j,{\mathcal T})=\frac{1}{\alpha}M_{ij}.
$$
For any two (possibly equal) nodes $u,v$ of ${\mathcal T}$, any $i,j\in[r]$, $i\ne j$, one has
\begin{equation}\label{cond_exp_tmp}
\dE(\phi_k(\sigma(u))\phi_j(\sigma(v))|{\mathcal T})=0.
\end{equation}
\end{lemma}
\begin{proof}
By standard properties of independent Poisson random variables, conditionally on the spin $\sigma(o)$ and on the number of children of the root $o$, then the spins of each of the children of the root are i.i.d., distributed according to $M_{\cdot \sigma(o)}/\alpha$. Moreover, $\pi$ is the stationary distribution for this transition kernel, which is reversible, as follows from the relation $M=\Pi W$ and the facts that $W$ is symmetric together with  the assumption (\ref{cond:deg}) that the column sums of $M$ all coincide with $\alpha$. 
The Markov random field property and the expression of the transition kernel follow by iterating this argument. 

We now evaluate the conditional expectation in (\ref{cond_exp_tmp}). Let $u_1=u,\ldots,u_t=v$ denote the unique path in ${\mathcal T}$ connecting nodes $u$ and $v$. Let ${\mathcal{F}}_s$ denote the $\sigma$-field generated by ${\mathcal T}$  and the spin variables $\sigma(u_1),\ldots,\sigma(u_s)$. We then have by the Markov random field property
$$
\dE(\phi_j(\sigma(u_{s+1}))|{\mathcal{F}}_s)=\sum_{i\in[r]}\frac{1}{\alpha}M_{i \sigma(u_s)}\phi_j(i)=\frac{\mu_j}{\alpha}\phi_j(\sigma(u_s)),
$$
where we used the fact that $\phi_j$ is a left-eigenvector of $M$ associated with eigenvalue $\mu_j$. Thus
$$
\begin{array}{ll}
\dE(\phi_k(\sigma(u))\phi_j(\sigma(v))|{\mathcal T})&=\left(\frac{\mu_j}{\alpha}\right)^{t-1}\dE(\phi_k(\sigma(u))\phi_j(\sigma(u))|{\mathcal T})\\
&=\left(\frac{\mu_j}{\alpha}\right)^{t-1}\sum_{i\in[r]}\pi_i \phi_k(i)\phi_j(i)\\
&=0,
\end{array}
$$
where the last equality follows from $\pi$-orthogonality (\ref{eq:phiONpi}) between vectors $\phi_k$ and $\phi_j$ for $j\ne k$.
\end{proof}

We now show the following

\begin{theorem}\label{le:ouff2}
Let $j \ne k \in [r]$ and $Z_0 = \delta_\iota$ where $\iota$ has distribution $(\pi(1), \ldots , \pi(r))$. Then for any $\ell \geq 0$, 
$$
\dE Q_{k,\ell}Q_{j,\ell} = 0. 
$$
\end{theorem}
\begin{proof}
Write $Q_{k,\ell}$ as 
$$
Q_{k,\ell}=\sum_{(v,w)\in{\mathcal P}({\mathcal T})}\phi_k(\sigma(w)),
$$
where the sum extends over a set  ${\mathcal P}({\mathcal T})$ of node pairs $(v,w)$ that depends only on the tree ${\mathcal T}$. Using the analogue expression for $Q_{j,\ell}$ one obtains
$$
\begin{array}{ll}
\dE(Q_{k,\ell}Q_{j,\ell}|{\mathcal T})&=\sum_{(v,w)\in{\mathcal P}({\mathcal T})}\sum_{(v',w')\in{\mathcal P}({\mathcal T})}\dE(\phi_k(\sigma(w))\phi_j(\sigma(w'))|{\mathcal T})\\
&=0
\end{array}
$$
by Lemma \ref{mkv_field}, Equation\eqref{cond_exp_tmp}. This concludes the proof.
\end{proof}

%% file: NBS_Local_Structure.tex
\section{Local structure of random graphs}
\label{sec:local}
We now derive the necessary controls on the local structure of the SBM random graphs under consideration. Coupling results will allow to bound the deviation of their local structure from branching processes. Asymptotic independence between local neighborhoods of distinct nodes will then be used to establish weak laws of large numbers.
\subsection{Coupling}

For $ e , f \in \vec E(V)$, we define the "oriented" distance 
$$
\vec d ( e, f) = \min_{\gamma} \ell( \gamma )
$$
where  the minimum is taken over all self-avoiding paths $\gamma = (\gamma_0, \gamma_1, \cdots , \gamma_{\ell+1} )$  in $G$ such that $(\gamma_0, \gamma_1) = e$, $(\gamma_{\ell} , \gamma_{\ell+1} ) = f$ and for all $1 \leq k \leq \ell+1$, $\{ \gamma_k , \gamma_{k+1} \} \in E$ (we do not require that $e \in \vec E$). Observe that $\vec d$ is not symmetric, we have instead $\vec d (e, f) = \vec d (f^{-1}, e^{-1})$.

Then,  for integer $t \geq 0$,  we introduce the vector $Y_t(e) = (Y_t(e) (i))_{i \in [r]}$ where, for $i \in [r]$,
\begin{equation}\label{eq:defYti}
Y_t (e)(i) = \ABS{ \left\{ f \in \vec E : \vec d ( e, f) = t  , \sigma(f_2) =  i \right\} }.
\end{equation}
We also set 
$$
S_t (e) = \|Y_t(e) \|_1 = \ABS{ \left\{ f \in \vec E : \vec d ( e, f) = t \right\} }.
$$
The vector $Y_t(e)$ counts the types at oriented distance $t$ from $e$.

We shall denote by $S_t(v)$ the set of vertices at distance $t$ from $v$.
We introduce 
\begin{align}
& n(i)  =   \sum_{v=1} ^n \IND ( \sigma(v)  = i) , 
 \quad \quad \pi_n(i)   = \frac { n(i)  }{ n} , \nonumber \\
& \alpha_n (i )  =  \sum_{j=1}^r \pi_n(i) W_{ij}  ,  \quad \quad 
  \bar \alpha_n  = \max_{i \in [r]} \alpha_n (i) = \alpha + O ( n^{-\gamma}),\label{eq:defalphan}
\end{align}
where at the last line we have used Assumption \eqref{cond:deg}-\eqref{eq:defgamma}.  Central to our local study is the classical exploration process of the neighborhood of $v$ which starts with $A_0 = \{ v \}$ and at stage $t \geq 0$, if $A_t$ is not empty, takes a vertex in $A_t$ at minimal distance  from $v$, say $v_t$, reveals its neighbors, say $N_{t+1}$, in $[n] \backslash A_t$, and update $A_{t+1} = ( A_t \cup N_{t+1}  ) \backslash \{ v_t\}$. We will denote by $\cF_t$ the  filtration generated by $(A_0, \cdots, A_{t})$ and by $D_t = \cup_{0 \leq s \leq t} A_s$ the set of discovered vertices at time $t$. We start by establishing a rough bound on the growth of $S_t$.

\begin{lemma}\label{le:growthSt}
There exist $c_0, c_1 >0$ such that for all $s \geq 0$ and for any $w \in [n] \cup \vec E(V)$,    
$$
\dP \PAR{ \forall t \ge 0: S_{t} (w)   \leq  s \bar \alpha_n^t  }  \geq 1 -  c_1 e^{ - c_0 s}. 
$$
Consequently, for any $p\geq 1$, there exists $c >0$ such that $$\dE \max_{v \in [n] , t \geq 0} \PAR{  \frac{ S_{t} (v) }{ \bar \alpha^t_n} }^p \leq c (\log n)^p.$$
 \end{lemma}

\begin{proof}
Recall that $\dE |X|^p = p \int_0^{\infty} x^{p-1} \dP ( |X| \geq t ) dt$ and $\dP ( \max_v X_v \geq t ) \leq 1 \wedge \sum_{v} \dP ( X_v \geq t)$. Then the second statement is a direct consequence of the first statement.

To prove the first statement, observe that, in the exploration process, given $\cF_t$, if $v_{t}$ has type $j$, the number of neighbors of $v_{t}$ in $[n] \backslash D_t$ is upper bounded stochastically by 
$$
V_j = \sum_{i=1} ^r V_{ij},
$$
where $V_{ij} \stackrel{d}{=} \Bin ( n(i), W_{ij} / n) = \Bin ( n(i), \pi_n (i) W_{ij} / n(i) ) $ are independent. In particular, for any $\theta \geq  0$, using $1 + t \leq e^t$, we find
$$
\dE e^{\theta V_j} =  \prod_{i=1}  ^r \dE  \PAR{ 1 - \frac{\pi_n(i) W_{ij}}{n(i) } + \frac {\pi_n (i) W_{ij}}{n(i) }  e^\theta }^{n(i)} \leq e^{ - \alpha_n ( j) + \alpha_n (j) e^{\theta} } \leq e^{ - \bar \alpha_n+ \bar \alpha_n  e^{\theta} }. 
$$

For any $j \in [r]$, we have thus bound for $\theta \geq 0$ the characteristic function of $V_j$ by the characteristic function of a $\Poi(\bar \alpha_n)$ variable. It remains finally to repeat the proof of Lemma \ref{le:growthS} from \eqref{eq:chernovPoi} with $\mu_1$ replaced by $\bar \alpha_n$. \end{proof}

We now check that the random graph $G$ is locally tree-like. For $v \in [n]$ and integer $h \geq 0$, we denote by $(G,v)_h$ the rooted subgraph of $G$ rooted at $v$, spanned by the vertices at distance at most $h$ from $v$. If $e = (u,v) \in \vec E (V)$, we set $(G,e)_h = (G',v)_h$ where $G'$ is the graph $G$ with the edge $\{u,v\}$ removed (if it was present in $G$). 

\begin{lemma} \label{le:tls2}
Let $\ell \sim \kappa \log_\alpha n $ with $\kappa < 1/2$. Then, \whp  the random graph $G$ is $\ell$-tangle-free and \whp there are less that $ \bar \alpha^\ell    \log n$ vertices whose $\ell$-neighborhood contains a cycle. 
\end{lemma}

\begin{proof}
We start by proving the second statement. In the exploration process all vertices in  $V(G,v)_\ell$ have been revealed at time $\tau$. Let $\tau$ be defined as the first time $t$ at which all nodes at distance $\ell$ or less from $v$ have been discovered.  It is clearly a stopping time for the filtration $ \cF_t$. By construction, given $\cF_\tau$, the set of discovered edges in $V(G,v)_{\ell}$ builds a spanning tree of $V(G,v)_{\ell}$. Also, given $\cF_\tau$, the number of undiscovered edges between two vertices in $V(G,v)_{\ell}$ is stochastically upper bounded by $\Bin (  m , a / n)$ where $m = | V ( G, v)_{\ell} | = S_{\ell} (v)$ and $a = \max_{i,j} W(i,j)$. It follows from Lemma \ref{le:growthSt} that, for some $c >0$,  
\begin{equation}\label{eq:probnottree}
\dP \PAR{ (G,v)_{\ell} \hbox{ is not a tree}} \leq \frac{ a \dE S_{\ell} (v) }{n} \leq \frac{ c \alpha^\ell  }{n}.
\end{equation}
Hence from Markov inequality,
$$
\dP \PAR{ \sum_{v} \IND ( (G,v)_{\ell} \hbox{ is not a tree} ) \geq   \alpha^\ell    \log n } \leq \frac{ c}{\log n}.
$$
The second statement follows.

We now turn to the first statement. First recall that the probability that $\Bin(m,q)$ is not in $\{0,1\}$ is at most $q^2 m (m-1) \leq q^2 m^2$.  Also, if $G$ is $\ell$-tangled, then there exists $v \in [n]$ such that $V(G,v)_{\ell}$ has at least two  undiscovered edges. In particular, from the union bound,  
$$
\dP \PAR{ G \hbox{ is $\ell$-tangled}} \leq \sum_{v=1} ^n \frac{ a^2 \dE S_{\ell} (v)^2 }{n^2} \leq  \frac{ c \alpha^{2 \ell}  }{n} = o(1).
$$
where $c >0$ and we have used again Lemma \ref{le:growthSt}. \end{proof}

We conclude this subsection with a coupling of the process $Z_t(e)$ and a multi-type Galton-Watson tree. Recall that for probability measures $P,Q$ on a countable set $\cX$, the total variation distance is given by 
$$
\DTV (P,Q) = \frac 1 2 \sum_{x \in \cX} | P(x) - Q(x) | = \min \dP ( X \ne Y), 
$$
where the minimum is over all coupling $(X,Y)$ such that $X \stackrel{d}{\sim} P$, $Y \stackrel{d}{\sim} Q$.   

\begin{proposition}\label{prop:couplingYZ} 
Let $\ell \sim  \kappa \log_\alpha n  $ with $0 \leq \kappa  < 1/2$ and $e = (u,v) \in \vec E (V)$. Let $(T,o)$ be the random rooted tree associated to the Galton-Watson branching process defined in Section \ref{subsec:MTGW} started from $Z_0 = \delta_{\sigma(v)}$. The total variation distance between the law of $(G,e)_{\ell}$ and $(T,o)_{\ell}$ goes to $0$ as $O (  (\log n ) \alpha^{\ell}n^{-\gamma \wedge ( 1- \kappa) })$. The same holds with $(G,e)_{\ell}$ replaced by $(G,v)_{\ell}$. 
\end{proposition}

\begin{proof}
We prove the first statement, the proof of the second statement is identical (see comment below \eqref{eq:nti}). If $e = (u,v)$ and $G' = G \backslash \{u,v\}$, we consider the filtration $\cF_t$ associated to the exploration process of $(G',v)$. We let $\tau$ be the stopping time where all vertices of $(G,e)_{\ell}$ have been revealed. We set $y_0 = \delta_{v} $ and at step $t \geq 0$, we denote by $y_{t+1} = (y_{t+1}(1), \ldots, y_{t+1} (r) )$ the number of discovered neighbors of $v_{t}$ in $[n] \backslash A_t$  of each type. If $\sigma(v_t) = j$ then, given $\cF_t$, the variables $(y_{t+1}(i))_{i \in [r]}$ are independent and $y_{t+1} (i)$ has distribution
$
\Bin ( n_{t} (i), W_{ij} / n )
$
where 
\begin{equation}\label{eq:nti}
n_{t}(i) = n(i) -  \sum_{s=0}^t  y_s (i) - \IND ( t= 0, \sigma(u) = i),
\end{equation}
(the last term comes from the difference between $G$ and $G'$; this term is not present in the case of the second statement on $(G,v)_{\ell}$). We perform the same exploration process on $(T,o)$, that is a breath-first search of the tree, we discover at each step the offsprings say $x_{t+1} = (x_{t+1}(1), \ldots, x_{t+1} (r) )$ of the active vertex $v_t$. In particular, if $v_t$ has type $j$ then the variables $(x_{t+1}(i))_{i \in [r]}$ are conditionally independent and $x_{t+1} (i)$ has distribution
$
\Poi ( \pi (i) W_{ij}).
$
To couple the two processes, we shall use the following classical bounds, (see e.g. \cite{BarbourChen05}),
\begin{equation}\label{eq:TVstein}
\DTV \PAR{ \Bin \PAR{  m , \frac \lambda  m } , \Poi ( \lambda) } \leq \frac \lambda  m \;\quad  \hbox{ and } \; \quad  \DTV \PAR{  \Poi ( \lambda) ,  \Poi ( \lambda') } \leq | \lambda - \lambda' |.
\end{equation}

For $0 \leq t \leq \tau$, define the event  $\Omega_t = \{   |D_t|    \leq c \alpha^\ell \log n \} \in \cF_t$, where $D_t$ is the set of discovered vertices. By Lemma \ref{le:growthSt}, for $c$ large enough, $\tau \leq c \alpha^{\ell} \log n$ and $\Omega_{\tau}$ holds with probability larger than $1 - 1/ n$. Also, by \eqref{eq:probnottree}, with probability at least $1 -c \alpha^{\ell}/ n$, $(G,e)_{\ell}$ is a tree. It follows that by iteration, it is enough to check that, if $\Omega_t$ holds,  there exists $C >0$ such that 
\begin{equation}\label{eq:couplingrec}
\DTV ( P_{t+1} , Q_{t+1} ) \leq C  n^{- \gamma \wedge ( 1- \kappa)} ,
\end{equation}
where $P_{t+1}$ is the distribution of  $y_{t+1}$ under $\dP ( \cdot | \cF_t)$ and $Q_{t+1}$ has the law of $x_{t+1} = (x_1, \cdots, x_r)$ where $x_i$ are independent with distribution $\Poi ( \pi (i) W_{ij} )$, with $j = \sigma ( v_{t})$. However from \eqref{eq:TVstein} and the triangle inequality, we have
\begin{eqnarray*}
\DTV ( P_{t+1} , Q_{t+1} )&  \leq &\DTV \PAR{ P_{t+1} ,\bigotimes_{i \in [r]} \Poi \PAR{ \frac{ n_t(i) W_{ij} }{n} }} + \DTV \PAR{ \bigotimes_{i \in [r]} \Poi \PAR{ \frac{ n_t(i) W_{ij} }{n}} , Q_{t+1} } \\
&\leq & \sum_{i=1}^r \PAR{\frac{ W_{ij} }{n} + W_{ij} \ABS{ \frac{n_t(i)}{n} - \pi(i)}} \\
& \leq & C \frac{\alpha^{\ell} (\log n)}{n} + \sum_{i=1}^r W_{ij} \ABS{ \pi_n(i) - \pi(i)}. 
\end{eqnarray*}
From \eqref{eq:defgamma}, the latter is $O ( n^{-\gamma \wedge ( 1 - \kappa)})$. We thus have proved that \eqref{eq:couplingrec} holds for some new $C >0$.
\end{proof}

We will use the following corollary of Proposition \ref{prop:couplingYZ}.

\begin{corollary}
\label{cor:vartree}
Let $\ell  \sim \kappa \log_\alpha n$ with $0 < \kappa  < \gamma \wedge 1/2$. For $e \in \vec E(V)$, we define the event $\cE (e)$ that for all $0 \leq t <  \ell $ and $k \in [r]$: 
$\ABS{  \langle \phi_k , Y_t (e) \rangle - \mu_k^{t - \ell}  \langle \phi_k , Y_\ell (e) \rangle } \leq (\log n)^4 \alpha^{t/2} $, if $k \in [r_0]$, and $\ABS{  \langle \phi_k , Y_t (e) \rangle } \leq (\log n)^4 \alpha^{t/2} $, if $k \in [r] \backslash [r_0]$.

Then,  \whp the number of edges $e \in \vec E$ such that  $\cE(e)$ does not hold is at most $(\log n)^{2} \alpha^{\ell} n^{1  - \gamma}$.\end{corollary}

\begin{proof}
First,  with exponentially large probability there are less that $2 \alpha n$ edges in $\vec E$. From the union bound, it is thus enough to prove that for any $e \in \vec E (V)$, $\dP ( \cE(e)^c  ) \leq C   (\log n) \alpha^{\ell} n^{- \gamma}$.   To show this, use the coupling result of  Proposition \ref{prop:couplingYZ} to deduce that with probability at least $1 - C (\log n) \alpha^{\ell} n^{- \gamma}$, the processes $(Y_t(e) )_{0 \leq t \leq \ell}$ and $(Z_t)_{0 \leq t \leq \ell}$ coincide. It then remains to use Theorem \ref{th:growthZBP} with $\beta = 1$.
\end{proof}

\subsection{Geometric growth of linear functions of non-backtracking walks}

For $k \in [r]$, we recall that
$$
\chi_k (e) = \phi_k (\sigma(e_2)).
$$
The next proposition asserts that for most $e \in \vec E$, $\langle B^t \chi_k , \delta_e \rangle$ grows nearly geometrically in $t$ with rate $\mu_k$ up to an error of order $\alpha^ {t/2}$.

\begin{proposition}\label{prop:localB}
Let $\ell \sim \kappa \log_{\alpha} n $ with $0 < \kappa < \gamma \wedge 1/2$. There exists  a random subset of edges $\vec E_{\ell} \subset \vec E$ such that \whp the following holds
\begin{enumerate}[(i)]
\item for all $e \in \vec E \backslash \vec E_{\ell}$, $0 \leq t \leq \ell$,  
\begin{eqnarray*}
\ABS{\langle B^t \chi_k , \delta_e \rangle -  \mu_k ^{t-\ell} \langle B^\ell \chi_k , \delta_e \rangle } & \leq  & (\log n )^4 \alpha^{t/2},  \quad \hbox{ if $k \in [r_0]$},\\
\ABS{\langle B^t \chi_k , \delta_e \rangle } & \leq  & (\log n )^4 \alpha^{t/2},  \quad \hbox{ if $k \in [r] \backslash [r_0]$},
\end{eqnarray*}
\item
for all $e \in \vec E_{\ell}$, $0 \leq t \leq \ell$ and $k \in [r]$,
\begin{eqnarray*}
\ABS{\langle B^t \chi_k , \delta_e \rangle } &  \leq  & (\log n )^2  \alpha^{t},
\end{eqnarray*}
\item 
$| \vec E_{\ell} | \leq (\log n)^3 \alpha^{\ell} n^{1 - \gamma}$.
\end{enumerate} 
\end{proposition}

\begin{proof}
We define $\vec E_{\ell}$ as the set of oriented edges such that either $(G,e_2)_{\ell}$ is not a tree or the event $\cE(e)$ defined  in Corollary \ref{cor:vartree}  does not holds.  Then by Lemma \ref{le:tls2} and Corollary \ref{cor:vartree}, \whp $\vec  E_{\ell}$ satisfies condition (iii). Moreover, by definition if  $(G,e_2)_{\ell}$ is a tree
\begin{equation}\label{eq:etachi}
\langle B^t \chi_k , \delta_e \rangle = \langle \phi_k , Y_t (e) \rangle. 
\end{equation}
and statement (i) follows from Corollary \ref{cor:vartree}. For statement (ii), we simply use that \whp  $G$ is tangle free (by Lemma \ref{le:tls2}), hence, there are at most two non-backtracking walks of length $t$ from $e$ to any $f$. We get  
$$
\langle B^t \chi_k , \delta_e \rangle \leq 2 \| \phi_k \|_{\infty} S_t (e). 
$$ 
However, by Lemma \ref{le:growthSt} \whp for all $t \geq 0$ and all $e \in \vec E$, $|S_t(e)| \leq C (\log n) \alpha^t$. \end{proof}

\begin{corollary}\label{cor:Blperp}
Let $\ell \sim \kappa \log_{\alpha} n $ with $0 < \kappa < \gamma/2$.  With high probability, for any $0 \leq t \leq \ell-1$ and $k \in [r]$,
$$
\sup_{\langle B^\ell \chi_k , x \rangle = 0, \| x\| = 1 } \ABS{ \langle B^t \chi_k , x \rangle } \leq  (\log n)^5 n^{1/2} \alpha^{t/2}.
$$
\end{corollary} 
\begin{proof}
We write
\BEAS
\langle B^t \chi_k , x \rangle  =  \sum_{e \in \vec E_{\ell} } x_e \langle B^t \chi_k , \delta_e \rangle +  \sum_{e \notin \vec E_{\ell} } x_e \langle B^t \chi_k , \delta_e \rangle  = I + J. 
\EEAS
From Cauchy-Schwartz inequality, the first term is bounded \whp by 
$$
\ABS { I }  \leq  (\log n )^2  \alpha^{t}  \sum_{e \in \vec E_{\ell} } |x_e|\leq   (\log n )^2  \alpha^{t}  \sqrt{ |\vec E_{\ell} |} \leq (\log n)^{4} \alpha^t \alpha^{\ell/2} n^{(1 - \gamma)/2} = o (   n^{1/2} \alpha^{t/2} ),
$$
where we have used  that $\alpha^{t /2}  \alpha^{\ell/2} n^{ - \gamma/2}  \leq  n^{\kappa - \gamma/2+o(1)} $ and $\kappa < \gamma/2$. 
For the second term, if $k \in [r_0]$, using $\langle B^{\ell} \chi_k , x \rangle = 0$, we get similarly \whp
\BEAS
\ABS { J } & \leq &  \mu_k^{t - \ell}  \sum_{e \in \vec E_{\ell} } |x_e| \ABS{ \langle B^\ell \chi_k , \delta_e \rangle} + \sum_{e \notin \vec E_\ell} |x_e |  \ABS{ \langle B^t \chi_k , \delta_e \rangle -  \mu_k^{t - \ell}   \langle B^\ell \chi_k , \delta_e \rangle} \\
& \leq & (\log n)^{2} \alpha^{t - \ell}   \alpha^{\ell} \alpha^{\ell /2} n^{(1  - \gamma)/2} +  (\log n)^4  \sqrt{ |\vec E |}  \alpha^{t/2} \\
& \leq & (\log n)^{2}  n^{1/2} \alpha^{t /2} \alpha^{\ell } n^{  - \gamma/2} + C (\log n)^4  n^{1/2} \alpha^{t/2}.
\EEAS
Finally, if $k \in [r] \backslash [r_0]$, we simply write \whp
$
\ABS { J }  \leq   \sum_{e \notin \vec E_\ell} |x_e |  \ABS{ \langle B^t \chi_k , \delta_e \rangle}   \leq   (\log n)^4  n^{1/2} \alpha^{t/2}.
$
\end{proof}

\subsection{Laws of large numbers for local functions}
We first prove weak laws of large numbers for general local functionals of SBM random graphs that will then be applied to specific functionals of interest.
\subsubsection{Weak laws of large numbers for local functionals: convergence speed}

We start with a general variance bound for local functions of an inhomogeneous random graph.  A colored graph  is a graph $G = (V,E)$ with a map $\sigma : V \to [r]$. We denote  by $\cG^*$ the set of rooted colored graphs, i.e. the set of pairs $(G,o)$ formed by a colored graph $G$ and  a distinguished vertex $o \in V$. We shall say that a function $\tau$ from $\cG^*$  to $\dR$ is $\ell$-local, if  $\tau (G,o)$ is only function of $(G,o)_\ell$.

\begin{proposition}
\label{prop:varloc}
There exists $c  >0$ such that if $\tau, \varphi : \cG^* \to \dR$ are $\ell$-local, $| \tau(G,o) | \leq  \varphi (G,o)$  and $\varphi$ is non-decreasing by the addition of edges, then 
$$
\VAR\PAR{  \sum_{v=1}^n    \tau ( G, v) } \leq c n  \bar \alpha_n^{2\ell }\PAR{ \dE \max_{v \in [n]} \varphi^4 (G,v)}^{1/ 2}.
$$
\end{proposition}

\begin{proof}
We first bound the expectation of 
$$
Z  =  \sum_{u=1} ^n \Sigma^2 ( G,u),
$$
where $\Sigma(G,u)$ is defined as the number of vertices at distance $\ell$ from $u \in V$.  By Lemma \ref{le:growthSt}, for any $u \in [n]$, 
$$
\dP \PAR{  \Sigma ( G , u ) \geq s \frac{ \bar \alpha_n ^{\ell+1} - 1}{\bar \alpha_n  -1 }  } \leq c_1 e^{-c_0 s}.
$$
Hence, for any $p \geq 1$, for some $c_p >0$, 
$$
 \dE Z^p  \leq n^{p-1} \dE \sum_{u=1} ^n \Sigma^{2p} ( G,u) \leq \ c_p  n^p  \bar  \alpha_n ^{2\ell p }.
$$ 
Now, for $1 \leq k \leq n$,  let $X_k = \{ 1 \leq v \leq k : \{v, k \} \in E \}$, where $E$ is the edge set of $G$. The vector $(X_1, \cdots, X_n)$ is an independent vector and for some function $F$,
$$
Y:= \sum_{u=1}^n \tau (G,u)  =  F ( X_1, \cdots, X_n).
$$
We also define $G_k$ as the graph with edge set $ \cup_{v \ne k}X_v $. We set 
$$
Y_k =   \sum_{u=1}^n  \tau ( G_k,u).
$$ 
Since $\tau$ is $\ell$-local, we observe that $\tau ( G,u) -  \tau ( G_k ,u)$ can be non zero only if $u \in V ( (G, k)_{\ell})$ and it is bounded by $\Lambda =  2 \max_{u \in [n]} \varphi( G,u) $. Consequently, 
$$
  \sum_{k=1} ^ n | Y - Y_k |^2  \leq  \sum_{k=1}^n  \Sigma^2 ( G,k) \Lambda^2 = Z \Lambda^2, 
$$
Finally, we conclude by using Efron-Stein's inequality: $
\VAR(Y) \leq    \dE  Z \Lambda^2 \leq   \sqrt{ \dE  Z^2 }\sqrt{ \dE \Lambda^4}. 
$ \end{proof}

We now apply the above proposition and Proposition \ref{prop:couplingYZ} to show that the SBM random graph with uniform root selection converges weakly to the multi-type Galton-Watson process previously studied. The established convergence implies convergence for the local weak topology of Benjamini and Schramm (see \cite{benjamini-schramm-2001}). Crucially we are able to consider local functions with logarithmic distance parameter $\ell$ and obtain bounds on the convergence speed.


\begin{proposition}\label{prop:locmart0}
Let $\ell \sim \kappa \log_\alpha n$ with $0 < \kappa < 1/2$. There exists $c  >0$, such that if $\tau, \varphi : \cG^* \to \dR$ are $\ell$-local, $| \tau(G,o) | \leq  \varphi (G,o)$  and $\varphi$ is non-decreasing by the addition of edges, then if $\dE \varphi ( T,o)$ is finite, 
\begin{eqnarray*}
\dE\ABS{\frac 1 n   \sum_{v=1}^n    \tau ( G, v)  - \dE \tau ( T, o)  }  & \leq & c   \frac{\alpha^{\ell/2} \sqrt{ \log n}  }{n ^{\gamma/2}}   \PAR{ \PAR{ \dE \max_{v \in [n]} \varphi^4 (G,v) } ^{1/ 4 } \vee \PAR{ \dE \varphi^2  (T,o)}^{1/2}  },
\end{eqnarray*}
where $(T,o)$ is the random rooted tree associated to the Galton-Watson branching process defined in Section \ref{subsec:MTGW} started from $Z_0 = \delta_{\iota}$ and $\iota$ has distribution $(\pi(1), \ldots, \pi(r))$.  
\end{proposition}

\begin{proof}
In view of Proposition \ref{prop:varloc} and Jensen's inequality,  it is sufficient to prove that 
\begin{equation}\label{eq:locmart0}
\ABS{ \frac 1 n   \sum_{v=1}^n    \dE \tau ( G, v) - \dE \tau ( T, o) } = O \PAR{\frac{\alpha^{\ell/2}}{n ^{\beta}} \sqrt{ \log n}  \PAR{ \max_{v \in [n]} \dE \varphi^2 (G,v)   \vee  \dE \varphi^2  (T,o)}^{1/2}},
\end{equation}
 with $\beta =(\gamma \wedge (1 - \kappa))/2$. For $i \in [r]$, we set $\tau_i = \tau ( T_i ,o)$ where $(T_i,o)$ has distribution the the random tree $(T,o)$ tarted from $Z_0 = \delta_{i}$. Let $v \in V$ with $\sigma(v) = i$. We denote by $\chi_v $  the indicator function that the coupling of $(G,v)_{\ell}$ and $(T_i,o)_{\ell}$ described in Proposition \ref{prop:couplingYZ} is not successful. We have, from Cauchy-Schwartz inequality,
\begin{eqnarray*}
 \ABS{ \dE \tau ( G,v)  - \dE \tau_i } & =& \ABS{   \dE  \chi_v \tau ( G,v)  - \dE  \chi_v   \tau_i  } \nonumber \\
 &\leq  & 2  \sqrt { ( \dE\chi_v )  \PAR{ \dE \varphi  (G,v)^2   \vee  \dE \varphi  (T_i,o)^2 }} \nonumber  \\
& = & O \PAR{\alpha^{\ell/2} n ^{-\beta}  \sqrt{ \log n}  \sqrt{  \dE \varphi  (G,v)^2   \vee  \dE \varphi  (T_i,o)^2 }  } .
\end{eqnarray*}
Let $v_1, \ldots v_r$ be fixed vertices such that $\sigma(v_i) = i$ (since $\pi(i) > 0$ such $v_i$ exists for $n$ large enough). We recall that $\dE \tau(G,v)$ depends only on $\sigma(v)$. Hence, using \eqref{eq:defgamma}
\begin{eqnarray*}
\dE \frac{1}{n} \sum_{v \in V}  \tau ( G,v)  & = &  \sum_{i=1}^r\frac{ n(i) }{n} \dE \tau ( G,v_i) \nonumber\\
&= & \sum_{i  = 1}^r  \BRA{(\pi(i) + O (n^{-\gamma})  )  \dE \tau_i +  O \PAR{\alpha^{\ell/2} n ^{-\beta}  \sqrt{ \log n}  \sqrt{  \dE \varphi  (G,v_i)^2   \vee  \dE \varphi  (T_i,o)^2 }  }} \nonumber\\
& =& \dE \tau(T,o) + O \PAR{\alpha^{\ell/2} n ^{-\beta}  \sqrt{ \log n}  \sqrt{  \dE \varphi  (G,v_i)^2   \vee  \dE \varphi  (T_i,o)^2 }  }. 
\end{eqnarray*}
It concludes the proof of \eqref{eq:locmart0}. \end{proof}

\subsubsection{Law of large numbers for specific local functions}

We will now apply Proposition \ref{prop:locmart0} to  deduce weak laws of large numbers for expressions closely related to $\langle B^{\ell} \chi_k , B^{\ell} \chi_j\rangle $, $\langle B^{2\ell} \chi_k , B^{\ell} \chi_j\rangle $ and $\langle B^{\ell} B^{*\ell}\check  \chi_k , B^{\ell} B^{*\ell} \check \chi_j \rangle $. Recall the definition of $Y_t (e)$ in  \eqref{eq:defYti}.

\begin{proposition}\label{th:locmart}
Let $\ell \sim   \kappa \log_\alpha n$ with $0 < \kappa < \gamma/4$. 
\begin{enumerate}[(i)]
\item \label{zz1} 
For any $ k\in [r_0]$, there exists $\rho_k >0$ such that, in probability, 
$$
\frac 1 {\alpha n} \sum_{e \in \vec E}  \frac{\langle \phi_k , Y_{\ell} (e) \rangle^2}{ \mu_k ^{2 \ell}}  \to \rho_k.
$$
\item \label{zz2} 
For any $k \in [r] \backslash [r_0]$, there exists $\rho_k >0$ such that \whp
$$
\frac 1 {\alpha n} \sum_{e \in \vec E} \frac{ \langle \phi_k , Y_{\ell} (e) \rangle^2 }{\alpha ^{ \ell}}  \geq \rho_k.
$$
\item \label{zz3} For any $k \ne j \in [r]$,
$$
\dE \ABS{ \frac 1 {\alpha n} \sum_{e \in \vec E} \langle \phi_k , Y_{\ell} (e) \rangle  \langle \phi_j , Y_{\ell} (e) \rangle} = O ( \alpha^{5 \ell/2} n^{-\gamma/2} (\log n)^{5/2} ) .
$$
\item  \label{zz4} For any $ k \ne j \in [r]$, 
$$
\dE \ABS{\frac 1 {\alpha n} \sum_{e \in \vec E} \langle \phi_k , Y_{2\ell} (e) \rangle    \langle \phi_j , Y_{\ell} (e) \rangle }  = O \PAR{ \alpha^{7\ell/2}n ^{-\gamma/2} (\log n)^{5/2} }.
$$
\item  \label{zz5} For any $ k  \in [r_0]$, in probability
$$
\frac 1 {\alpha n} \sum_{e \in \vec E} \frac{\langle \phi_k , Y_{2\ell} (e) \rangle    \langle \phi_k , Y_{\ell} (e) \rangle}{\mu_k ^ {3\ell}  }  \to  \rho_k.
$$

\end{enumerate}
\end{proposition}

\begin{proof}
Let $Z_t$, $t\geq 0$, be the Galton-Watson branching process defined in Section \ref{subsec:MTGW} started from $Z_0 = \delta_{\iota}$ and $\iota$ has distribution $(\pi(1), \ldots, \pi(r))$. We denote by $(T,o)$ the associated random rooted tree.   If $k \in [r_0]$, by Theorem \ref{th:kestenstigum}, for some $\rho_k >0$, 
$$
\dE \langle \phi_k , Z_\ell \rangle^2 \mu_k ^{-2\ell} =  \rho_k + o(1). 
$$
We set $\tau(G,v) = \sum_{e \in \vec E : e_2 = v} \langle \phi_k , Y_{\ell} (e) \rangle^2 \mu_k ^{- 2 \ell} $. We observe that $$\sum_{e \in \vec E : e_2 = v}   \langle \phi_k , Y_{\ell} (e) \rangle^2 \leq  \sum_{e \in \vec E : e_2 = v} S^2 _{\ell} (e) \leq  S^2_{\ell +1}(v).$$ We get, using $\mu_k^{-2}<\alpha^{-1}$, that $\tau(G,v)  \leq \varphi( G,v) :=  \alpha^{-\ell} S^2_{\ell+1} (v)$. Also,  Lemma \ref{le:growthSt} implies that $\dE \max_v \varphi( G,v)^4  = O ( ( \log n)^8 \alpha^{4\ell})$. The same upper bound holds for  $\varphi(T,o)$ by Lemma \ref{le:growthS}.  We deduce from Proposition \ref{prop:locmart0} that 
$$
\dE \ABS{ \frac 1 {\alpha n} \sum_{e \in \vec E}  \frac{\langle \phi_k , Y_{\ell} (e) \rangle^2}{ \mu_k ^{2 \ell}}  - \frac{\dE \langle \phi_k , Z_\ell \rangle^2 }{\mu_k ^{2\ell}} } = O\PAR{\alpha^{3\ell/2} (\log n)^{5/2} n^{-\gamma/2}} .
$$
It proves statement \eqref{zz1} of the proposition.

For statement \eqref{zz2}, we use Theorem  \ref{th:kestenstigum2} instead. We denote by $X_{k}$ the limit martingale in Theorem  \ref{th:kestenstigum2} when $Z_0 = \delta_\iota$. If $\mu_k ^2 < \mu_1$, we find similarly that  for any $\theta >0$,
$$
 \frac{1}{\alpha n} \sum_{e \in \vec E}  \PAR{ \frac{ \langle \phi_k , Y_{\ell} (e) \rangle }{ \alpha ^{\ell/2}}}^2 \wedge \theta 
$$ 
converges in $L^1$ to  $\dE (  | X_{k} |^2 \wedge \theta ) $. 
From Theorem  \ref{th:kestenstigum2}, the latter is positive if $\theta$ is large enough. In the case $\mu_k ^2 = \mu_1$, by Theorem  \ref{th:kestenstigum2}, we need to normalize by $ \alpha ^{\ell/2}\ell^{1/2}$.

For statement \eqref{zz3}, we use Lemma \ref{mkv_field}. We set $\tau(G,v) := \sum_{e \in \vec E : e_2 = v} \langle \phi_k , Y_{\ell} (e) \rangle  \langle \phi_j , Y_{\ell} (e) \rangle$. As above, we have $\tau(G,v)  \leq \varphi( G,v) =  S^2_{\ell+1} (v)$ and $\dE \max_v \varphi( G,v)^4  = O ( ( \log n)^8 \alpha^{8\ell})$ by Lemma \ref{le:growthSt}. The same upper bound holds for  $\varphi(T,o)$  from Lemma \ref{le:growthS}. It remains to apply Proposition \ref{prop:locmart0}.

For statement \eqref{zz4}, we have $\dE \langle \phi_k , Z_{2\ell}  \rangle    \langle \phi_j , Z_{\ell}  \rangle = \mu_k^{\ell} \dE \langle \phi_k , Z_{\ell}  \rangle    \langle \phi_j , Z_{\ell}  \rangle = 0$ by Lemma \ref{mkv_field}.  We set $\tau(G,v) = \sum_{e \in \vec E : e_2 = v} \langle \phi_k , Y_{2\ell} (e) \rangle  \langle \phi_j , Y_{\ell} (e) \rangle$. We have 
$$\tau(G,v)  \leq \varphi( G,v) =  \sum_{ e \in \vec E : e_2 = v } S_{2\ell} (e) S_{\ell} (e) \leq S_{2\ell+1}(v)S_{\ell+1} (v).$$ 
Moreover, $\dE \max_v \varphi( G,v)^4  = O ( ( \log n)^8 \alpha^{12\ell})$ by Lemma \ref{le:growthSt} and Cauchy-Schwarz inequality. The same upper bound holds for  $\varphi(T,o)$  from Lemma \ref{le:growthS}. It remains to apply Proposition \ref{prop:locmart0}.
 
Finally, for statement \eqref{zz5}, $\dE \langle \phi_k , Z_{2\ell}  \rangle    \langle \phi_k , Z_{\ell}  \rangle =  \mu_k^{\ell} \dE \langle \phi_k , Z_{\ell}  \rangle^2 =  \mu_k^{3\ell} (\rho_k +o(1))$. We use $\mu_k ^2 > \alpha$ and then repeat the proof of statement \eqref{zz4}, we get
$$
\dE \ABS{ \frac 1 {\alpha n} \sum_{e \in \vec E} \frac{\langle \phi_k , Y_{2\ell} (e) \rangle    \langle \phi_k , Y_{\ell} (e) \rangle}{\mu_k ^ {3\ell}  }  -  \rho_k} = O \PAR{\alpha^{7 \ell / 2 - 3 \ell /2 } n ^{-\gamma /2 } (\log n)^{5/2}}.
$$ 
Since $\kappa   < \gamma /4$, the right hand side is $o(1)$. 
\end{proof}

We conclude this subsection by estimates on quantities which are closely related to $B^{\ell} B^{*\ell} \check \chi_k$.  For $e \in \vec E(V)$, we define for $t \geq 0$, $\cY_t (e) = \{ f \in \vec E : \vec d (e,f) = t \}$. For $k \in [r]$, we set
\begin{equation}\label{eq:defPkl}
P_{k,\ell} (e) = \sum_{t = 0} ^{\ell -1} \sum_{f \in \cY_{t} (e) } L_k(f).
\end{equation}
where $$L_k(f) = \sum_{(g,h) \in \cY_1 (f) \backslash \cY_t(e); g \ne h  } \langle \phi_k , \tilde Y_t(g)\rangle \tilde S_{\ell - t -1} (h),$$
and $\tilde Y_t(g)$, $\tilde S_{\ell - t -1}(h) = \|\tilde Y_{\ell - t-1}(h)\|_1$ are the variables $ Y_t(g)$, $S_{\ell - t -1}(h)$ defined on the graph $G$ where all  edges in $(G,e_2)_t$ have been removed. In particular, if $(G,e)_{2\ell}$ is a tree, $\tilde Y_s(g)$ and $Y_s(g)$ coincide for $s \leq 2 \ell -t$.

We also define
\begin{equation}\label{eq:defSkl}
S_{k,\ell} (e) = S_{\ell}(e) \phi_k ( \sigma (e_1)).
\end{equation}

As can be seen from \eqref{eq:QklNBW}, when $(G,e_2)_{2\ell}$ is a tree, it follows that 
\begin{equation}\label{why_Q_kl}
B^{\ell} B^{*\ell} \check \chi_k (e) = P_{k,\ell} (e) + S_{k,\ell} (e). 
\end{equation}
This is the reason why controls of quantities $P_{k,\ell}$ and $S_{k,\ell}$, themselves based on the previous branching analysis of quantities $Q_{k,\ell}$ and $S_{\ell}$ respectively, will be instrumental in our analysis of vectors $B^{\ell} B^{*\ell} \check \chi_k$.  We have the following
\begin{proposition}\label{th:locmart2}
Let $\ell \sim   \kappa \log_\alpha n$ with $0 < \kappa < \gamma/5$.  
\begin{enumerate}[(i)]
\item \label{yy1} 
For any $ k\in [r_0]$, there exists $\rho'_k >0$ such that \whp
$$
\frac 1 {\alpha n} \sum_{e \in \vec E}  \frac{P^2_{k,\ell} (e) }{ \mu_k ^{4 \ell}}  \to  \rho'_k.
$$
\item \label{yy3} For any $k \in [r] \backslash [r_0]$, there exists $C_k >0$ such that \whp
$$
\frac 1 {\alpha n} \sum_{e \in \vec E} \frac{ P^2_{k,\ell}(e)  }{\alpha ^{ 2\ell} (\log n)^5 }  \leq C_k.
$$
\item  \label{yy4} For any $ k \ne j \in [r]$, for some $c> 0$,
$$
\dE \ABS{\frac 1 {\alpha n} \sum_{e \in \vec E} ( P_{k,\ell} (e) + S_{k,\ell} (e) ) ( P_{j,\ell}(e) + S_{k,\ell} (e))  }  = O \PAR{ \alpha^{9 \ell /2 }n ^{-\gamma/2} (\log n)^{c} }.
$$
\end{enumerate}
\end{proposition}
\begin{proof}
Let $Z_t$, $t\geq 0$, be the Galton-Watson branching process defined in Section \ref{subsec:MTGW} started from $Z_0 = \delta_{\iota}$ and $\iota$ has distribution $(\pi(1), \ldots, \pi(r))$. We denote by $(T,o)$ the associated random rooted tree.   For any $k \in [r_0]$, by Theorem \ref{th:growthQkl}, for some positive constant $\rho'_k$, 
\begin{equation}\label{eq:gdueoudge}
 \frac{\dE [Q_{k,\ell}^2  ] }{ \mu_k ^{2 \ell}}  = \rho'_k + o(1).
\end{equation}
On the other hand, Theorem \ref{th:growthQkl} also ensures that for some $C_k > 0$, for $k \in [r]\backslash [r_0]$, 
\begin{equation}\label{eq:gdueoudg} \frac{\dE [Q^2_{k,\ell} ] }{ \alpha ^{2 \ell} (\log n) ^5} \leq 2 C_k.
\end{equation}

Let us denote here by $\cF_t$ the $\sigma$-algebra spanned by $(G,e)_t$, given $\cF_{t+1}$ and $g \in (G,e)_{t+1}$. By a monotonicity argument, the statement of Lemma \ref{le:growthSt} applies to variables $(\tilde S_s (g), s \geq 0)$.  Thus for any $p \geq 1$, there is a constant $c >0$ such that for any integer $s \geq 0$, $\dE [ \tilde S_s (g) ^p | \cF_{t+1} ] \leq c \bar \alpha_n^{sp} $. We  thus have, repeating the proof of Lemma \ref{le:growthQklLp}, that for any fixed $p \geq 1$, 
$$
\dE P_{k,\ell}(e)^p \leq C ( \bar \alpha_n ) ^{2 \ell p} = O ( \alpha^{2\ell p}).
$$

Then the argument in the proof of Proposition \ref{th:locmart} can be applied. For statement \eqref{yy1}, we let  $k \in [r_0]$ and define $\tau(G,v) = \sum_{e \in \vec E, e_2 = v} P^2_{k,\ell} (e)  \mu^{-4 \ell}_k$. Let $$M (v) = \max_{0 \leq t \leq \ell} \max_{u \in (G,v)_t } \max_{s \leq 2 \ell - t}  (S_s (u) / \alpha^s).$$ 
Since $\mu_k^2 > \alpha$, we have the rough bound
\begin{eqnarray*}
\tau(G,v) & \leq & \alpha^{-2 \ell}  \sum_{e \in \vec E, e_2 = v} \PAR{ \sum_{t=0}^{\ell -1} \sum_{f \in \cY_t(e)}   M^2 (v) \alpha^{t+1} \alpha^{\ell -t} }^2 \nonumber \\
& = & \PAR{   M^2 (v) \alpha }^2    \sum_{e \in \vec E, e_2 = v} \PAR{ \sum_{t=0}^{\ell -1} \sum_{f \in \cY_t(e)}  1 }^2 \label{eq:taugv2}
\end{eqnarray*}
Hence $
\tau(G,v) \leq CM(v)^6 \alpha^{2\ell} = \varphi(G,v).
$
However, we have by Lemma \ref{le:growthSt} that $\dE \max_v \varphi( G,v)^4  = O ( ( \log n)^{24}  \alpha^{8\ell})$. By Lemma \ref{le:growthS}, the same bound holds for $\varphi(T,o)$.   We deduce from Proposition \ref{prop:locmart0}  that 
$$
\dE \ABS{  \frac 1 {\alpha n} \sum_{e \in \vec E}  \frac{P^2_{k,\ell} (e) }{ \mu_k ^{4 \ell}}  -  \frac{ \dE Q^2_{k,\ell} }{ \mu_k ^{4 \ell}} } = O \PAR{ \alpha^{5\ell/2}  n^{- \gamma /2} (\log n)^c   }.
$$
Since $\kappa < \gamma / 5$, the right hand side goes to $0$ and statement \eqref{yy1} follows from \eqref{eq:gdueoudge}. Statement \eqref{yy3} is proved similarly using \eqref{eq:gdueoudg}. For statement \eqref{yy4}, we use the above computation, together with Theorem \ref{le:ouff2} and \eqref{eq:phiONpi}. It gives the claimed bound.  \end{proof}

\subsection{Proof of Proposition \ref{prop:Bellphi2}}

For the next lemma, recall the definitions \eqref{eq:defPkl}-\eqref{eq:defSkl} of the vectors $P_{k,\ell}$ and $S_{k,\ell}$ in $\dR^{\vec E}$. We also introduce the vector in $\dR^{\vec E}$, for $k \in [r]$,
$$
N_{k,\ell} (e)=  \langle \phi_k  , Y_{\ell} (e) \rangle. 
$$

\begin{lemma}\label{le:Beta}
Let $\ell \sim \kappa \log_\alpha n$ with $0 < \kappa < \gamma \wedge 1/2$.Then, \whp 
$\| B^{\ell} \chi_k - N_{k,\ell} \| =  O \PAR {  (\log n )^{5/2} \alpha ^{3\ell/2} } = o \PAR{  \alpha^{\ell/2} \sqrt n } $, $ \| B^{\ell} B^{* \ell} \check \chi_k - P_{k,\ell} - S_{k,\ell} \| =  O ( (\log n)^4 \alpha^{3\ell} )  $ and $ \| B^{\ell} B^{* \ell} \check \chi_k - P_{k,\ell} \| = O ( \alpha^{\ell} \sqrt n )$.
\end{lemma}

\begin{proof}Let $\vec E_\ell$ be as in Proposition \ref{prop:localB} and let $\vec E'_{\ell} \subset \vec E_{\ell}$ be the subset of edges such that $ (G,e_2)_{\ell}$ is a not tree. From \eqref{eq:etachi}, Lemma \ref{le:growthSt} and Proposition \ref{prop:localB} we have \whp  
\begin{eqnarray*}
\| N_{k,\ell} - B^{\ell} \chi_k \|^2  &  = & \sum_{e \in \vec E'_\ell}  \PAR{ \langle \phi_k  , Y_{\ell} (e) \rangle  - \langle B^{\ell} \chi_k , \delta_e \rangle }^2 \\
& \leq & 2 \sum_{e \in \vec E'_\ell} \PAR{  S_{\ell} (e) ^2 +  \langle B^{\ell} \chi_k , \delta_e \rangle ^2} \\
&= &
O \PAR{     | \vec E'_\ell|   \log (n)^4  \alpha^{2 \ell}} = O \PAR {  (\log n )^5 \alpha ^{3\ell} },  
\end{eqnarray*}
where at the last line, we have used Lemma \ref{le:tls2}. Since $\kappa < 1/2$, it proves the first statement.

Similarly for the second statement, recall as stated in \eqref{why_Q_kl} that when $(G,e_2)_{2\ell}$ is a tree, then $
B^{\ell} B^{*\ell} \check \chi_k (e) = P_{k,\ell} (e) + S_{k,\ell} (e)$.
Let $\vec E'_{2\ell} \subset \vec E_{2\ell}$ be the subset of edges such that $ (G,e_2)_{2\ell}$ is not a tree. If $G$ is $2\ell$-tangle free then there  are at most two different non-backtracking paths between two edges. Hence, if $e \in \vec E'_{2\ell}$, 
$$
| B^{\ell} B^{*\ell} \check \chi_k (e) | \leq 2  \| \phi_k \|_{\infty}( P_{1,\ell} (e) + S_\ell (e)) \leq 2 ( P_{1,\ell} (e) + S_\ell (e)).
$$
Now, by Lemma \ref{le:growthSt}, \whp $S_{\ell} (e) \leq C (\log n)  \alpha^{\ell}$. Moreover, if $M = \max_{v , t \leq \ell} S_{t} (v) / \alpha^t \leq  C \log n$,  $P_{1,\ell} (e)  \leq \sum_{t = 0} ^{\ell-1} \sum_{f \in \cY_t (e)} \alpha^{t+1} \alpha^{\ell - t} M^2  \leq M^3 \sum_{t = 0}^{\ell -1} \alpha^{t + \ell +1} $. So finally, \whp for all $e\in \vec E'_{2 \ell}$, 
$$
| B^{\ell} B^{*\ell} \check \chi_k (e) | = O ( (\log n)^3 \alpha^{2\ell} ). 
$$
 Hence, by Lemma \ref{le:tls2}, \whp, 
\begin{equation*}\label{eq:nrmBB}
\| B^{\ell} B^{*\ell} \check \chi_k - P_{k,\ell} - S_{k,\ell} \| = O ( \sqrt{| E'_{2\ell}|} (\log n)^3 \alpha^{2\ell} ) = O ( (\log n)^4 \alpha^{3\ell} ) . 
\end{equation*}
On the other hand,  from Proposition \ref{th:locmart}\eqref{zz1}, \whp 
\begin{equation}\label{eq:nrmveps}
\| S_{k,\ell} \|  = O ( \sqrt n \alpha^{\ell} ).
\end{equation} 
The conclusion follows since $\kappa < 1/2$.\end{proof}

All ingredients are finally gathered to prove Proposition \ref{prop:Bellphi2}. 

\begin{proof}[Proof of Proposition \ref{prop:Bellphi2}]  We use the notation of Lemma \ref{le:Beta}.

\noindent{\em{Proof of \eqref{i1}. }} Let $k \in [r_0]$. By definition 
$$
\theta_k = \frac{ \| B^{\ell} B^{*\ell} \check \chi_k \|}{\| B^{\ell} \chi_k \|}.
$$
From Proposition  \ref{th:locmart}\eqref{zz1} and Proposition \ref{th:locmart2}\eqref{yy1} respectively, for some positive constants $c_0, c_1$, \whp 
$$\frac{c_0}{2}   \leq \frac{  \| N_{k,\ell} \| }{\sqrt n  \mu_k ^{\ell}}  \leq 2 c_1  \quad  \hbox{ and } \quad \frac {c_0}{2}  \leq \frac{\| P_{k,\ell} \| }{ \sqrt n  \mu_k ^{2\ell}} \leq 2 c_1  .$$ It remains to use Lemma \ref{le:Beta} and the assumption $\mu_k ^2 > \alpha$. We find \whp
\begin{equation}\label{eq:normeBBB}
c_0 \leq \frac{\| B^{\ell} \chi_k \| }{ \sqrt n  \mu_k ^{\ell}  } \leq c_1  \quad  \hbox{ and } \quad  c_0 \leq \frac{ \| B^{\ell} B^{*\ell} \check \chi_k \|}{ \sqrt n  \mu_k ^{2\ell}} \leq c_1.
\end{equation}

\noindent{\em{Proof of \eqref{i2}. }}
Let $k \in [r] \backslash [r_0]$. From  Proposition  \ref{th:locmart}\eqref{zz2} \whp $\| N_{k,\ell} \| \geq (c_0/2) \sqrt n \alpha^{\ell/2}$ and from Proposition  \ref{th:locmart2}\eqref{yy3} \whp $ \| P_{k,\ell} \| = O ( \sqrt n (\log n)^{5/2} \alpha^{\ell} ) $. Using Lemma \ref{le:Beta}, we find \whp
\begin{equation}\label{eq:normeBBB2}
c_0 \leq \frac{\| B^{\ell} \chi_k \| }{ \sqrt n  \alpha ^{\ell/2}  }  \quad  \hbox{ and } \quad   \frac{ \| B^{\ell} B^{*\ell} \check \chi_k \|}{ \sqrt n ( \log n)^{5/2} \alpha ^{\ell}} \leq c_1.
\end{equation}

\noindent{\em{Proof of \eqref{i3}. }}
Let $k \in [r_0]$. Since $Px = \check x$ is an isometry,
$$
\langle \zeta_k , \check \varphi_k \rangle = \frac{\langle  B^{\ell} B^{*\ell}\check \chi_k  , B^{*\ell}  \check \chi_k \rangle}{ \| B^{\ell} B^{*\ell}\check \chi_k \|  \| B^{*\ell}  \check \chi_k \| } =  \frac{\langle  B^{\ell} \chi_k  , B^{2\ell}   \chi_k \rangle}{ \|  B^{\ell} B^{*\ell}\check \chi_k \| \|  B^{\ell}  \chi_k \|}
$$
In view of \eqref{eq:normeBBB}, it is sufficient to prove that for some $c_0 >0$, \whp $\langle  B^{\ell} \chi_k  , B^{2\ell}   \chi_k \rangle > c_0 \mu_k^{3\ell} n$. Note then that
\begin{eqnarray}\label{eq:uiuo}
\ABS{ \langle  B^{\ell} \chi_k  , B^{2\ell}   \chi_k \rangle -  \langle  N_{k,\ell}  ,  N_{k,2 \ell} \rangle }&  \leq&  \| B^{\ell} \chi_k  \| \| B^{2\ell}   \chi_k  - N_{k,2 \ell} \|  + \| B^{\ell} \chi_k   - N_{k,\ell} \| \| N_{k,2 \ell} \|.
\end{eqnarray}
However, from \eqref{eq:normeBBB} and Lemma \ref{le:Beta}, we have \whp $ \| B^{\ell} \chi_k  \| \| B^{2\ell}   \chi_k  - N_{k,2 \ell} \| = o ( \mu_k ^{3\ell} n)$. Also, from Proposition  \ref{th:locmart}\eqref{zz1}, and  Lemma \ref{le:Beta},  we find \whp $\| B^{\ell} \chi_k   - N_{k,\ell} \| \|  N_{k,2 \ell} \| = o ( \mu_k ^{3\ell} n)$. So finally, \whp 
$$
\ABS{ \langle  B^{\ell} \chi_k  , B^{2\ell}   \chi_k \rangle -  \langle  N_{k,\ell}  ,  N_{k,2 \ell} \rangle } = o ( \mu_k ^{3\ell} n).
$$
On the other hand, by Proposition  \ref{th:locmart}\eqref{zz5}, $\langle  N_{k,\ell}  ,  N_{k,2 \ell} \rangle$ is \whp larger than $c_0 \mu_{k}^{3\ell} n$ for some $c_0 >0$. It concludes the proof   \eqref{i3}.

\noindent{\em{Proof of \eqref{i4}. }}
Let $\bar \mu_k = \mu_k \vee \sqrt \alpha$. From \eqref{eq:check} and \eqref{eq:normeBBB}-\eqref{eq:normeBBB2} for $k , j \in  [r]$, \whp
$$
\ABS{\langle \check \varphi_j , \check \varphi_k \rangle } =  \frac{\ABS{ \langle  B^{\ell} \chi_j  , B^{\ell}   \chi_k \rangle} }{ \|  B^{\ell}   \chi_j  \| \|  B^{\ell}  \chi_k \|} \leq  \frac{\ABS{\langle  B^{\ell} \chi_j  , B^{\ell}   \chi_k \rangle}}{ c^2_0 n  \bar \mu_j ^ {\ell}  \bar \mu_k ^ {\ell}} \leq  \frac{\ABS{\langle  B^{\ell} \chi_j  , B^{\ell}   \chi_k \rangle}}{ c^2_0 n   \alpha^ {\ell}}.
$$
In addition, Equations \eqref{eq:normeBBB}-\eqref{eq:normeBBB2}, Proposition  \ref{th:locmart}\eqref{zz1} and Lemma \ref{le:Beta} entail that \whp
\begin{eqnarray*}
\ABS{ \langle  B^{\ell} \chi_j  , B^{\ell}   \chi_k \rangle  -  \langle  N_{j,\ell}  ,  N_{k,\ell} \rangle } \leq   \| B^{\ell} \chi_j  \| \| B^{\ell}   \chi_k  - N_{k,\ell} \|  + \| B^{\ell} \chi_j   - N_{j,\ell} \| \|  N_{k,\ell} \| =  O  \PAR{  \alpha^{5 \ell/2} (\log n)^{5/2} \sqrt n   }. 
\end{eqnarray*}
From Proposition  \ref{th:locmart}\eqref{zz3}, we get if $k \ne j$ that \whp
$$
\ABS{ \langle \check \varphi_j , \check \varphi_k \rangle} = O \PAR{  \alpha^{3\ell /2} n^{-\gamma/2} (\log n)^{5/2} }.
$$
\noindent{\em{Proof of \eqref{i5}. }} Let $k \ne j  \in [r_0]$. From \eqref{eq:check} and \eqref{eq:normeBBB}, \whp
$$
\ABS{\langle \zeta_j, \check \varphi_k \rangle } =  \frac{\ABS{ \langle  B^{\ell} \chi_j  , B^{2\ell}   \chi_k \rangle} }{ \|  B^{\ell} B^{*\ell}\check \chi_k \| \|  B^{\ell}  \chi_k \|} \leq  \frac{\ABS{ \langle  B^{\ell} \chi_j  , B^{2\ell}   \chi_k \rangle} }{ c_0^2 n  \alpha^ {3\ell/2}}.
$$
As in \eqref{eq:uiuo}, we use $\ABS{ \langle x , y \rangle -  \langle x' , y' \rangle } \leq \|x'\| \| y - y ' \| + \|y \| \| x - x ' \|$. We find from \eqref{eq:normeBBB}, Proposition \ref{th:locmart}\eqref{zz1} and Lemma \ref{le:Beta} that  \whp
$$
\ABS{\langle  B^{\ell} \chi_j  , B^{2\ell}   \chi_k \rangle -\langle  N_{j,\ell}  , N_{k,2 \ell} \rangle } = O ( \alpha^{4\ell} \sqrt n  (\log n)^{5/2}). 
$$
Also, from Proposition  \ref{th:locmart}\eqref{zz4}, \whp $\langle  N_{j,\ell}  ,  N_{k,2 \ell} \rangle$ is $O( n ^{1  - \gamma/2} \alpha^{7\ell/2} (\log n)^{5/2})$. 
We conclude finally that 
$$
\ABS{\langle \zeta_j, \check \varphi_k \rangle }  =  O \PAR{ \alpha^{2\ell} n^ {-\gamma/2} (\log n)^{5/2} }. 
$$
\noindent{\em{Proof of \eqref{i6}. }} 
We again use the same argument. Let $k\ne j  \in [r_0]$. From \eqref{eq:normeBBB},
$$
\ABS{\langle \zeta_k , \zeta_j \rangle} = \frac{\ABS{\langle  B^{\ell} B^{*\ell}\check \chi_k  , B^{\ell}B^{*\ell}  \check \chi_k \rangle} }{ \| B^{\ell} B^{*\ell}\check \chi_k \|  \| B^{\ell} B^{*\ell}\check \chi_j \|  } \leq   \frac{\ABS{\langle  B^{\ell} B^{*\ell}\check \chi_k  ,B^{\ell} B^{*\ell}  \check \chi_k \rangle} }{ c_0^2 n  \alpha^ {2\ell}}.$$
Recall that $\ABS{ \langle x , y \rangle -  \langle x' , y' \rangle } \leq \|x'\| \| y - y ' \| + \|y \| \| x - x ' \|$. Then from \eqref{eq:normeBBB}, Proposition \ref{th:locmart2}\eqref{yy1} and Lemma \ref{le:Beta}, we obtain \whp
$$\ABS{ \langle B^{\ell} B^{* \ell} \check \chi_k  , B^{\ell} B^{* \ell} \check \chi_j  \rangle -  \langle P_{k,\ell}  + S_{k,\ell}, N_{j,\ell} + S_{k,\ell} \rangle } = O ( (\log n)^4   \alpha^{5\ell} \sqrt n).$$ 
Finally, from Proposition \ref{th:locmart2}\eqref{yy4}, $\langle P_{k,\ell}  + S_{k,\ell}, N_{j,\ell} + S_{k,\ell} \rangle$ is $O (n ^{1 - \gamma/2} \alpha^{9\ell/2} (\log n) ^{5/2} ) $. 
\end{proof}

%% file: NBS_Norm_NBM.tex
\section{Norm of non-backtracking matrices}
\label{sec:norm_nbm}
In this section we prove Proposition \ref{prop:Bellx2}. The argument used for Erd\H{o}s-R\'enyi graphs extends rather directly to the stochastic block model.

\subsection{Decomposition of $B^{\ell}$}

In this paragraph, we essentially repeat the argument of subsection \ref{subsec:decomp}. We define, for $u \ne v \in V$, the centered variable, 
$$
\underline A_{uv} = A_{uv} - W_{\sigma(u) \sigma(v)}.
$$
We now re-define $K$ as the weighted non-backtracking matrix on the complete graph on $V$, for $e,f \in \vec E (V)$,
$$
K_{e f} = \IND ( e \to  f ) W_{\sigma(e_1) \sigma(e_2)},  
$$
where $e \to f$ represents the non-backtracking property, $e_2 =  f_1$ and $e \ne f^{-1}$. We also introduce
$$
K^{(2)}_{ef}  = \IND ( e \stackrel{2}{\to}  f ) W_{\sigma(e_2) \sigma(f_1)},  
$$
where $e \stackrel{2}{\to}  f$ means that there is a non-backtracking path with one intermediate edge between $e$ and $f$.  We define $\Delta^{(\ell)}$, $B^{(\ell)}$ as in subsection \ref{subsec:decomp}. $R^{(\ell)}_t$ is now defined as
$$
(R_t ^{(\ell)} )_{ef}   =  \sum_{\gamma \in F^{\ell+1} _{t,e f}} \prod_{s=0}^{t-1} \underline A_{\gamma_{s} \gamma_{s+1}}  W_{\sigma(\gamma_t) \sigma(\gamma_{t+1})} \prod_{t+1}^\ell A_{\gamma_{s} \gamma_{s+1}},
$$
where the set of paths $F^{\ell+1} _{t,e f}$ is still defined as in Section \ref{subsec:decomp}.
We again use the  decomposition 
\begin{eqnarray*}
B^{(\ell)} _{e f} &  =  & \Delta^{(\ell)}_{ef}  + \sum_{\gamma \in F^{\ell+1} _{e f}}    \sum_{t = 0}^\ell \prod_{s=0}^{t-1} \underline A_{\gamma_{s} \gamma_{s+1}} \PAR{ \frac {W_{\sigma(\gamma_t) \sigma(\gamma_{t+1})} } {n  } }\prod_{t+1}^\ell A_{\gamma_{s} \gamma_{s+1}}
\end{eqnarray*}
to obtain
\begin{equation}\label{eq:decompBk2}
B^{(\ell)}  =  \Delta^{(\ell)}   +  \frac 1 n K B^{(\ell-1)}  +  \frac 1 n  \sum_{t = 1} ^{\ell-1}  \Delta^{(t-1)} K^{(2)}  B^{(\ell - t -1)}  +  \frac 1 n \Delta^{(\ell-1)} K  -   \frac \alpha n    \sum_{t = 0}^\ell R^{(\ell)}_t.
\end{equation}
We introduce
$$
\bar W = \sum_k \mu_k  \chi_k \check \chi_k ^* \quad \hbox{ and } \quad L = K^{(2)} - \bar W.
$$
It now follows from \eqref{eq:decompBk2} that when $G$ is $\ell$-tangle-free,
\begin{eqnarray}
\| B^{\ell} x \| &   \leq &  \|  \Delta^{(\ell)} \|   +    \frac{ 1}{  n}  \| K  B^{(\ell-1)} \|   +    \frac 1  n  \sum_{j=1} ^{r} \mu_j  \sum_{t = 1} ^{\ell-1}    \| \Delta^{(t-1)} \chi_j  \| \ABS{ \langle \check \chi_j , B^{\ell-t-1} x \rangle }  \nonumber \\
& & \quad +  \;  \frac{1}{n}    \sum_{t = 1} ^{\ell-1} \| S_t^{(\ell)} \|   +   \| \Delta^{(\ell-1)}  \|   +  \frac{\alpha}{n}   \sum_{t = 0}^\ell \| R^{(\ell)}_t \|\label{eq:decompBkx2},
\end{eqnarray}
where we have again let $S_t^{(\ell)}:=\Delta^{(t-1)}LB^{(\ell-t-1)}$ as in Section \ref{subsec:decomp}.
We will now upper bound the above expression over all $x$ such that $\langle \check \chi_j , B^{\ell} x \rangle = 0$. 

\subsection{Proof of Proposition \ref{prop:Bellx2}}

The proof of Proposition \ref{prop:Bellx2} parallels that  of Proposition \ref{prop:Bellx}. 

The main task is to adapt the arguments of Section \ref{sec:path_counts}  to bound  the norms $\| \Delta^{(t)} \|$, $\| \Delta^{(t)} \chi_k \|$, $\| R^{(\ell)}_t \| $, $\|B^{(t)} \|$, $\| K B^{(t)} \|$, $\|S^{(\ell)}_t \|$. We only the two main differences. 

First, the expressions \eqref{eq:expAunder}-\eqref{eq:expAunder2}-\eqref{eq:expA} now depend on the types of the vertices involved in a path. We treat for example the case of \eqref{eq:expAunder} needed for Proposition \ref{prop:normDelta}. We claim that  if $\gamma \in W_{k,m}$ is a canonical path with $e $ edges and $v$ vertices, 
\begin{equation}\label{eq:expAunder000}
\frac{1}{ n^v}  \sum_{ \tau } \dE \prod_{i=1} ^{2m} \prod_{s=1} ^{k} \underline A_{\tau(\gamma_{i,s-1}), \tau (\gamma_{i,s}) } \leq \PAR{ \frac{\bar  \alpha_n}{n} }^{v-1} \PAR{ \frac{a}{n} }^{e - v + 1}, 
\end{equation}
where the sum is over all injections $\tau : [k] \to [n]$, $a = \max_{i,j} W_{ij}$ and $\bar   \alpha_n = \alpha + O (n^{-\gamma})$ is defined in \eqref{eq:defalphan}. Indeed, we consider a spanning tree of $\gamma$, for the $e- v + 1$ edges not present in the spanning tree, we use the bound, $\dE \underline A_{uv} ^p \leq W_{\sigma(u) \sigma(v)} / n  \leq a / n$ for any $p \geq 1$ and $u,v \in [n]$. For the remaining $v-1$ edges, we take a leaf, say $l$, of the spanning tree of $\gamma$, and denote its unique neighbor by $g$. Then, the injection $\tau : [k] \to [n]$ will give a label say $i = \sigma( \tau (g))$ to $g$ and $j  = \sigma( \tau (l))$ to $l$. We use the bound that for any $p \geq 1$ and $i \in [n]$, $ \sum_{j=1} ^r n(j) W_{ij}  / n \leq \bar   \alpha_n$. Hence, summing over all possible values of $\tau(l)$ while fixing $\tau(q), q \ne l$, gives a factor of at most  $ \alpha_n / n $ in \eqref{eq:expAunder000}. We then remove $l$ from the spanning tree and repeat this procedure $v-1$ times, this yields \eqref{eq:expAunder000}.

With \eqref{eq:expAunder000} in place of \eqref{eq:expAunder}, using Lemma \ref{le:enumpath} we then bound $S$ given by \eqref{eq:boundS} as follows 
\begin{eqnarray*}
S & \leq & \sum_{v=3}^{k m +1} \sum_{e = v - 1} ^{ km } | \cW _{k,m} (v,e) | \PAR{ \frac{ \bar  \alpha_n}{n} }^{v-1}  \PAR{ \frac{a}{n} }^{e - v + 1} n ^v \\
& \leq & n  \bar  \alpha_n^{km}\sum_{v=3}^{k m +1} \sum_{e = v - 1} ^{ km } (2 k  m )^{ 10 m ( e -v +1) + 8 m} a^{e-v+1}  n ^{v - e -1} \\
& \leq &  n \bar   \alpha_n^{km} ( 2 \ell m )^{ 8m } (\ell m) \sum_{s = 0} ^{ \infty } \PAR{ \frac{ (2 \ell a m)^{10m }}{n}}  ^{ s}. 
 \end{eqnarray*}
In the range of $k \leq \ell$ and $m$ defined by \eqref{eq:choicem}, we have $\bar   \alpha_n^{km} = \alpha^{km} ( 1+ o(1))$. We deduce that the bound \eqref{prop:normDelta} on $\|\Delta\|$ continues to hold for the stochastic block model. Similarly, by the same adaptation of \eqref{eq:expAunder2}-\eqref{eq:expA}, we find that bounds \eqref{prop:normR} and \eqref{prop:normB2} on $\|R_t^{(\ell)}\|$ and $\|B^{(t)}\|$ continue to hold for the stochastic block model.

The second difference lies in the definition of the matrix $L = K^{(2)} - \bar W$. For the stochastic block model, from \eqref{eq:Wspec}, the entry $L_{ef}$ is zero unless $e = f$, $e \to f$, $f^{-1} \to e$ or $e \to f^{-1}$. Moreover, the non-zero entries are bounded by $a$. Then the argument of the proof of bound \eqref{prop:normS} carries over easily. 

With bounds \eqref{prop:normDelta}-\eqref{prop:normR}-\eqref{prop:normB2}-\eqref{prop:normS} available for the stochastic block model, the remainder of the proof of Proposition \ref{prop:Bellx2} repeats the argument of subsection \ref{subsec:proofBellx}.

%% file: NBS_vecteurs.tex
\section{Stochastic Block Model : proof of Theorem \ref{th:vecteurs}}
\label{sec:vecteurs}

The strategy of proof is based on the following lemma which asserts that the existence of a Boolean function non constant over the classes ensures the existence of an estimation with  asymptotically  positive overlap. 

\begin{lemma}\label{le:ncov}
Assume that $\pi(i) \equiv 1/r$ and there exists a function $F : V \to \{0,1\}$ of the graph $G$ such that in probability, for any $i \in [r]$,
$$
 \lim_{n\to \infty} \frac{1}{n} \sum_{v = 1}^n  \IND_{\BRA{\sigma(v) = i} }  F(v) = \frac{ f(i)}{r}, 
$$
where $f : [r] \to [0,1]$ is not  a constant function (there exists $(i,j)$ such that $f(i) \ne f(j)$). Let $(I^+,I^-)$ be a partition of $[r]$, such that $0<  |I^+| < r$ and
\begin{equation}\label{eq:balanceIpm}
\frac{1}{|I^+|} \sum_{i \in I^+}  f(i) > \frac{1}{|I^-|} \sum_{i \in I^-}  f(i).
\end{equation}
Then the following estimation procedure yields asymptotically positive overlap with permutation $p$ in \eqref{eq:defoverlap} equal to the identity : assign to each vertex $v$ a label $\hat \sigma(v)$ picked uniformly at random from $I^+$ if $F(v) = 1$ and from $I^-$ if $F(v) = 0$.
\end{lemma}

Observe that the existence of a non trivial partition $(I^+,I^-)$ satisfying \eqref{eq:balanceIpm} is implied by the assumption that $f$ is not constant. 

\begin{proof}[Proof of Lemma \ref{le:ncov}]
Let $j\in I^+$ and $v \in V$ such that $\sigma( v) = j$, then, given the realization of the graph, the event $\hat \sigma (v) = \sigma(v)$ is 
equal to $
F(v) \veps_v, 
$
where $\veps_v$ is an independent Bernoulli $\{0,1\}$-random variable with $\dP ( \veps_v = 1) = 1/ |I^+|$. From the law of large numbers, we deduce that, in probability, 
$$
\frac 1 n \sum_{v = 1}^n \IND_{\si(v) = j}   \IND_{\hat \sigma (v) = \sigma(v) }  \to \frac{f(i)}{ r |I^+|} 
$$
Summing over all $j \in I^+$, in probability, 
$$
\frac 1 n \sum_{v = 1}^n \IND_{\si(v)  \in I^+ }   \IND_{\hat \sigma (v) = \sigma(v) }  \to   \frac { f_+}{r},
$$
where $f_+$ is  the left hand side of \eqref{eq:balanceIpm}. Similarly, if $f_-$ is  the right hand side of \eqref{eq:balanceIpm}, in probability, 
$$
\frac 1 n \sum_{v = 1}^n \IND_{\si(v)  \in I^- }   \IND_{\hat \sigma (v) = \sigma(v) }  \to \frac { 1 - f_-}{r}.
$$
Finally, in probability 
$$
\frac 1 n \sum_{v = 1}^n    \IND_{\hat \sigma (v) = \sigma(v) }  - \frac 1 r \to  \frac 1 r \PAR{ f_+ + 1 -  f_- - 1 } = \frac{1} r \PAR{  f_+ - f_-} > 0,
$$
where the strict inequality comes from  (\ref{eq:balanceIpm}). 
\end{proof}

Our aim is now to find a non constant functions over the classes which depends on the eigenvector $\xi_k$. To this end, we introduce a new random variable, for $v \in V$, 
$$
I_{k,\ell} (v) = \sum_{e \in \vec E : e_2 = v} P_{k,\ell} (e),
$$
where $P_{k,\ell}$ was defined by \eqref{eq:defPkl}. Our first lemma is an extension of Proposition \ref{th:locmart2}. 

\begin{lemma}\label{le:locmartv}
Let $\ell \sim   \kappa \log_\alpha n$ with $0 < \kappa < \gamma \wedge 1/2$, $k \in [r_0]$ and $i \in [r]$. There exists a  random variable $Y_{k,i}$ such that $\dE Y_{k,i} = 0$, $\dE |Y_{k,i}| < \infty$ and for any continuity point $t$ of the distribution of $|Y_{k,i}|$, in $L^2$, $$
\frac 1 {n} \sum_{v = 1}^n \IND_{\BRA{\sigma(v) = i} }  \IND_{\BRA{\ABS{ I_{k,\ell} (v) \mu_k ^{- 2\ell}  -  \alpha \mu_k\phi_k (i) / ( \mu_k ^2 / \alpha -1)  } \geq t }} \to  \pi (i)  \dP (  | Y_{k,i} | \geq t ).
$$
\end{lemma}
\begin{proof}
Let $Z_t$, $t\geq 0$, be the Galton-Watson branching process defined in Section \ref{subsec:MTGW} started from $Z_0 = \delta_{\iota}$ and $\iota$ has distribution $(\pi(1), \ldots, \pi(r))$. We denote by $(T,o)$ the associated random rooted tree.  Let $D$ be the number of offspring of the root and for $1 \leq x \leq D$,  let $Q_{k,\ell}(x) $ be the random variable $Q_{k,\ell}$ defined on the tree $T^x$ where the subtree attached to $x$ is removed and set 
$$
J_{k,\ell} = \sum_{x=1}^D Q_{k,\ell} (x).
$$ 
We observe that $$
\dE J_{k,\ell} = \sum_{n=0}^\infty \frac{\alpha^n e^{-\alpha} }{ n !} n \dE [ Q_{k,\ell} | D = n-1] =\sum_{n=1}^\infty \frac{\alpha^n e^{-\alpha} }{ (n-1) !} \dE [ Q_{k,\ell} | D = n-1]  =  \alpha \dE  Q_{k,\ell}. 
$$
Also, by Theorem  \ref{th:growthQkl}, the variable $Q_{k,\ell}  \mu_k ^{-2\ell} -  \mu_k\phi_k (\iota)/(\mu_k ^2 / \alpha -1)$ converges in $L^2$ to a centered variable $X_{k}$ satisfying $\dE X_{k}^2 \leq C$. However, the variables $Q_{k,\ell}$ and $J_{k,\ell}$ are closely related, indeed,
$$
J_{k,\ell} = (D-2) L^o_{k,\ell} + (D-1) \sum_{t=1}^\ell \sum_{u \in Y^o_t} L^u _{k,\ell} = (D-1) Q_{k,\ell}  -  L^o_{k,\ell},
$$
where $L^{u}_{k,\ell}$ was defined above \eqref{eq:defQkl}. The inequality \eqref{eq:varLu} for $t =0$, shows that $\dE  | L^o_{k,\ell} |^2 = O ( \alpha^{2\ell})$. Hence,  $L^o_{k,\ell} / \mu_k ^{2\ell}$ converges in $L^2$ to $0$. From Theorem \ref{th:growthQkl}, we find that $J_{k,\ell}   \mu_k ^{-2\ell}  - \alpha  \mu_k\phi_k (\iota)/(\mu_k ^2 / \alpha -1)$ converges weakly to a centered variable $Y_{k}$ satisfying $\dE  |Y_{k}| \leq C$. 
In particular, if $t$ is a continuity point of $|Y_{k}|$, $\IND ( \sigma(o) = i ) \IND (|J_{k,\ell} \mu_k ^{- 2\ell}  -  \alpha \mu_k\phi_k ( i  )/(\mu_k ^2 / \alpha -1) | \geq t )$ converges weakly to $ \IND (\sigma(o) = i) \IND (| Y_{k} | \geq t)$.
It then remains to apply Proposition \ref{prop:locmart0}.
\end{proof}

Our second lemma checks that we may replace $P_{k,\ell}$ in the above statement by the 
eigenvector $\xi'_k$ properly renormalized to be asymptotically close to \eqref{eq:defpseudvec}. More precisely, we set 
$$
I_{k} (v) = \sum_{e : e_2 = v}  s \sqrt { n} \xi'_k (e),
$$
where $s = \sqrt{ \alpha \rho'_k}$ and $\rho'_k$ was defined in Proposition \ref{th:locmart2}. 

\begin{lemma}\label{le:locmartvv}
Let $k \in [r_0]$, $i \in [r]$ and $Y_{k,i}$ be as in Lemma \ref{le:locmartv}. For any continuity point $t$ of the distribution of $|Y_{k,i}|$, in $L^2$, $$
\frac 1 {n} \sum_{v = 1}^n \IND_{\BRA{\sigma(v) = i} }  \IND_{\BRA{\ABS{ I_{k} (v) -  \alpha \mu_k\phi_k (i) / ( \mu_k ^2 / \alpha -1)  } \geq t }} \to  \pi (i)  \dP (  | Y_{k,i} | \geq t ).
$$
\end{lemma}

\begin{proof}
From Proposition \ref{th:locmart2}\eqref{yy1} we find that, in probability,
$$
\frac{1}{  n} \sum_{e \in \vec E} \frac{P_{k,\ell}^2}{\mu_{k}^{4\ell}} \to s^ 2. 
$$ 
We set $\tilde \xi_k(e) = P_{k,\ell} / ( \mu_{k}^{2\ell} s \sqrt n) $.  From Lemma \ref{le:Beta}, we have \whp
$$
\NRM{ \tilde \xi_k  - \zeta_k  } = O ( \alpha^{\ell} \mu_k^{-2\ell} ) = o (1).
$$
Also, from Theorem \ref{th:main}, \whp
$$
\NRM{ \zeta_k - \xi'_k } = o(1). 
$$ 
Hence, from the triangle inequality, \whp
$$
\NRM{ \tilde \xi_k  - \xi'_k  }  = o (1 ) . 
$$
We deduce from Cauchy-Schwartz inequality that \whp
$$
\frac{1} n \sum_{v=1}^n  \ABS{  I_k (v)   -  \frac{I_{k,\ell} (v)}{\mu_k ^ {2 \ell} }}  \leq \frac{s\sqrt n}{n}  \sum_{e \in \vec E}  \ABS{  \xi'_k (e)   -   \tilde \xi_k (e)  } \leq  \frac{s \sqrt { n  | \vec E |}}{n} \NRM{ \tilde \xi_k  - \xi'_k  }  = o(1).  
$$
Since $t$ is a continuity point of $|Y_{k,i}|$, it is then a routine to deduce Lemma \ref{le:locmartvv} from Lemma \ref{le:locmartv}. 
\end{proof}

All ingredients are now gathered to prove Theorem \ref{th:vecteurs}. We fix $k \in [r_0]$ as in Theorem \ref{th:vecteurs} and let $\xi'_k$ be as above.  We set 
$$J^+:=\{i\in[r]: \phi_k(i)>0\} \quad \hbox{ and } \quad J^-=[r]\setminus J^+.$$
From Lemma \ref{le:locmartvv}, there exist random variables $X_j$, $j\in [r]$ on $\dR$ such that $\dE X_j= \alpha \mu_k\phi_k (j) / ( \mu_k ^2 / \alpha -1)  $ and the following holds  for all $j\in[r]$. With $I_k$ as above, for all $t\in\dR$ that is a continuity point of the distribution of $X_j$, the following convergence in probability holds:
\begin{equation}\label{overlap0}
\lim_{n\to\infty}\frac{1}{ n}\sum_{v = 1}^n \IND_{\si(v) = j} \IND_{ I_k(v)>t}=\pi(j)\dP(X_j>t).
\end{equation}
Write, for $\veps  = \pm$, $\pi_\veps=\sum_{j\in I^\veps}\pi(j)$, $g_\veps=\sum_{j\in J^\veps} \pi(j)\phi_k(j)$.  Note that $g_+>0$ by definition of $J^+$. Also by the orthogonality relation (\ref{eq:phiONpi}) between $\phi_1=\IND$ and $\phi_k$ we obtain $g_+ + g_- =0$, so that $g_-<0$. For $ \veps = \pm$, we shall denote by $X_\veps$ the random variable obtained as a mixture of the $X_j$ for $j\in J^\veps$, with weights $\pi(j)/\pi_\veps$. Note that $X_\veps$ has mean $g_\veps /\pi_\veps$. 

We now establish the existence of $t_0 \in \dR$ that is a continuity point of the distribution of both $X_+$ and $X_-$, and such that
\begin{equation}\label{overlap1}
\dP(X_+> t_0 )>\dP(X_-> t_0 ).
\end{equation}
To this end, since $\dE X_+ > 0$, we write
$$
\int_0^{+\infty}\dP(X_+>t)dt > \int_0^{+\infty}\dP(-X_+ > t)dt  =   \int_0^{+\infty}\dP(-X_+\geq t)dt.
$$
The same argument yields 
$$
\int_0^{+\infty}\dP(X_- >t)dt<\int_0^{+\infty}\dP(-X_-\geq t)dt.
$$
Combined, these two inequalities imply
$$
\int_0^{+\infty}\left\{[\dP(X_+>t)- \dP(X_->t)]+[\dP(X_+>-t)- \dP(X_->-t)] \right\}dt>0.
$$
Thus there is a subset of $\dR_+$ of positive Lebesgue measure on which either $\dP(X_+>t)> \dP(X_->t)$ or $\dP(X_+>-t)> \dP(X_->-t)$. This implies the existence of a continuity point $t_0 \in\dR$ of both $X_+$, $X_-$, $-X_+$ and $-X_-$ such that~(\ref{overlap1}) holds.

We may now come back to the eigenvector $\xi_k$ in Theorem \ref{th:vecteurs}.  We set $ \tau   = s t_0$ in Theorem \ref{th:vecteurs}.  For some unknown sign $\omega \in \{-1,1\}$ we have 
\begin{equation}\label{eq:defsignom}
\xi_k = \omega \xi'_k. 
\end{equation}
\noindent\paragraph{Case 1: the sign can be estimated. }
We first assume that $\omega$ is known and $\xi'_k = \xi_k$.  We consider the function 
$$
F(v)  =  \IND_{\BRA{\sum_{e : e_2 = v} \xi'_k (e) > \tau / \sqrt n }}= \IND_{\BRA{I_k (v) > t_0}}. 
$$
From \eqref{overlap1}, \eqref{eq:balanceIpm} is satisfied with $I^\pm = J^\pm$ and we can apply Lemma \ref{le:ncov} to obtain an asymptotically positive overlap.   We note that the sign $\omega$ is easy to estimate consistently if the random variable $X$ which is the mixture of the $X_j$ with weights $\pi(j) = 1/r$ is not symmetric. Indeed, in this case, for some bounded continuous function $f$, 
$$
\dE f (X) = \sum_{j=1}^r \pi(j) \dE f(X_j) \ne \sum_{j=1}^r \pi(j) \dE f(-X_j) = \dE f (-X).  
$$ Then,  from \eqref{overlap0}, given $\omega$, in probability,
 \begin{equation*} 
\lim_{n\to\infty}\frac{1}{ n}\sum_{v = 1}^n   f( \omega I_k(v) ) = \dE f (\omega X)
\end{equation*}
takes a different value for $\omega = 1$ and $\omega = -1$.

\noindent\paragraph{Case 2: fully symmetric case. }
Another simple case is if $X$ defined above is symmetric and $|J^+| = |J^-|$. If this occurs, $X^+$ and $-X_-$ have the same distribution. We consider the function $F (v) = \IND ( \sum_{e : e_2 = v} \xi_k (e) > \tau / \sqrt n ) =  \IND (  \omega  I_k (v) > t_0 )$ and the estimation where a vertex such that $F(v) = 1$ receives a uniform label in $J^+$, and otherwise a label uniform in $J^-$. By Lemma \ref{le:ncov} applied to $F$, if $\omega = 1$, we obtain an positive overlap with $I^\pm = J^\pm$ and the permutation $p$ in \eqref{eq:defoverlap} equal to the identity. If $\omega = -1$, we obtain an positive overlap with $I^\pm = J^\mp$ and any permutation $p$ in \eqref{eq:defoverlap} such that $p(J^\pm) = J^\mp$,

\noindent\paragraph{Case 3: general case. }
In the general case, we may use the same idea : apply Lemma \ref{le:ncov} for a partition $(I^+,I^-)$ which may depend on $\omega$ but such that the cardinal of $I^\pm$ does not. First, from \eqref{overlap0}, the function $f_1(j)  = \dP ( X_j > t_0)$ is not constant on $[r]$ and there exists $j_1$ such that 
$$ f_1(j_1) > \frac{1}{r-1} \sum_{j \ne i_1} f_1(j).$$
We distinguish two subcases. The first case is when the function  $f_{-1} (j) = \dP(-X_j > t_0)$ is also non-constant. Then there exists $j_{-1}$ such that 
$$ f_{-1} (j_{-1}) > \frac{1}{r-1} \sum_{j \ne j_{-1}}  f_{-1}(j).$$
We consider the function $F  (v) =   \IND (    \sum_{e : e_2 = v} \xi_k (e) > \tau / \sqrt n ) =  \IND (  \omega  I_k (v) > t_0 )$ and the estimation where  $\hat \si(v) = 1$ if $F(v) =1$ and $\hat \si(v)$ uniform on $\{2, \ldots , r\}$ otherwise. We apply Lemma \ref{le:ncov} to $F$ and the partition $I^+ = \{ j_{\omega} \}$, $I^{-} =  [r] \backslash \{j_\omega\}$. We obtain an asymptotically positive overlap for any permutation $p$ in \eqref{eq:defoverlap} such that $p(j_\omega) = 1$.

In the other case, $f_{-1}$ is constant and equal to say $a$. We introduce extra random independent variables $\omega' (v) \in \{-1,1\}$ iid such that $\dP ( \omega'(v) = 1) = \dP ( \omega'(v) = -1) = 1/2$.  We consider the function $F(v) = \IND (  \omega'(v)  \sum_{e : e_2 = v} \xi_k (e) > \tau / \sqrt n )$. Then, by \eqref{overlap0}, in probability,
\begin{equation*}
\lim_{n\to\infty}\frac{1}{ n}\sum_{v = 1}^n \IND_{\si(v) = j} F(v) = \frac{\pi(j)}{2} ( \dP(X_j>t_0) + a).
\end{equation*} 
Hence, it follows from \eqref{overlap1} that \eqref{eq:balanceIpm} is satisfied with $I^\pm = J^\pm$. We can then apply Lemma \ref{le:ncov} to obtain an asymptotically positive overlap.

It concludes the proof of Theorem \ref{th:vecteurs}.

%% file: NBS_mars.bbl
\begin{thebibliography}{10}

\bibitem{AFH}
O.~Angel, J.~Friedman, and S.~Hoory.
\newblock The non-backtracking spectrum of the universal cover of a graph.
\newblock arXiv:0712.0192, 2007.

\bibitem{BarbourChen05}
A.~D. Barbour and L.~H.~Y. Chen, editors.
\newblock {\em An introduction to {S}tein's method}, volume~4 of {\em Lecture
  Notes Series. Institute for Mathematical Sciences. National University of
  Singapore}.
\newblock Singapore University Press, Singapore, 2005.

\bibitem{benjamini-schramm-2001}
I.~Benjamini and O.~Schramm.
\newblock Recurrence of distributional limits of finite planar graphs.
\newblock {\em Electronic J. Probab.}, 6:--, 2001.

\bibitem{MR1477662}
R.~Bhatia.
\newblock {\em Matrix analysis}, volume 169 of {\em Graduate Texts in
  Mathematics}.
\newblock Springer-Verlag, New York, 1997.

\bibitem{MR2337396}
B.~Bollob{\'a}s, S.~Janson, and O.~Riordan.
\newblock The phase transition in inhomogeneous random graphs.
\newblock {\em Random Structures Algorithms}, 31(1):3--122, 2007.

\bibitem{B_Fr}
C.~Bordenave.
\newblock A new proof of friedman's second eigenvalue theorem and its extension
  to random lifts.
\newblock arXiv:1502.04482, 2015.

\bibitem{MR965008}
F.~R.~K. Chung.
\newblock Diameters and eigenvalues.
\newblock {\em J. Amer. Math. Soc.}, 2(2):187--196, 1989.

\bibitem{Decelle11}
A.~Decelle, F.~Krzakala, C.~Moore, and L.~Zdeborov{\'a}.
\newblock Asymptotic analysis of the stochastic block model for modular
  networks and its algorithmic applications.
\newblock {\em Physics Review E}, 84:066106, 2011.

\bibitem{MR2437174}
J.~Friedman.
\newblock A proof of {A}lon's second eigenvalue conjecture and related
  problems.
\newblock {\em Mem. Amer. Math. Soc.}, 195(910):viii+100, 2008.

\bibitem{FrKo}
J.~Friedman and D.-E. Kohler.
\newblock The relativized second eigenvalue conjecture of alon.
\newblock arXiv:1403.3462, 2014.

\bibitem{MR637828}
Z.~F{\"u}redi and J.~Koml{\'o}s.
\newblock The eigenvalues of random symmetric matrices.
\newblock {\em Combinatorica}, 1(3):233--241, 1981.

\bibitem{MR1040609}
K.-i. Hashimoto.
\newblock Zeta functions of finite graphs and representations of {$p$}-adic
  groups.
\newblock In {\em Automorphic forms and geometry of arithmetic varieties},
  volume~15 of {\em Adv. Stud. Pure Math.}, pages 211--280. Academic Press,
  Boston, MA, 1989.

\bibitem{Holland83}
P.~W. Holland, K.~B. Laskey, and S.~Leinhardt.
\newblock Stochastic blockmodels: First steps.
\newblock {\em Social Networks}, 5(2):109--137, 1983.

\bibitem{MR2247919}
S.~Hoory, N.~Linial, and A.~Wigderson.
\newblock Expander graphs and their applications.
\newblock {\em Bull. Amer. Math. Soc. (N.S.)}, 43(4):439--561 (electronic),
  2006.

\bibitem{MR1091716}
R.~A. Horn and C.~R. Johnson.
\newblock {\em Topics in matrix analysis}.
\newblock Cambridge University Press, Cambridge, 1991.

\bibitem{horton-stark-terras-2006}
M.~D. Horton, H.~M. Stark, and A.~A. Terras.
\newblock What are zeta functions of graphs and what are they good for?
\newblock {\em Contemporary Mathematics, Quantum Graphs and Their Applications;
  Edited by Gregory Berkolaiko, Robert Carlson, Stephen A. Fulling, and Peter
  Kuchment}, 415:173--190, 2006.

\bibitem{MR0200979}
H.~Kesten and B.~P. Stigum.
\newblock Additional limit theorems for indecomposable multidimensional
  {G}alton-{W}atson processes.
\newblock {\em Ann. Math. Statist.}, 37:1463--1481, 1966.

\bibitem{MR0198552}
H.~Kesten and B.~P. Stigum.
\newblock A limit theorem for multidimensional {G}alton-{W}atson processes.
\newblock {\em Ann. Math. Statist.}, 37:1211--1223, 1966.

\bibitem{MR1749978}
M.~Kotani and T.~Sunada.
\newblock Zeta functions of finite graphs.
\newblock {\em J. Math. Sci. Univ. Tokyo}, 7(1):7--25, 2000.

\bibitem{spec_red}
F.~Krzakala, C.~Moore, E.~Mossel, J.~Neeman, A.~Sly, L.~Zdeborov{\'a}, and
  P.~Zhang.
\newblock Spectral redemption: clustering sparse networks.
\newblock {\em Proceedings of the National Academy of Sciences},
  (110(52)):20935--20940, 2013.

\bibitem{lubotzky95}
A.~Lubotzky.
\newblock Cayley graphs: eigenvalues, expanders and random walks.
\newblock {\em Surveys in Combinatorics, London Math. Soc. Lecture Notes},
  218:155--189, 1995.

\bibitem{lubotzky}
A.~Lubotzky, R.~Phillips, and P.~Sarnak.
\newblock Ramanujan graphs.
\newblock {\em Combinatorica}, (9):261--277, 1988.

\bibitem{LM13}
L.~Massouli\'e.
\newblock Community detection thresholds and the weak {R}amanujan property.
\newblock {\em ACM Symposium on the Theory of Computing (STOC)}, 2014, see also
  http://arxiv.org/abs/arXiv:1311.3085v1.

\bibitem{MNS}
E.~Mossel, J.~Neeman, and A.~Sly.
\newblock A proof of the block model threshold conjecture.
\newblock arXiv:1311.4115v2, 2013.

\bibitem{Mossel12}
E.~Mossel, J.~Neeman, and A.~Sly.
\newblock Reconstruction and estimation in the planted partition model.
\newblock Feb. 2012, available at: http://arxiv.org/abs/1202.1499.

\bibitem{murty}
R.~Murty.
\newblock Ramanujan graphs.
\newblock {\em J. Ramanujan Math. Soc.}, 18(1):1--20, 2003.

\bibitem{MR1124768}
A.~Nilli.
\newblock On the second eigenvalue of a graph.
\newblock {\em Discrete Math.}, 91(2):207--210, 1991.

\bibitem{MR961485}
T.~Sunada.
\newblock Fundamental groups and {L}aplacians.
\newblock In {\em Geometry and analysis on manifolds ({K}atata/{K}yoto, 1987)},
  volume 1339 of {\em Lecture Notes in Math.}, pages 248--277. Springer,
  Berlin, 1988.

\bibitem{MR2768284}
A.~Terras.
\newblock {\em Zeta functions of graphs}, volume 128 of {\em Cambridge Studies
  in Advanced Mathematics}.
\newblock Cambridge University Press, Cambridge, 2011.
\newblock A stroll through the garden.

\end{thebibliography}
